\UseRawInputEncoding
\documentclass[11pt]{article}
\usepackage[letterpaper, left=1in, top=1in, right=1in, bottom=1in, verbose, ignoremp]{geometry}

\usepackage{url}\RequirePackage[colorlinks,citecolor=blue, linkcolor=blue,urlcolor = blue]{hyperref}
\usepackage{latexsym,amssymb,amsmath,amsfonts,graphicx,color,fancyvrb,amsthm,enumitem,subcaption,mathrsfs}
\usepackage[longnamesfirst,authoryear,round]{natbib}
\usepackage[dvipsnames]{xcolor}
\usepackage{xy}\xyoption{all} \xyoption{poly} \xyoption{knot}
\usepackage{float}
\usepackage{bm}
\usepackage{bbm}
\usepackage{multicol,multirow}
\usepackage{array}
\usepackage{relsize}
\usepackage{chngcntr}
\usepackage{etoolbox}
\usepackage{caption}
\usepackage{makecell}
\usepackage{tikz}
\usepackage{mdframed}
\usetikzlibrary{decorations.pathreplacing,calc}
\usetikzlibrary{shapes,backgrounds}
\usetikzlibrary{patterns}
\usetikzlibrary{cd}
\usepackage{amsopn}
\usepackage{calc}
\usepackage{algorithm}
\usepackage[noend]{algpseudocode}
\usepackage{hyperref}
\usepackage{mathtools}

\usepackage{tikz-cd}
\usepackage{amscd}
\usepackage{comment}
\usepackage[capitalize,nameinlink,compress]{cleveref}
\crefname{figure}{Figure}{Figures} 
\crefname{equation}{}{} 
\crefname{assumption}{Assumption}{Assumptions}
\crefname{subsection}{Subsection}{Subsections}

\usepackage{graphicx} 
\usetikzlibrary{decorations.pathreplacing,backgrounds}
\usepackage{wrapfig}
\usepackage{pgf}
\usepackage{pgfplots}
\usetikzlibrary{pgfplots.external}
\usetikzlibrary{arrows}
\usetikzlibrary{arrows.meta}
\tikzcdset{arrow style=tikz,
           diagrams={>={Stealth[length=2mm]}}}
\usepackage{nicematrix}
\pgfplotsset{compat=1.18}

\usepackage[title]{appendix}

\newcounter{cdrow}

\newtheorem{theorem}{Theorem}[section] 
\newtheorem*{theorem*}{Theorem}
\newtheorem{cor}[theorem]{Corollary}
\newtheorem{lemma}[theorem]{Lemma}
\newtheorem{prop}[theorem]{Proposition}
\newtheorem{claim}{Claim}[theorem]
\newtheorem*{claim*}{Claim}
\newtheorem*{claimproof*}{Proof of claim}

\theoremstyle{definition}
\newtheorem{definition}[theorem]{Definition}
\newtheorem*{definition*}{Definition}

\theoremstyle{remark}
\newtheorem{remark}[theorem]{Remark}

\newtheorem{example}[theorem]{Example}
\newtheorem*{example*}{Example}

\DeclareMathOperator\dom{dom}

\makeatletter
\newcommand*{\op}{%
  \DOTSB
  \mathop{\vphantom{\bigoplus}\mathpalette\matt@op\relax}%
  \slimits@
}
\newcommand\matt@op[2]{%
  \vcenter{\m@th\hbox{\resizebox{\widthof{$#1\bigoplus$}}{!}{$\boxplus$}}}%
}
\makeatother

\DeclareMathOperator{\im}{im}
\newcommand{\Hh}{\mathcal{H}}
\newcommand{\Kk}{\mathcal{K}}
\newcommand{\Ff}{\mathcal{F}}
\newcommand{\Gg}{\mathcal{G}}
\newcommand{\Pp}{\mathcal{P}}
\newcommand{\Qq}{\mathcal{Q}}
\newcommand{\R}{\mathbb{R}}
\newcommand{\IK}{\mathbb{K}}
\newcommand{\IC}{\mathbb{C}}
\newcommand{\IZ}{\mathbb{Z}}
\newcommand{\N}{\mathbb{N}}
\newcommand{\red}[1]{\textcolor{red}{#1}}
\definecolor{skyblue}{RGB}{0, 180, 250}

\makeatletter
\def\@biblabel#1{}
\makeatother

\makeatletter
\patchcmd{\NAT@citex}
  {\@citea\NAT@hyper@{%
     \NAT@nmfmt{\NAT@nm}%
     \hyper@natlinkbreak{\NAT@aysep\NAT@spacechar}{\@citeb\@extra@b@citeb}%
     \NAT@date}}
  {\@citea\NAT@nmfmt{\NAT@nm}%
   \NAT@aysep\NAT@spacechar\NAT@hyper@{\NAT@date}}{}{}

\patchcmd{\NAT@citex}
  {\@citea\NAT@hyper@{%
     \NAT@nmfmt{\NAT@nm}%
     \hyper@natlinkbreak{\NAT@spacechar\NAT@@open\if*#1*\else#1\NAT@spacechar\fi}%
       {\@citeb\@extra@b@citeb}%
     \NAT@date}}
  {\@citea\NAT@nmfmt{\NAT@nm}%
   \NAT@spacechar\NAT@@open\if*#1*\else#1\NAT@spacechar\fi\NAT@hyper@{\NAT@date}}
  {}{}

\makeatother

\begin{document}

\def\spacingset#1{\renewcommand{\baselinestretch}%
{#1}\small\normalsize} \spacingset{1}

\begin{flushleft}
{\Large{\textbf{Generalized Persistent Laplacians and their Spectral Properties}}}
\newline
\\
Arne Wolf$^{1,2,\dagger}$, Jiyu Fan$^{1}$, and  Anthea Monod$^{1}$
\\
\bigskip
\bf{1} Department of Mathematics, Imperial College London, UK
\\
\bf{2} London School of Geometry and Number Theory, UK
\\
\bigskip
$\dagger$ Corresponding e-mail: a.wolf22@imperial.ac.uk
\end{flushleft}


\section*{Abstract}

Laplacian operators are classical objects that are fundamental in both pure and applied mathematics and are becoming increasingly prominent in modern computational and data science fields such as applied and computational topology and application areas such as machine learning and network science.  In this paper, we introduce a unifying operator-theoretic framework of generalized Laplacians as invariants that encompasses and extends all existing constructions, from discrete combinatorial settings to de Rham complexes of smooth manifolds.  Within this framework, we introduce and study a generalized notion of persistent Laplacians. While the classical persistent Laplacian fails to satisfy the desirable properties of monotonicity and stability---both crucial for robustness and interpretability---our framework allows us to isolate and analyze these properties systematically.  We demonstrate that their component maps, the up- and down-persistent Laplacians, satisfy these properties individually. Moreover, we prove that the spectra of these separate components fully determine the spectra of the full Laplacians, making them not only preferable but sufficient for analysis.  We study these questions comprehensively, in both the finite and infinite dimensional settings. Our work expands and strengthens the theoretical foundation of generalized Laplacian-based methods in pure, applied, and computational mathematics.

\paragraph{Keywords:} Laplacians; Monotonicity; Operator Theory; Persistent Homology; Stability\\

\vfill\eject

\tableofcontents


\section{Introduction}
\label{sec:intro}

Laplacian operators are central objects in mathematics, arising across a wide range of areas from graph theory \citep{graph_1, graph_2, graph_3,josthorak}, to Hodge theory of manifolds \citep{hodge_1, hodge_2, hodge_3,schwarz}, and, more recently, to computational fields including applied and computational topology \citep{tda_1, tda_2, tda_3,dws}. In computational topology, algebraic tools are used to study the intrinsic properties of underlying spaces or structures (such as manifolds, graphs, simplicial complexes, and sheaves) through their topology. Homology is the most widely used tool for such a study. Laplacians, however, offer a stronger algebraic approach, as their kernels recover homology groups \citep[see, e.g.,][p.~7]{survey}, while their nonzero spectra encode additional geometric information \citep{graphlaplacian,beers_mulas,curvature}. This dual role positions the eigenvalues of Laplacians as a natural, more powerful, and meaningful invariant that is simultaneously algebraic and geometric.

The key concept that we generalize and then study in this work is the \emph{persistent Laplacian} map that was initially proposed by \citet{first_peristent_lap2} and \citet{first_persistent_lap}, and further studied by \citet{mww}. This operator generalizes the classical Laplacian by incorporating \emph{filtrations}, thereby capturing how topological and geometric features evolve across scales. For instance, given a point cloud, a common approach in computational topology is to consider a nested sequence of simplicial complexes and track the evolution of homology under the chain of inclusions \citep[e.g.,][\S4.1]{roadmap}. This ``dynamic'' version of homology yields \emph{persistent homology}; it is arguably the most prominent theory in the area of computational topology, as well as the most popular and widely-applied methodology implemented in topological data analysis (TDA). The persistent Laplacian extends this theory: its kernel recovers persistent homology, while its nonzero spectra contain additional geometric information, just as in the classical case. Thus, persistent Laplacians comprise strictly more information than persistent homology alone and are proper extensions of persistent homology.

In this work, we study two key geometric and faithfulness properties of our generalized persistent Laplacian: stability and monotonicity.  \emph{Stability} is a crucial property of topological invariants, especially in applications, as it ensures robustness: small perturbations of the input (manifold, graph, simplicial complex, etc.) result in only small perturbations of the eigenvalues, with respect to a suitable metric. Stability has been widely studied in computational topology and TDA \citep[e.g.,][]{cohen-steiner_Stability_2017,stability3,stability1,stability4} and has also been studied for persistent Laplacians constructed from simplicial complexes \citep{mww}. \emph{Monotonicity} is another related and desirable property of the persistent Laplacian as an invariant, which has been considered in the case of simplicial complexes \citep{mww}.  Monotonicity expresses a natural geometric relation: as the underlying object grows under a filtration, its Laplacian spectra should grow accordingly.

It turns out that persistent Laplacians in general do not have these properties of monotonicity and stability. However, each persistent Laplacian is a sum of two maps: the \emph{up}- and the \emph{down}-persistent Laplacian. Monotonicity and stability have previously been studied and established for the up-persistent Laplacian constructed from a simplicial complex in the finite-dimensional setting \citep{mww}. Moreover, the spectra of the up- and down-Laplacians in this setting completely determine the spectra of the full Laplacian \citep{mww}. In this paper, we extend all of these results to persistent Laplacians arising from general chain complexes, both in the finite- and infinite-dimensional settings. 

\paragraph{Contributions.} We introduce a generalized persistent Laplacian constructed from general chain complexes.  Existing occurrences of persistent Laplacians \citep[see, e.g.,][]{survey} are constructed from finite-dimensional spaces where operators are bounded.  We extend this setting to consider operators that are densely defined and closed, which provides a unifying framework to study persistent Laplacians arising from all underlying objects of interest arising in TDA.  Moreover, our framework additionally comprises de Rham complexes of smooth manifolds.

We establish montonicity and stability for our generalized persistent Laplacians, building on and extending previous results in a threefold manner: 
\begin{enumerate}
\item[(i)] Where previous studies only considered the up-persistent Laplacian map, we study both up and down components.  We prove that the spectra of these two maps fully determine the important spectra of the full generalized persistent Laplacians; therefore, not only is it preferable to work with these maps for simplicity, it is moreover sufficient to only consider these two maps.
\item[(ii)] We establish these properties for a general class of underlying objects comprising and extending beyond existing cases studied in TDA to include de Rham complexes of smooth manifolds.
\item[(iii)] We lift the theory from the finite-dimensional setting to the infinite-dimensional case by developing a novel operator-theoretic framework. This leads to a more general notion of spectrum and requires entirely new proofs, while also subsuming the finite-dimensional results as a special case.
\end{enumerate}

Finally, we provide a sufficient condition for when full Laplacians satisfy monotonicity and stability.

Our theoretical contributions position persistent Laplacians as natural topological invariants that encompass and extend beyond persistent homology. Our results lay the groundwork for a potential paradigm shift toward using persistent Laplacians as fundamental topological and geometric invariants.

\paragraph{Outline.} The remainder of this paper is organized as follows: We close this section with an overview of related work. In \Cref{sec:motivation} we present the Laplacians and persistent Laplacians in the case of finite simplicial complexes to motivate our study and context for the relationship between (persistent) Laplacians and (persistent) homology, providing intuition throughout the paper. In \Cref{sec:background}, we provide the main tools from operator theory needed for our generalized study of persistent Laplacians, including for infinite-dimensional spaces. We reprove some of these results for operators mapping between pairs of spaces, which yields new results in the area of operator theory supporting our study. We then define the Laplacians in full detail and generality in \Cref{sec:laplacians} as well as their persistent versions, and provide examples in a variety of general settings. \Cref{sec:stab} presents our main results: we define and establish monotonicity for persistent up- and down-Laplacians and then show how the spectra of up-, down-, and full-Laplacians are related. Finally, in \cref{subsec:fullmon}, we discuss the connection between stability and monotonicity as well as provide a sufficient condition for monotonicity of full Laplacians. We close with a summary and discussion of the impact of our work in \Cref{sec:discussion} and suggest directions for future research.

\Cref{appendix:operator_theory} collects supplementary material on functional analysis relevant to the work we develop on operator theory. These are included for completeness and to support the main exposition.  \Cref{appendix:category} provides additional background on the homology of Hilbert complexes from a categorical perspective, to highlight why reduced homology is a particularly natural notion for our setting.

\subsection{Related Work} \label{subsection:relwork}

The persistent Laplacian was first introduced by \citet{first_peristent_lap2} and \citet{first_persistent_lap}, extending the classical combinatorial Laplacian to study pairs of simplicial complexes $K \hookrightarrow L$ connected by an inclusion. \citet{mww} define stability and monotonicity of persistent Laplacians as we will do in this work and they prove stability of only the up-Laplacian. \citet{pl_simpmaps} revisit the findings of \citet{mww} in the more general setting that replaces inclusions in the definition of persistent Laplacians by so-called weight-preserving simplicial maps, where they consider inner products on the simplicial chain complex of simplicial complexes induced by weight functions and simplicial maps that preserve weights. We note that this approach is more general than the inclusions we will later define in our work (specifically, \cref{def:inclusion}). However, both of these results by \citet{mww} and \citet{pl_simpmaps} work with concrete basis representations and therefore with finite-dimensional vector spaces, which we extend to infinite-dimensional settings in our work. Moreover, they restrict to chain complexes arising from simplicial complexes, while in our work, we consider general chain complexes which encompass and expand beyond those obtained from simplicial settings. \citet{algstab} prove an algebraic version of stability for persistent Laplacians considered in the setting of \citet{mww}. \citet{hou_harmonic_chains} define harmonic chain barcodes, which provide a basis of the harmonic chain groups at every stage of a filtration, and prove their bottleneck stability. Very recently, \citet{su2024persistentrhamhodgelaplacianseulerian} define persistent Laplacians and their discrete versions for filtrations of manifolds. Their framework, however, is restricted to smooth differential forms, and thus differs from the definitions we present in \cref{subsub:derham}.

In terms of applications, persistent Laplacians have been implemented in numerous situations within the context of TDA: \citet{plapp_1,plapp_3} use them for data classification and \citet{pl_omicron} predicted dominating COVID-19 variants using persistent Laplacian-based deep learning on molecular data. \citet{wbx} defined persistent Dirac operators which combine all persistent Laplacians of a pair of simplicial complexes into one operator and use them for molecular structure classification. The success of all these methods, which mostly outperform persistent homology, suggests that the (nonzero) spectra of persistent Laplacians indeed contain relevant information about the underlying spaces which contribute to better performance in downstream data analysis and machine learning tasks.

Efficient algorithms to compute spectra of persistent Laplacians have been proposed by \citet{mww}.
\citet{htp_cont} suggest homotopy continuation to find the eigenvalues of up-Laplacians, while \citet{fastalg} improves the algorithm from \citet{mww} for a specific type of simplicial complex to compute a matrix representation of the up-Laplacians in quadratic time. The current fastest software to compute and use persistent Laplacians in practical settings was developed by \citet{petls}.


\section{Background and Preliminaries} 
\label{sec:motivation}

In this section, we present the objects and setting of our study: Laplacians and the theory of homology arising from algebraic topology.  In particular, we study \emph{persistent homology}, which can be thought of as a dynamic version of homology by studying its evolution with respect to a filtration. We also overview the well-known Hodge theorems, connecting the kernel of the Laplacians to homology.  

\subsection{Simplicial Homology and The Combinatorial Laplacian} 
\label{sec:tda}

A finite point cloud, $S = (x_1, \dots, x_m)$ with $x_i \in \mathbb{R}^n$, may be studied by constructing a topological space $X$ from $S$ and applying algebraic tools to $X$. Simplicial complexes are a natural choice of topological space to build from $S$, since they are both computationally amenable and general enough to approximate any topological space.  Recall that every topological space $X$ can be approximated by a simplicial complex $K$ that is weakly homotopy equivalent to $X$ \citep[Proposition 4.13, Theorem 2C.5]{hatcher}. 

\begin{definition}\textnormal{\citep[Combinatorial Simplicial Complex,][p.~107]{hatcher}}
\label{def:comb_simp_cplx}
    Let $V$ be a (possibly infinite) set of vertices or points. A \emph{simplicial complex} $K$ is a set of non-empty finite subsets of $V$ with the property that for all $\sigma \in K$ and $\tau \subset \sigma$, $\tau \in K$.
    
    If $\sigma \in K$ consists of $k+1$ points $p_0,\dots,p_k$, it is called a $k$\emph{-simplex} and written as $\sigma^k$ or as $[p_0,\dots,p_k]$; $\tau \subset \sigma$ is a \emph{face} of $\sigma$. The largest $k$ such that there is a $k$-simplex in $K$ is called the \emph{dimension} $\dim(K) \in \mathbb{N} \cup \{\infty\}$ of $K$.
\end{definition}

Discussions in this section will be restricted to \emph{finite} simplicial complexes, i.e., with $|V|< \infty$. 

In this paper, we use homology to extract an algebraic summary of a simplicial complex. The homology of a finite simplicial complex can be computed from an associated \emph{chain complex}:
\[
\begin{tikzcd}
\cdots \arrow[r] & C^{K}_{k+1}  \arrow[r, "d^{K}_{k+1}"] & C^{K}_{k} \arrow[r, "d^{K}_{k}"] & C^{K}_{k-1} \arrow[r] & \cdots
\end{tikzcd}
\]
which comprises of chain groups 
$$
C_k^K:=\bigoplus_{\sigma^k \in K} \R
$$ 
and boundary maps $d_k^K$. To define the latter, we need a function that keeps track of orientation between subfaces. Assuming that $K$ is finite, we can obtain such a function as follows, by enumerating its vertices. Let $\sigma^k=\{p_0,\dots,p_k\}$ be a face of $K$, where the $p_i$ are ordered according to their numbers. The $(k-1)$-dimensional subfaces of $\sigma$ are precisely those of the form 
    $$
    \tau^{k-1}=\{p_0,\dots,p_{\ell-1},p_{\ell+1},\dots,p_k\},
    $$
    i.e., where one point (in this case, the $\ell$th point) is removed from $\sigma$.
    For this $\tau$, we set 
    \begin{equation} \label{eq:sign}
    [\sigma:\tau]:=(-1)^\ell.
    \end{equation}
By imposing that the boundary maps (homomorphisms) are linear, they are then uniquely defined from their action on the basis elements and are given by 
$$
    d_k^K: e_{\sigma^k} \mapsto \sum_{\tau^{k-1}\subset \sigma} [\sigma:\tau] e_\tau.
$$
The boundary maps satisfy the defining property of chain complexes, namely that $d^K_{k-1} \circ d^K_k = 0$ \citep[Lemma 2.1]{hatcher}. Thus we can define the \emph{$k$th homology group} of $K$ by
\[
H_k(K) := \ker(d^K_k)/\im(d^K_{k+1}).
\]

Intuitively, the dimensions of these groups, the \emph{$k$th Betti numbers}, count the number of $k$-dimensional holes of $X$ discretized by $K$, which can be used as topological features themselves. However, often we are interested in not only the Betti numbers, but also in more concrete representatives of homology elements which are more informative than just the Betti numbers alone. There are various possible choices for such representatives. For example, \citet{tightestrep} chooses representatives that ``wrap around'' the holes in the tightest way possible. Another option to obtain more informative representatives is via the \emph{Hodge Laplacians}, which is the object of central interest in our work.

By equipping the chain groups with inner products, we can define adjoints of the boundary homomorphisms $d^K_k$, denoted by $\big(d^K_k \big)^*$. These inner products can carry geometric information, such as weights on vertices, edges, or even other faces \citep[p.~5]{mww}.

Given these adjoints, we can now define the Hodge Laplacian, from which representatives of homology elements can be recovered.

\begin{definition}\textnormal{\citep[Hodge Laplacian,][Definition 2.1]{josthorak}}
    The \emph{$k$th Hodge Laplacian} on $K$ is the map $\Delta^K_k: C^K_k \to C^K_k$ given by
    \[
    \Delta^K_k := d^K_{k+1} \circ \left(d^K_{k+1}\right)^* + \left(d^K_{k}\right)^* \circ d^K_{k}.
    \]
\end{definition}

The kernels of the Hodge Laplacians contain the homology information of a simplicial complex and, as such, they can be used to extract algebraic summaries of simplicial complexes. Moreover, they can provide homological representatives of simplicial complexes and their underlying point clouds. These representatives can be effectively used for data analysis, such as spectral clustering \citep{spec_clust}.

\begin{theorem}\textnormal{\citep[Combinatorial Hodge Theorem,][Theorem 5.3]{lim}} \label{thm:comb_ker}
    $\ker(\Delta^K_k) \cong H_k(K)$ and the isomorphism is given by $v \mapsto [v]$.
\end{theorem}

\subsection{The Persistent Combinatorial Laplacian} \label{sec:persmot}

\Cref{def:comb_simp_cplx} gives the notion of a general simplicial complex, with specific instances arising from the choice of construction rule. Two examples of specific simplicial complexes widely studied, especially in computational topology, for their desirable combinatorial and computational properties are the \emph{\v{C}ech} and \emph{Vietoris--Rips} complexes.  Their construction approach is similar: Proceeding from a point cloud, a threshold $\varepsilon$ is fixed and $\varepsilon$-radius balls are formed around each point in the point cloud, $S$. Simplices between points are then formed and included in the complex if points are ``close enough'' according to a construction rule depending on the overlap of the $\varepsilon$-balls; see \cref{fig:complexes} for example illustrations of \v{C}ech and Vietoris--Rips complexes.  There are many other types of simplicial complexes in addition to the \v{C}ech and Vietoris--Rips complexes; see \citet[pp.~27-33]{sctypes} for a list of further simplicial complex types.

\begin{figure}[h]
    \centering
    \includegraphics{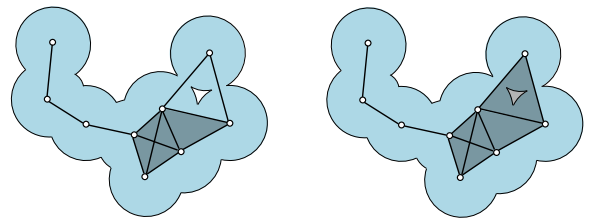}
    \caption{Examples of a \v{C}ech complex (left) and a Vietoris--Rips complex (right) for the same $\varepsilon$ \citep[p.~28]{sctypes}.}
    \label{fig:complexes}
\end{figure}

A key question in constructing simplicial complexes following the approach of \v{C}ech and Vietoris--Rips complexes is that of choosing the precise value of $\varepsilon$, which dictates the entire structure of the resulting complex and, as a result, its homology and its homology representatives as well.  Rather than fixing a particular value of $\varepsilon$, the idea of \emph{persistent homology} (or simply, \emph{persistence}) is to consider all values of $\varepsilon \in [0,\infty)$ and how homology elements are created, evolve, and are destroyed as $\varepsilon$ ranges from 0 to $\infty$ \citep{phbarcode}.  Varying $\varepsilon$ creates a sequence of nested simplicial complexes, called a \emph{filtration}, $\{K_t\}_{t \in \mathbb{R}}$. Moreover, it creates a filtration of associated chain complexes, i.e., we have inclusion maps $\iota_k$ between the $k$th chain groups for each $k$ that commute with the boundary maps.

By studying the evolution of homology elements with $\varepsilon$, persistent homology tracks and measures the relative importance of topological features by considering and collating their creation and destruction, termed as \emph{birth} and \emph{death}, respectively. It turns out that pairs $K_s \subset K_t$ of simplicial complexes are sufficient to encode all the persistent homology information by considering the pushforward of the inclusion maps on the homology level \citep{PHfrompairs}. 

\begin{definition} \citep[Persistent Homology,][]{PHfrompairs} \label{defn:pershommot}
    Let $K \subset L$ be two simplicial complexes. The \textit{persistent homology groups} of this pair are ${H_k}(K,L) := \im(\iota_{*} : {H_k}(K) \to {H_k}(L))$ for every $k$.
\end{definition}

\Cref{thm:comb_ker} for a fixed simplicial complex relates the kernel of the combinatorial Hodge Laplacian to its homology groups. It turns out that a persistent version of this relation also exists for evolving simplicial complexes and a corresponding Laplacian for persistence \citep{mww}. Such a suitable Laplacian and its corresponding harmonic representatives, which can be used as persistent homology representatives, can be obtained by equipping the chain groups $C_k^K$ and $C_k^L$ with inner products such that the inclusions $\iota_k:C_k^K \to C_k^L$ are isometries, which can be done by fixing inner products on $C_k^L$ and pulling them back to $C_k^K$ via $\iota_k$. Adjoints can thus be defined and we obtain the diagram in \cref{fig:perslapsc}.

\begin{figure}
\begin{center}
\begin{tikzpicture}[x=1cm,y=1cm]
        \clip(0,-0.8) rectangle (9.5,5);
        \draw [-{Stealth[length=2mm]}] (1.6,4) --  node[above=1pt] {$d_{k+1}^{K}$} (4.5,4);
        \draw [color=black,{Stealth[length=2mm]}-] (5.5,3.9) --  node[below=1pt] {$\left(d_{k}^{K}\right)^*$} (8.4,3.9);
        \draw [color=black,-{Stealth[length=2mm]}] (5.5,4.1) --  node[above=1pt] {$d_{k}^{K}$} (8.4,4.1);
        \draw [-{Stealth[length=2mm]}] (1.6,0) --  node[below=1pt] {$d_{k+1}^{L}$} (4.5,0);
        \draw [-{Stealth[length=2mm]}] (5.4,0) --  node[below=1pt] {$d_{k}^{L}$} (8.4,0);
        \draw [color=black,-{Stealth[length=2mm]}] (2.9,2) --  node[anchor=south east] {$d_{k+1}^{K,L}$} (4.7,3.8);
        \draw [color=black,{Stealth[length=2mm]}-] (3,1.8) -- (4.8,3.6);
        \draw [right hook-{Stealth[length=2mm]}] (1,3.6) -- node[left]{$\iota_{k+1}$} (1,0.4);
        \draw [right hook-{Stealth[length=2mm]}] (5,3.6) -- node[right]{$\iota_{k}$}(5,0.4);
        \draw [right hook-{Stealth[length=2mm]}] (9,3.6) -- node[left]{$\iota_{k-1}$}(9,0.4);
        \draw [right hook-{Stealth[length=2mm]}] (2.2,1.2) -- (1.2,0.3);
        \draw[color=black] (4.2,2) node {$\left(d_{k+1}^{K,L}\right)^*$};
        \draw[color=black] (1,4) node {$C_{k+1}^{K}$};
        \draw[color=black] (5,4) node {$C_k^{K}$};
        \draw[color=black] (9,4) node {$C_{k-1}^{K}$};
        \draw[color=black] (9,0) node {$C_{k-1}^{L}$};
        \draw[color=black] (5,0) node {$C_k^{L}$};
        \draw[color=black] (1,0) node {$C_{k+1}^{L}$};
        \draw[color=black] (2.5,1.5) node {$C_{k+1}^{K,L}$};
    \end{tikzpicture}
\caption{Setting in which the persistent Laplacian is defined.} \label{fig:perslapsc}
\end{center}
\end{figure}
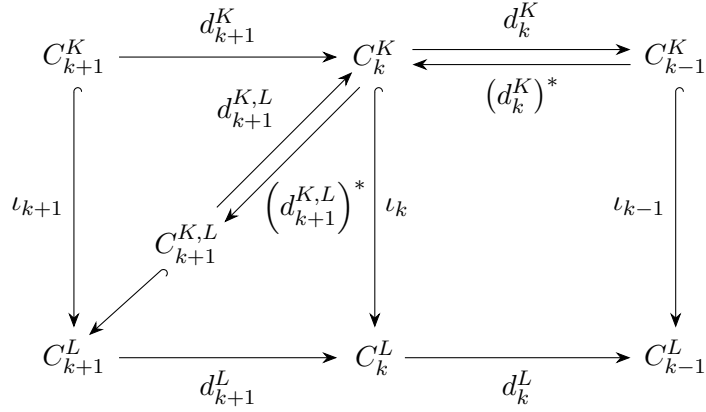

Additionally, two auxiliary objects are further required: 
\begin{enumerate}[noitemsep]
    \item[(i)] the space $C^{K,L}_{k+1} := \{c \in C^{L}_{k+1} \lvert d^{L}_{k+1}(c) \in C^{K}_k\}$ equipped with the restriction of the inner product of $C^{L}_{k+1}$; and
    \item [(ii)] the map $d^{K,L}_{k+1}: C^{K,L}_{k+1} \rightarrow C_k^K$ defined by $d^{K,L}_{k+1} = d^{L}_{k+1} \lvert_{C^{K,L}_{k+1}}.$
\end{enumerate}

These components together give rise to the \emph{persistent Laplacian}.
\begin{definition} \citep[Persistent Laplacian,][\S 2.2]{mww} \label{def:comb_PL}
    The \textit{$k$th persistent Laplacian} of the pair $K,L$ is the linear map $\Delta^{K,L}_k : C^{K}_k \to C^{K}_k$ given by 
    \[
    \Delta^{K,L}_k := d^{K,L}_{k+1} \circ (d^{K,L}_{k+1})^{*} +(d^{K}_k)^{*} \circ d^{K}_k.
    \]
\end{definition}

\begin{example} \label{ex:algperslap}  
    Consider the simplicial complexes $K$ (in blue) and $L$ (everything) as shown in \cref{fig:algexperslap}. We will compute the first persistent Laplacian (i.e., $k=1$) by following the approach of \citet{mww}. We choose standard inner products everywhere and the purple bases of the chain groups.
    
    \begin{figure}[H]
        \centering
        \begin{tikzpicture}[x=0.7cm,y=0.7cm]
    \clip (-4.6,-3.8) rectangle (4.6,3.2);
    \node[circle, fill=blue, minimum size=5pt, inner sep=0pt, label=above right:$c$] (c) at (2,0) {};
    \node[circle, fill=blue, minimum size=5pt, inner sep=0pt, label=above left:$a$] (a) at (-2,0) {};
    \node[circle, fill=blue, minimum size=5pt, inner sep=0pt, label=right:$d$] (d) at (0,3) {};            
    \node[circle, fill=blue, minimum size=5pt, inner sep=0pt, label=below:$b$] (b) at (0,-3) {};            
    \draw[blue,line width=2pt] (a) -- (b) -- (c) -- (d) --(a);
    \draw[black,line width=2pt] (a) -- (c);
    \begin{scope}[on background layer]
    \fill[black!20] (a.center) -- (b.center) -- (c.center) -- cycle;
    \fill[black!20] (a.center) -- (c.center) -- (d.center) -- cycle;
    \end{scope}
\end{tikzpicture}
\hspace{1cm}
\begin{tikzpicture}[x=1cm,y=1cm]
	\clip(0,-0.8) rectangle (9.4,6);
    \draw [-{Stealth[length=2mm]}] (1.6,4) --  node[above] {$0$} (4.5,4);
    \draw [-{Stealth[length=2mm]}] (5.5,4) --  node[above] {\scriptsize \setlength{\arraycolsep}{1pt} $\begin{pNiceMatrix}
        -1&0&-1&0\\
        1&-1&0&0\\
        0&1&0&-1\\
        0&0&1&1
    \end{pNiceMatrix}$} (8.4,4);
    \draw [-{Stealth[length=2mm]}] (1.6,0) --  node[above] {\scriptsize $\left(\begin{matrix}
        1&0\\
        1&0\\
        0&-1\\
        0&1\\
        -1&1\\
    \end{matrix}\right)$} (4.5,0);
    \draw [right hook-{Stealth[length=2mm]}] (1,3.6) -- node[left]{$0$} (1,0.4);
    \draw [right hook-{Stealth[length=2mm]}] (5,3.6) -- node[right]{\scriptsize $\left(\begin{matrix}
        1&0&0&0\\ 0&1&0&0\\ 0&0&1&0 \\ 0&0&0&1 \\0&0&0&0
    \end{matrix}\right)$}(5,0.4);
    \draw[color=black] (1,4) node {$\{0\}$};
    \draw[color=black] (5,4) node {$\mathbb{R}^4$};
    \draw[color=black] (9,4) node {$\mathbb{R}^4$};
    \draw[color=black] (5,0) node {$\mathbb{R}^5$};
    \draw[color=black] (1,0) node {$\mathbb{R}^2$};
    \begin{scriptsize}
    \node[color=purple,text width=0.7cm,align=center] at (0.3,0) {abc acd};
    \node[color=purple,text width=1cm,align=center] at (6,0) {ab, bc, ad, cd, ac};
    \node[color=purple,align=center,text width=1.8cm] at (5,4.6) {ab, bc, ad, cd};
    \node[color=purple,align=center] at (9,4.5) {a, b, c, d};   
    \end{scriptsize}
\end{tikzpicture}
        \caption{Above: The simplicial complexes studied in \cref{ex:algperslap}, $\textcolor{blue}{K} \subset L$. Below: The $k\in \{0,1,2\}$ component of the diagram from \cref{fig:perslapsc} for the chosen purple bases.}
        \label{fig:algexperslap}
    \end{figure}
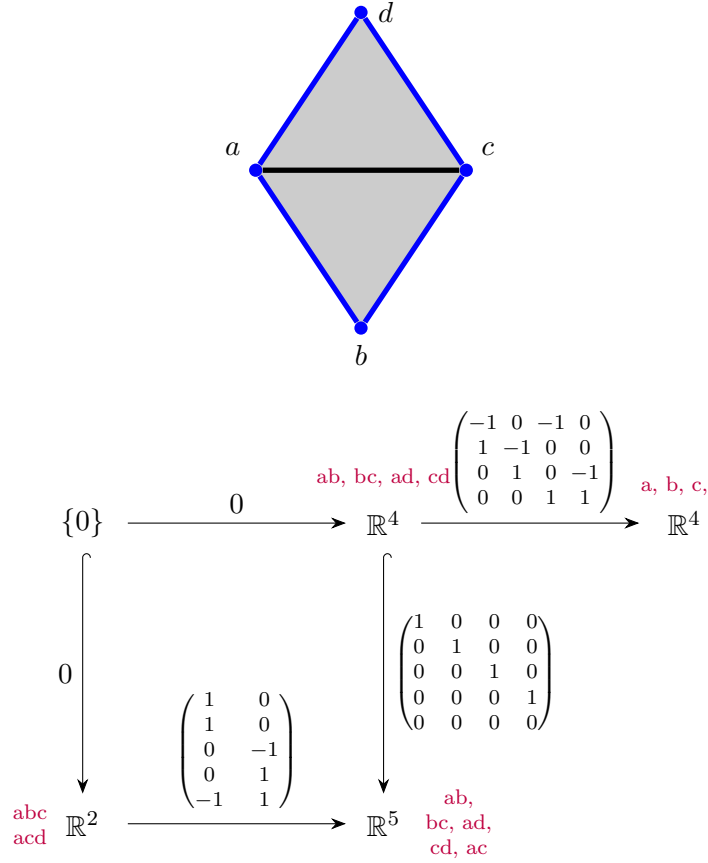
    We find 
    $$
        C_2^{K,L} = \text{span}\left\{ \! \begin{pNiceMatrix}
            1\\1
        \end{pNiceMatrix}  \! \right\} \subset \R^2
    $$
    The orthonormal basis for $C_2^{K,L}$ is given by $\frac 1{\sqrt{2}} \begin{pNiceMatrix}
            1\\1
        \end{pNiceMatrix}$ and with respect to this basis, $d_2^{K,L}$ is the matrix
    $$\frac 1 {\sqrt{2}}\begin{pNiceMatrix}
            1 \\1 \\ -1 \\ 1
        \end{pNiceMatrix} \raisebox{-5ex}{.}$$
    Consequently, the matrix representation of $\Delta_1^{K,L}$ for the given basis is
    $$
        \frac{1}{2} \begin{pNiceMatrix}
            1 \\1 \\ -1 \\ 1
        \end{pNiceMatrix}
        \begin{pNiceMatrix}
            1 &1 & -1 & 1
        \end{pNiceMatrix} + \begin{pNiceMatrix}
        -1&1&0&0\\
        0&-1&1&0\\
        -1&0&0&1\\
        0&0&-1&1
    \end{pNiceMatrix}  \begin{pNiceMatrix}
        -1&0&-1&0\\
        1&-1&0&0\\
        0&1&0&-1\\
        0&0&1&1
    \end{pNiceMatrix} = \frac{1}{2} \begin{pNiceMatrix}
        5&-1&1&1\\
        -1&5&-1&-1\\
        1&-1&5&1\\
        1&-1&1&5
    \end{pNiceMatrix}\raisebox{-5ex}{,}
    $$
    which has eigenvalues $\{4,4,4,8\}$.
\end{example}

Persistent Laplacians give rise to a persistent combinatorial Hodge theorem, a result that we aim to generalize as one of our main contributions in \cref{sec:laplacians}.

\begin{theorem}[Persistent Combinatorial Hodge Theorem, \citet{mww}]
     We have \\$\ker(\Delta_k^{K,L}) \cong H_k(K,L)$ with the isomorphism given by $v \mapsto [\iota_k (v)]$.
\end{theorem} 

\section{Operator Theory: Main Tools and Supporting Results}
\label{sec:background}

To achieve a level of generality that comprises both chain complexes of Hilbert spaces as in \cref{sec:motivation}, as well as de Rham complexes on manifolds, we will formulate our concepts in terms of \emph{densely defined operators between Hilbert spaces}. In this section, we outline some useful notions from operator theory in order to work with such operators. Our discussion is intended for readers with knowledge of the main concepts of functional analysis, but does not require prior experience with operator theory.

Throughout this section, we consider Hilbert spaces $\Hh$ and $\Kk$ over a field $\IK \in \{\R,\IC\}$ and a linear operator $T: \Hh \supset \dom(T) \rightarrow \Kk$ such that the domain $\dom(T)$ is a dense subset of $\Hh$. We call such operators \textit{densely defined}. In operator theory, we usually work with densely defined operators on Hilbert spaces rather than everywhere defined operators on their domains. The main reason for this is that otherwise we could not define adjoint or closed operators. As we shall see, these two concepts are essential for our work.

\subsection{Adjoint and Closed Operators}

Standard results from operator theory are typically formulated for operators acting on one Hilbert space, that is, for the case that $\Hh=\Kk$ (see \citet{schroedingerop}). For our studies, we require operators acting between Hilbert spaces, thus, we will re-prove some results from operator theory for operators acting between two different Hilbert spaces. We restrict to densely defined operators which enable us to define an adjoint. The following of this section consists of the more detailed and generalized versions of the definitions and results given by \citet[p.~12]{schroedingerop}.

\begin{definition}[Adjoint Operator, \citet{schroedingerop}]
    The \emph{adjoint} $T^*$ of a densely defined operator $T$ is defined on
$$
    \dom T^*=\{ v\in \mathcal{K}:\text{there is } g \in \mathcal{H} \text{ such that for all } u\in \dom T \text{ we have }\langle g,u\rangle_\mathcal{H}=\langle v,Tu\rangle_\mathcal{K}\}.
$$
For $v\in \dom T^*$, we set $T^*v:=g$.
\end{definition}

Notice that since $T$ is densely defined, such a $g$ is unique and therefore $T^*$ is well-defined: Consider some $v\in \dom T^*$, and some $g_1,g_2\in \mathcal{H}$ with $\langle g_1,u\rangle_\mathcal{H}=\langle g_2,u\rangle_\mathcal{H}=\langle v,Tu\rangle_\mathcal{K}$ for all $u \in \dom(T)$. Then $T$ being densely defined implies that there exists $(u_n)\subset \dom T$ such that $u_n\to g_1-g_2$. Thus we have
\begin{align*}
    \langle g_1-g_2,g_1-g_2\rangle_\mathcal{H}&=\langle g_1,g_1-g_2\rangle_\mathcal{H}-\langle g_2,g_1-g_2\rangle_\mathcal{H}=\lim_{n\to\infty}(\langle g_1,u_n\rangle_\mathcal{H}-\langle g_2,u_n\rangle_\mathcal{H})\\
    &=\lim_{n\to\infty}(\langle v,Tu_n\rangle_\mathcal{K}-\langle v,Tu_n\rangle_\mathcal{K})=0.
\end{align*}
Hence $g_1=g_2$.

Next, we introduce the property of closedness of operators.

\begin{prop}[Closed Operators, \citet{schroedingerop}] \label{def:closedop}
    The following are equivalent:
    \begin{enumerate}
        \item[(a)] The graph of $T$, $\Gg(T):=\{(u,T(u)): u \in \dom(T)\} \subset \Hh \times \Kk$ is a closed subset of $\Hh \times \Kk$.
        \item[(b)] $\dom(T)$ is complete in $\Hh$ with respect to the norm $\displaystyle ||u||_g:= \sqrt{||T(u)||^2_\Kk + ||u||^2_\Hh}.$
        \item[(c)] $\dom(T)$ is complete in $\Hh$ with respect to the weighted norm $\displaystyle ||u||_{g,m}:= \sqrt{||T(u)||^2_\Kk + m||u||^2_\Hh}$ for all $m \in \R_{> 0}$.
    \end{enumerate}
    If these properties are satisfied, $T$ is a \emph{closed} operator.
\end{prop}
Closedness of operators is essential for defining their spectrum; it may be viewed as a weaker analogue of continuity since the usual notion of continuity is not applicable for operators that are not defined everywhere. Closedness serves as a substitute, ensuring that closed operators ``behave well'' with respect to limits.

\begin{proof}
    (a)$\implies$(c): Assume that the graph of $T$ is closed and there are $u_n \in \dom T$ for $n \in \N$ such that $(u_n)$ is a Cauchy sequence with respect to the weighted graph norm. We aim to show that for some $u\in \dom T$, as $n \to 0$, 
    $$
    \|u_{n}-u\|_{g,m} = (\|Tu_{n}-Tu\|^2+m\|u_{n}-u\|^2)^{1/2} \to 0.
    $$
    Since $(u_n)$ is Cauchy, for any $\varepsilon>0$, there exists $N\in\mathbb{N}$ such that
    $$
    \|u_n-u_\ell\|_{g,m}=(\|Tu_n-Tu_\ell\|^2+m\|u_n-u_\ell\|^2)^{1/2}<\sqrt{m}\, \varepsilon
    $$
    for all $n,\ell>N$. Given $\|Tu_n-Tu_\ell\|^2\geq 0$, it implies that
    $$
    m\|u_n-u_\ell\|^2<m\varepsilon^2 \Longrightarrow \|u_n-u_\ell\|<\varepsilon,
    $$
    for all $n,\ell>N$. Thus, $(u_n)$ is Cauchy with respect to the norm on $\Hh$. Since $\mathcal{H}$ is complete, $(u_{n})$ converges to some $u\in\mathcal{H}$. Similarly, $(Tu_n)$ is Cauchy with respect to the usual norm, so $Tu_{n}\to y$ for some $y\in\mathcal{K}$. Since $T$ has a closed graph, $u\in\dom T$ and $y=Tu$, which completes our proof.\\

    \noindent (c)$\implies$(b): (b) is a special case of (c) by choosing $m=1$.\\

    \noindent (b)$\implies$(a): Now assume $\dom T$ is complete with respect to the graph norm. Consider a sequence
    $((u_n,Tu_n))_{n\in \mathbb{N}}$ in $\dom T \times \mathcal{K}$ that converges to some $(u,y)$ in $\mathcal{H}\times \mathcal{K}$. Since $u_n \to u$, for any $\varepsilon>0$, there exists $N_1\in\mathbb{N}$ such that $\|u_n-u_\ell\|<\frac{\varepsilon}{\sqrt{2}}$ for all $n,\ell>N_1$. Similarly, there exists $N_2\in\mathbb{N}$ such that $\|Tu_n-Tu_\ell\|<\frac{\varepsilon}{\sqrt{2}}$ for all $n,\ell>N_2$. Choose $N = \max\{N_1,N_2\}$, then for all $n,\ell>N$, we have
    $$
    (\|Tu_n-Tu_\ell\|^2+\|u_n-u_\ell\|^2)^{1/2}<\varepsilon.
    $$
    Thus, $(u_n)$ is Cauchy with respect to the graph norm, hence it has a limit $v$ such that 
    $$
    \|u_n-v\|_g\to 0 \Longrightarrow \|u_n-v\| \to 0 \ \text{ and }\ \|Tu_n-Tv\| \to 0.
    $$
    By uniqueness of the limit, we have $v=u\in \dom T$ and $Tv=Tu=y$.
\end{proof}

It turns out that the adjoint of a densely defined operator is always closed.
\begin{lemma}\label{T^* closed}
    $T^*$ is closed.
\end{lemma}
\begin{proof}
    Take a sequence $(v_n)\subset \dom T^*$ with $v_n\to v$ and $T^*v_n \to g$ for some $v \in \mathcal{K}$ and $g\in\mathcal{H}$. By \cref{def:closedop}(a), it suffices to show that $\langle g,u\rangle_\mathcal{H}=\langle v,Tu\rangle_\mathcal{K}$ for all $u\in \dom T$. The continuity of inner products yields
    $$
        \langle g,u\rangle_\mathcal{H} = \lim_{n\to\infty}\langle T^*v_n,u\rangle_\mathcal{H}=\lim_{n\to\infty}\langle v_n,Tu\rangle_\mathcal{K}=\langle v,Tu\rangle_\mathcal{K},
    $$
    as needed.
\end{proof}

Conversely, we will show that the adjoint of a closed (and densely defined) operator is densely defined. Hence we can define its adjoint, which coincides with the operator itself.
\begin{lemma}\label{T^**=T}
    If $T$ is closed, then $T^*$ is densely defined and $T^{**}:=(T^*)^*=T$, meaning that $\dom(T^{**})=\dom(T)$ and for all $u \in \dom(T): \ Tu=T^{**}u$.
\end{lemma}
\begin{proof} \hspace{-0.2cm}\footnote{This proof is an adaptation of the one in \url{https://proofwiki.org/wiki/Densely-Defined_Linear_Operator_is_Closable_iff_Adjoint_is_Densely-Defined}}
    Define the operator $V:\mathcal{H}\times\mathcal{K}\to\mathcal{K}\times\mathcal{H}$ by $V(x,y)=(-y,x)$. Assume $h\in (\dom T^*)^\perp$, then $\langle h,u\rangle_\mathcal{K}=0$ for all $u\in\dom (T^*)$, so $\langle h,u\rangle_\mathcal{K}+\langle 0,T^*u\rangle_\mathcal{H}=0$, which implies that $(h,0)\in(\mathcal{G}(T^*))^\perp.$ Thus, $(0,h)\in V((\mathcal{G}(T^*))^\perp)=(V(\mathcal{G}(T^*)))^\perp$. Indeed, since $V$ is surjective, for any $x\in \mathcal{H}\times \mathcal{K}$, there exists some $y\in \mathcal{K}\times \mathcal{H}$ such that $x=Vy$. Then
    $$
    x\in(V(\mathcal{G}(T^*)))^\perp\Leftrightarrow \langle x,V\psi\rangle=0, \forall\ \psi\in \mathcal{G}(T^*)\Leftrightarrow \langle y,\psi\rangle=0,\forall\ \psi\in\mathcal{G}(T^*),
    $$
    where the last equivalence follows from $V$ preserving inner products. Hence
    $$
    x\in(V(\mathcal{G}(T^*)))^\perp\Leftrightarrow x\in V((\mathcal{G}(T^*))^\perp),
    $$
    as claimed.

    Note that $v\in \dom T^*$ if and only if there exists $f=T^*v \in \mathcal{H}$ such that for all $u\in \dom T$,
    $$\langle v,-Tu\rangle_\mathcal{K}+\langle f,u\rangle_\mathcal{H}=0,$$
    which implies that $\mathcal{G}(T^*) = (V(\mathcal{G}(T)))^\perp$. Therefore, given $T$ is closed
    $$
    \mathcal{G}(T)=((V(V(\mathcal{G}(T))))^\perp)^\perp=(V((V(\mathcal{G}(T)))^\perp))^\perp= (V(\mathcal{G}(T^*)))^\perp.
    $$
    Hence, $(0,h)\in \mathcal{G}(T)$, which yields that $h=0$. Thus, $T^*$ is densely.
    
    Finally, given above and $T$ is closed, we have
    $$
    \mathcal{G}((T^*)^*)=(V(\mathcal{G}(T^*)))^\perp=V((\mathcal{G}(T^*))^\perp)=V(V(\mathcal{G}(T)))=\mathcal{G}(T).
    $$
    This concludes our proof.
\end{proof}

We have seen that the property of being densely defined and closed is preserved under taking adjoints, which makes such operators natural objects to study when working with adjoints.

\subsection{Self-Adjoint Operators and Quadratic Forms}
In this section, we revert to the case of classical operator theory where $\Kk=\Hh$, so we consider densely defined closed operators $T:\Hh \rightarrow \Hh$. We are interested in the following class of such operators. 

\begin{definition}[Self-Adjoint Operators, \citet{schroedingerop}]
    A densely defined operator $T$ is called \emph{symmetric} if $\dom T \subset \dom T^*$ and for all $u\in \dom T$, $T^*u=Tu$. It is called \emph{self-adjoint} if moreover $\dom T=\dom T^*$.
    
    A symmetric operator $T$ is \emph{lower semi-bounded}, if $$m_T:=\inf_{0\neq u \in \dom(T)}\frac{\langle u, Tu\rangle}{\|u\|^2} > -\infty$$
    and \emph{nonnegative} if $m_T \ge 0$.
\end{definition}

In what follows, we will present a more implicit way to define self-adjoint operators, namely via quadratic forms. The remainder of this section contains more detailed and generalized versions of definitions and results given by \citet[pp.~29-31]{schroedingerop}.

\begin{definition}[Sesquilinear Form, \citet{schroedingerop}]
    Let $\Hh$ be a Hilbert space and $d[a]$ a subspace.
    A \emph{sesquilinear form} $a[\cdot,\cdot]$ in $\mathcal{H}$ is a map from its domain $d[a]\times d[a]$ to $\mathbb{K} \in \{\R,\IC\}$ that is linear in its first and anti-linear in its second argument.
    
    The \emph{quadratic form} associated with $a[\cdot,\cdot]$ is defined by
    $$a[u]:=a[u,u]\quad\forall\ u\in d[a].$$
    A quadratic form $a[\cdot]$ is \emph{densely defined} if $d[a]$ is dense in $\Hh$ and \emph{real-valued} if
    $a[u]\in \R$ for all $u \in d[a].$
\end{definition}

In general, a quadratic form $q[\cdot]$ is a map from $d[q]$ to $\IK$ such that for all $x,y\in d[q]$, we have $q[x+y]+q[x-y]=2q[x]+2q[y]$ \citep{kurepa}.  However, in our work, every quadratic form will be associated with a sesquilinear form.

Given a quadratic form, the \emph{polarization identities}
\begin{subequations}
    \begin{align}
    a[u,v]&=\frac{1}{4}(a[u+v]-a[u-v]+ia[u+iv]-ia[u-iv])\quad  \forall\ u,v\in d[a] \quad \text{if } \IK=\IC, \label{eq:polc}\\
    a[u,v]&=\frac{1}{4}(a[u+v]-a[u-v])\quad \forall\ u,v\in d[a] \quad \text{if } \IK=\R, \label{eq:polr}
\end{align}
\end{subequations}
allow us to recover the original sesquilinear form.

\begin{definition}[Lower Semi-Boundedness, Closedness of Quadratic Forms; \citet{schroedingerop}] \label{def:quadratic_form}
    A quadratic form $a=a[\cdot]$ is called \textit{lower semi-bounded} in $\mathcal{H}$ if it is real-valued and 
    $$m_a:=\inf_{0\neq u \in d[a]}\frac{a[u]}{\|u\|^2}>-\infty,$$
and it is called \textit{closed} in $\mathcal{H}$ if, for some (and hence any) $m>-m_a$, the set $d[a]$ is complete with respect to the norm
    $$||u||_{g,m}:=\sqrt{a[u]+m\|u\|^2}.$$
\end{definition}

Quadratic forms give rise to self-adjoint operators. We will make use of this fact later on to generalize persistent Laplacians, by looking at their respective quadratic forms.

\begin{theorem}\textnormal{\citep[Theorem 1.18]{schroedingerop}} \label{a-induced A}
    Let $a$ be a densely defined, lower semi-bounded, and closed induced quadratic form. Then there is a unique self-adjoint operator $A$ satisfying \begin{equation*}\label{domain_of_a}
        \dom A\subset d[a],
    \end{equation*}
    and $a[u,v]=\langle Au,v\rangle$ for all $u\in\dom A,v\in d[a].$
    Moreover, $A$ is lower semi-bounded with $m_a=m_A$ and the domain of $A$ is given by
    $$
    \dom A=\{u\in d[a]:\text{$\exists f\in\mathcal{H}$ such that for all $v\in d[a],\, a[u,v]=\langle f,v\rangle$}\}.
    $$
\end{theorem}

The advantage of introducing operators through quadratic forms is that it guarantees a self-adjoint, and in particular, densely defined, operator without explicitly having to define its domain. We will see an example of this further on in \cref{sec:laplacians}. 

\subsection{Spectra of Self-Adjoint Operators} \label{subsec:spectra}

Spectral theory is a powerful and arguably the most frequently-used tool to analyze operators. In finite-dimensional linear algebra, the spectrum of an operator consists precisely of its eigenvalues.

However, when working with infinite-dimensional spaces, a more general definition of spectrum has proven to be more useful; see \citet[p.~12]{schroedingerop}. We will now summarize some definitions and facts about the spectra of self-adjoint operators. Spectra can be defined in a similar fashion for all closed operators, but the way we define the essential and discrete spectrum will only work for self-adjoint operators, which will be our main focus.

In this section, we assume that $T:\Hh \to \Hh$ is self-adjoint. 

\begin{definition}[Spectrum, \citet{schroedingerop}]
\label{def:spectrum}
    The \emph{spectrum} of $T$ is defined by
    $$
        \sigma(T)=\{z \in \IC: (T-z) \text{ does not have a bounded inverse}  \}.
    $$
    This happens in particular if $T-z$ is not injective, i.e., has non-vanishing kernel. The \emph{point spectrum} is defined as
    $$
        \sigma_p(T):=\{z \in \IC: \ker(T-z)\not=0 \}.
    $$
    Values $\lambda$ in the point spectrum are called \emph{eigenvalues} and elements $u \in \ker(T-\lambda)$ are \emph{eigenvectors}. The \emph{multiplicity} of $\lambda$ is the dimension of $\ker(T-\lambda)$.
\end{definition}
    
    In finite-dimensional spaces, all eigenvalues are of finite multiplicity and the point spectrum forms a discrete subset of $\R$. However, in infinite dimensions, this is not true in general. The part of the spectrum of $T$ that behaves as in the finite-dimensional setting is referred to as the \emph{discrete spectrum} $\sigma_\text{disc}(T)$.  In particular, $\sigma_\text{disc}(T)$ contains precisely the isolated points of $\sigma(T)$ that are eigenvalues of finite multiplicity.
    
    The reminder of the spectrum is referred to as the \emph{essential spectrum}:
    $$
        \sigma_e(T)=\sigma(T) \setminus \sigma_\text{disc}(T).
    $$

We now overview the results on spectra that will be important for our work.

\begin{lemma} \label{cor:saspecclosed}
    If $T$ is self-adjoint, then $\sigma(T)$ is closed.
\end{lemma}
\begin{proof}
    It follows from \citet[Lemma 1.6]{schroedingerop} that $z\in \sigma(T)$ if and only if for any $\varepsilon>0$, there exists some $u\in \dom T$ with $\|(T-z)u\|<\varepsilon\|u\|$.
    
    Let $z\in\overline{\sigma(T)}$. For any $\varepsilon>0$, we have some $\zeta\in (z-\frac{\varepsilon}{2},z+\frac{\varepsilon}{2}) \cap \sigma(T)$. Thus, $\|(T-\zeta)u\|<\frac{\varepsilon}{2}\|u\|$ for some $u\in \dom T$. Then
    $$
    \|(T-z)u\|\leq \|(T-\zeta)u\|+\|(\zeta-z)u\|<  \varepsilon\|u\|.
    $$
    This implies that $z\in\sigma(T)$, so the proof is completed.
\end{proof}

\begin{lemma} \label{rem:spectrum} The following hold:
    \begin{enumerate}
        \item[(a)] If the Hilbert space $\Hh$ is finite dimensional, then $\sigma(T)=\sigma_p(T)=\sigma_{\text{disc}}(T)$.
        \item[(b)] The essential spectrum $\sigma_e(T)$ is closed.
        \item[(c)] $\lambda \in \sigma_e(T)$ if and only if one or more of the following statements hold: 
        \begin{enumerate}
        \item[(i)] $\im(T - \lambda)$ is not closed; or 
        \item[(ii)] $\lambda$ is an accumulation point of eigenvalues; or 
        \item[(iii)] $\lambda$ is an eigenvalue of infinite multiplicity.
        \end{enumerate}
    \end{enumerate}
\end{lemma}
\begin{proof}
    The proof of (a) is given as Remark 1.5 in \citet{spectraltheory}.
    
    To show (b), assume there is a sequence $(\lambda_n)_n \subset \sigma_e(T)$ that converges to $\lambda \in \sigma(T)$. Then $\lambda \not\in \sigma_\text{disc}(T)$.
    
    Part (c) follows from \citet[p.~26]{schroedingerop}.
\end{proof}

In order to generalize the notion of the $q$th smallest eigenvalue of a positive semidefinite self-adjoint operator to the infinite-dimensional setting, we define the following function, which serves as a generalized notion of eigenvalues. Let $T$ be self-adjoint, nonnegative and $m(T):= \inf\{ z \in \sigma_e(T) \}$.
    Then we define 
    \begin{equation}
    \label{eq:specctng}
    \lambda_{\bullet}(T): \N \rightarrow \R_{\ge 0}
    \end{equation}
    such that $\lambda_q(T)$ is the $q$th smallest eigenvalue (counted with multiplicity) of $T$ that is smaller than $m(T)$. If there are $n(T)$ such eigenvalues, we set
    $$
    \lambda_q(T):= m(T) \qquad \forall\ q > n(T).
    $$

Note that $n(T)$ can be $\infty$ (in which case $q>n(T)$ never applies) and $m(T)= \infty$ if $\sigma_e(T)=\varnothing$.\\
If $\Hh$ is finite dimensional, $\lambda_\bullet (T)$ ranges through all the eigenvalues from the smallest to the largest and then takes the value $\infty$.

Theorems 1.27 and 1.28 of \citet{schroedingerop} provide formulas to compute $\lambda_\bullet(T)$ (where $n(T)$ is called $N_T$). We will only exploit the following result:

\begin{theorem}\textnormal{\citep[Theorem 1.28]{schroedingerop}}
    Let $A$ be self-adjoint and lower semi-bounded with corresponding quadratic form $a$. Then for all $n \in \N$ \label{thm:ari1.28}
    $$
        \lambda_n(A) = \inf_{\substack{u_1,\dots,u_n \subset d[a] \\ \text{lin.~indep.}}} \ \sup_{0 \not= u \in \text{span}\{u_1,\dots,u_n\} } \ \frac{a[u]}{||u||^2}.
    $$
\end{theorem}
Again, we follow the convention that $\inf\varnothing=\infty$.

\subsection{The Polar Decomposition of the Operators $T^*T$ and $TT^*$}
Let $T:\mathcal{H}\to \mathcal{K}$ be a closed, densely defined operator between two Hilbert spaces. Consider the sesquilinear form and its associated quadratic form
$$a[u,v]=\langle Tu,Tv\rangle,\quad a[u]=\|Tu\|^2$$
on $\mathcal{H}$ with $d[a]=\dom T$. According to \cref{a-induced A}, there exists a unique self-adjoint operator induced by $a[u]$, which we will denote by $T^*T$.

\begin{remark}\label{remark:T*T}
    Note that $a[u]$ is indeed densely defined, lower semi-bounded, and closed. The closedness of $a[u]$ is inherited from $\dom T$. It is lower semi-bounded because $\|Tu\|^2\geq 0$ for any $u\in\dom T$. Finally, since $m_a\geq0$, we can choose $m=1$ in \cref{def:quadratic_form} to show $d[a]$ is complete with respect to $\|\cdot\|_g:=\sqrt{a[u]+\|u\|^2}=\sqrt{\|Tu\|^2+\|u\|^2}$. However, given that $T$ is closed and by \cref{def:closedop}, $\dom T$ is complete with respect to $\|\cdot\|_g$, which implies that $a[u]$ is also closed.
\end{remark}

\begin{lemma}\label{lm:equivalence of T^*T}
    The operator $T^*T$ defined via quadratic form $a[u]$ is equivalent to the operator $T^*\circ T$ defined on
    $$\dom T^*\circ T=\{u\in\dom T:Tu\in\dom T^*\}.$$
\end{lemma}
\begin{proof}
    Note that from \cref{a-induced A}, we have
    \begin{align*}
        \dom T^*T&=\{u\in \dom T:\text{$\exists f\in\mathcal{H}$ such that for all $v\in \dom T,\, \langle Tu, Tv\rangle=\langle f,v\rangle$}\}\\
        &=\{u\in \dom T:Tu\in \dom T^*\}\\
        &=\dom T^*\circ T.
    \end{align*}
    Then by the uniqueness of $f$, we have $T^*T\equiv T^*\circ T$.
\end{proof}

\begin{remark}\label{remark:not always equivalent}
    \cref{lm:equivalence of T^*T} shows that it is not necessary to distinguish between $T^*T$ and $T^*\circ T$. However, in general the assumption that operators defined via the quadratic forms coincide with the ones defined via operator composition may fail, and we will explicitly point this out when it occurs.
\end{remark}

Similarly, we define $TT^*$ via the quadratic form $b[u]:=\|T^*u\|^2$, and it is again equivalent to $T\circ T^*$. Our goal is to relate the spectra of $T^*T$ and $TT^*$.  In our construction towards this goal, we will leverage some well-known results in functional analysis, which we provide in \Cref{appendix:operator_theory} for completeness.

The results we derive in this section toward that end are similar to those presented in \S1.2.2 of \citet{schroedingerop}. However, since they consider the case where $\Hh=\Kk$, we are required to re-derive the results, which we achieve using \emph{spectral measures} of self-adjoint operators. Intuitively, these can be seen as an infinite-dimensional version of the eigendecomposition of self-adjoint matrices; for further details, see \S1.1.6 of \citet{schroedingerop}. In particular, Theorem 1.9 of \cite{schroedingerop} guarantees the existence of spectral measures. Spectral measures provide a natural way to define what it means to apply a function to an operator (for example, the square-root of an operator). 

    Let $A$ be a self-adjoint operator on a Hilbert space $\mathcal{H}$ and $P$ its spectral measure \citep[see][pp.~17-22]{schroedingerop}. For a measurable, $P$-almost everywhere finite function $\varphi$ on $\mathbb{R}$, define
    $$
    \varphi(A):=\int_\mathbb{R} \varphi(\lambda) dP_\lambda,
    $$
    with
    $$
    \dom \varphi(A)=\left\{f\in \mathcal{H}:\int_{\mathbb{R}}|\varphi(A)|^2d\langle P_\lambda f,f\rangle<\infty\right\}.
    $$

By Theorem 1.8(b) of \citet{schroedingerop}, $\varphi(A)$ is self-adjoint as long as $\varphi(r) \in \R$ for all $r \in \R$.  If we consider $P$ as the spectral measure of $T^*T$, we can define
$$
|T|:=(T^*T)^{1/2}:=\int_\mathbb{R}\lambda^{1/2}dP_\lambda,
$$
which is well-defined since $T^*T$ is nonnegative. $|T|$ is the unique self-adjoint, nonnegative operator on $\mathcal{H}$ with
\begin{equation*}
    \||T|f\|=\|Tf\|\qquad \text{for any}\qquad f\in\dom T=\dom|T|.
\end{equation*}
Indeed, we have
$$\dom |T|=\left\{f\in \mathcal{H}:\int_{\mathbb{R}}\lambda d\langle P_\lambda f,f\rangle<\infty\right\},$$
which is equivalent to 
$\dom |T|=\left\{f\in \mathcal{H}:\|Tf\|^2<\infty\right\}=\dom T.$
Moreover,
$$\||T|f\|^2=\langle|T|^2f,f\rangle=\int_{\mathbb{R}}\lambda d\langle P_\lambda f,f\rangle=\|Tf\|^2.$$

Next we discuss how we can relate two different self-adjoint operators.

\begin{definition}\citep[Unitary Equivalence,][p.~105]{teschl}
    Recall that a linear map $U:\mathcal{H}\to \mathcal{K}$ is an isomorphism if it is bijective and preserves the norm. An isomorphism is said to be \emph{unitary} if $U^*U=\text{id}_\mathcal{H}$ and $UU^*=\text{id}_\mathcal{K}$. Two operators $T:\mathcal{H}\to\mathcal{H}$ and $S:\mathcal{K}\to\mathcal{K}$ are said to be \emph{unitarily equivalent} if there exists a unitary isomorphism $U$ such that  
    $$U(\dom T)=\dom S\quad \text{and}\quad T=U^*SU.$$
\end{definition}
Notice that unitary equivalence is an equivalence relation.
The following result substantiates the important properties shared by unitarily equivalent operators.

\begin{lemma}\textnormal{\citep[see e.g.,][\S73]{akglazman}}\label{lem:allthesame}
    If two self-adjoint operators are unitarily equivalent, then they have the same eigenvalues with the same multiplicity, the same essential spectrum, and the same spectrum.
\end{lemma}

\begin{proof}
   Assume $T:\mathcal{H}\to\Hh$ and $S:\Kk\to\Kk$ are unitarily equivalent, with $T=U^*SU$ for some unitary isomorphism $U:\Hh\to\Kk$. Then
    $$(T-\lambda I)^{-1}=\left(U^*(S-\lambda I)U\right)^{-1}=U^*(S-\lambda I)^{-1}U,$$
    if it exists. Thus, $\sigma(T)=\sigma(S)$. Moreover, note that
    $$x\in\ker (T-\lambda I)\Leftrightarrow Ux\in \ker(S-\lambda I),$$ so $U:\ker (T-\lambda I)\to\ker (S-\lambda I)$ is a unitary bijection. Hence, the eigenvalues and their multiplicities agree. Together with \cref{rem:spectrum}(c) this yields $\sigma_{\text{ess}}(T)=\sigma_{\text{ess}}(S)$.
\end{proof}

We will now show that $T^*T$ and $TT^*$ are indeed unitarily equivalent when restricted to the appropriate subspaces.

\begin{lemma}
    \label{lem:u}There exists a unique operator $U:\Hh \to \Kk$ such that $T=U \, |T|$ and 
    \begin{equation}\label{relation_of_U}
        \|Uf\|=\|f\|\quad \text{for all}\quad f\in\overline{\im |T|}=\ker(T)^\perp \text{ and } U|_{\overline{\im |T|}} \text{ is unitary.}
    \end{equation}
    \begin{equation}\label{kernel_of_U}
        \ker U = \ker T,
    \end{equation}
    \begin{equation}\label{range_of_U}
        \im U =\overline{\im T}.
    \end{equation}   
\end{lemma}
\begin{proof}
    We define $U:\im |T|\to\im T$ by $U\,|T|f:=Tf$. To see that it is well-defined, take $f_1,f_2 \in \im |T|$ with $|T|f_1=|T|f_2$. We then have $|T|(f_1-f_2)=0$, hence $||T(f_1-f_2)||=0$ and consequently $Tf_1=Tf_2$.
    
    Moreover, $U$ maps onto $\im T$ and is norm-preserving and therefore an isometry. Since $U$ is bounded, it extends to a unique map from $\overline{\im |T|}$ to $\overline{\im T}$ that by continuity of inner products is again an isometry. Hence, $U(\overline{\im |T|})$ is closed and therefore must be all of $\overline{\im T}$.
    
    To see that $U$ is unitary, take any $g,h \in \Hh$. We have $\langle U^*Ug, h \rangle=\langle Ug, Uh \rangle = \langle g, h \rangle$. Thus, $U^*U = I_H.$
    
    Since $U$ is surjective, for every $k \in K$, there exists $h \in H$ such that $Uh = k$. Then
    $$
    UU^* k = UU^* Uh = U (U^* U)h = U h = k,
    $$
    so $UU^* = I_K$.
    
    Next, we can extend $U$ by zero on $(\im |T|)^\perp$, so that it becomes a bounded operator satisfying (\ref{range_of_U}). Also, note that $\ker T=\ker |T|$. Thus, $\ker U=(\im|T|)^\perp=\ker |T|=\ker T$, by self-adjointness of $|T|$ and \cref{lem:imperp}. $U$ is unique, because the extension of $U$ from $\overline{\im |T|}$ to $\overline{\im T}$ is unique. Hence, \eqref{kernel_of_U} and \eqref{range_of_U} follow immediately. 
\end{proof}

\begin{theorem}
    \label{TT^*=T^*T} Let $T:\mathcal{H}\to\mathcal{K}$ be a densely defined, closed operator. Then $T^*T$ restricted to $(\ker T^*T)^\perp$ is unitarily equivalent to $TT^*$ restricted to $(\ker TT^*)^\perp$.
\end{theorem}

\begin{proof}
    Write $T=U|T|$ and consider $T^*$ with
    \begin{align*}
    \dom T^*&=\{ f\in \mathcal{K}:\ \exists h \in \mathcal{H} \text{ such that for all } g\in \dom T : \ \langle h,g\rangle=\langle f,U|T|g\rangle\}\\
    &=\{ f\in \mathcal{K}:\ \exists h \in \mathcal{H} \text{ such that for all } g\in \dom T: \ \langle h,g\rangle=\langle U^*f,|T|g\rangle\}\\
    &=\{f\in \mathcal{K}|U^*f\in\dom |T|^*\}.
\end{align*}
    Given $|T|=|T|^*$, we have just shown that $\dom |T|U^*=\dom T^*$ and that $T^*f=|T|U^*f$ for all $f \in \dom(T^*)$ (by the uniqueness of $h$). Hence, $T^*=|T|U^*$.
    
    Since $\dom T=\dom |T|$, the previous argument also implies $\dom T^*=\dom TU^*$.
    Thus, for all $u\in \dom T^*=\dom TU^*$,
    $\|T^*u\|=\||T|U^*u\|=\|TU^*u\|.$
    Therefore, $TT^*=(TU^*)^*TU^*$, since their defining quadratic forms agree.
    
    The same argument gives $(UT^*)^*=TU^*$. Since $TU^*$ is closed, we have $UT^*=(TU^*)^*$, which now yields
    \begin{align*}
    \dom (UT^*)^*&=\{ f\in \Hh:\ \exists k \in \mathcal{K} \text{ such that for all } g\in \dom UT^* : \ \langle k,g\rangle=\langle f,UT^* g\rangle\}\\
    &=\{f\in \mathcal{H}|U^*f\in\dom T\}=\dom(TU^*)
\end{align*}
    and by uniqueness, the operators agree. Hence, we have
    $TT^*=UT^*TU^*.$
    
    Let $V:\overline{\im(|T|)} \to \overline{\im(T)}$ be the unitary restriction of $U$ with $V^*=U^*|_{\overline{\im(T)}}$. Note that $\ker(T^*T)=\ker T$ and $\ker(TT^*)^\perp=\ker (T^*)^\perp=\overline{\im (T)}$. From \cref{lem:imperp,lem:u}, we know that $\overline{\im(U^*)}=\ker(U)^\perp=\ker(T)^\perp$ and $\im(T^*)\subset \ker(T)^\perp =\dom(V)$.
    Therefore, if we restrict the above to $\overline{\im(T)}$, we obtain
    $$
    TT^*|_{(\ker TT^*)^\perp}=V(T^*T|_{(\ker T^*T)^\perp})V^*.
    $$
    For $V(\dom T^*T\cap \overline{\im(|T|})=\dom TT^*\cap\overline{\im(T)}$, note that
    \begin{align*}
        \dom TT^*\cap(\ker TT^*)^\perp&=\dom UT^*TU^*\cap(\ker TT^*)^\perp\\
        &=\{x\in(\ker TT^*)^\perp:U^*x\in\dom T^*T\}\\
        &=\{x\in\overline{\im(T)}:V^*x\in\dom T^*T\}\\
        &=\{Vy:y\in\overline{\im(|T|)}\cap \dom T^*T\}\tag{$\dagger$}\label{eq:dagger}\\
        &=V(\dom(T^*T)\cap(\ker T^*T)^\perp).
    \end{align*}
    Here, \eqref{eq:dagger} uses the fact that $V$ is unitary. Thus for any $x\in\overline{\im(T)}$, there exists a unique $y\in\overline{\im(|T|)}$ such that $Vy=x$. Then $V^*x\in \dom T^*T$ implies $V^*Vy=y\in \dom T^*T$. This completes our proof.
\end{proof}

Combining this result with \cref{lem:allthesame}, concretely, for spectra, we have the following result.

\begin{cor} \label{lem:changeorder}
    Let $T:\mathcal{H}\to\mathcal{K}$ be a densely defined, closed operator. Then the nonzero spectra of $TT^*$ and $T^*T$ coincide. In particular, they have the same nonzero eigenvalues with the same multiplicities and for all real $\lambda \not= 0$
    \begin{align*}
        \lambda \in \sigma_e(T^*T) &\Leftrightarrow \lambda \in \sigma_e(TT^*)\\
        \lambda \in \sigma(T^*T) &\Leftrightarrow \lambda \in \sigma(TT^*).
    \end{align*}
\end{cor}

This is because nonzero spectra do not change when restricting operators to the orthogonal complement of their kernels.

\section{The Chain Laplacians} \label{sec:laplacians}

Equipped with these notions from operator theory, we will now establish the chain Laplacians and their persistent versions. The background presented in \Cref{sec:motivation} serves as our starting point, which we will generalize here.  We will abstractify from the underlying simplicial complexes and frame our construction from the more general setting of chain complexes. 

In the following table, we overview and summarize the relationship between objects in \cref{sec:motivation} and their generalized counterparts that we will introduce in this section.

\begin{center}
    \begin{tabular}{c||c}
  For finite dimensional spaces  & Operator theoretic version \\ \hline
    chain complex with inner products on chain groups & Hilbert complex\\
    (everywhere defined, bounded) boundary maps & (closed, densely defined) boundary maps\\
    Hodge Laplacian & Chain Laplacian\\
    Persistent Laplacian & Persistent Chain Laplacian
\end{tabular}
\end{center}

We will close this section with examples of all the generalized objects that we will introduce here in two classes of settings: for persistent cosheaves and for de Rham complexes.  We illustrate the possible forms of spectra arising from persistent Laplacians.

\subsection{Hilbert Complexes}

We begin by presenting the setting in which we will achieve our generalization, namely that of \emph{Hilbert complexes}. These were first systematically studied by \citet{Hilbert} and can be seen as the version of chain complexes of Hilbert spaces that is suitable for operator theoretic studies.

\begin{definition}[Hilbert Complex, \citet{Hilbert}] \label{def:cc}
A \textit{Hilbert complex} $P=(C^P_\bullet,d^{P}_{\bullet})$ is a sequence of Hilbert spaces $C^{P}_{k}$ for $k \in \IZ$ over the same field $\mathbb{K} \in \{\mathbb{R},\mathbb{C}\}$
and connected by densely defined closed linear \textit{boundary homomorphisms} $d^{P}_k : C^{P}_{k} \to C^{P}_{k-1}$ which for all $k$ satisfy 
\begin{enumerate}
    \item[(i)] $\im \! \left(d_{k+1}^P\right) \subset \ker\! \left(d_k^P \right) \subset \dom\! \left(d_k^P \right)$,
    \item[(ii)] $\dom\! \left(d_k^P \right) \cap \, \dom\! \left(d_{k+1}^P \right) ^*$ is dense in $C_k^P$.
\end{enumerate}

\end{definition}

Schematically, a Hilbert complex can be represented in the following way:
\[
\begin{tikzcd}
\cdots \arrow[r] & C^{P}_{k+1}  \arrow[r, "d^{P}_{k+1}"] & C^{P}_{k} \arrow[r, "d^{P}_{k}"] & C^{P}_{k-1} \arrow[r] & \cdots
\end{tikzcd}.
\]

We will usually denote the inner product on $C^P_k$ by $\langle \cdot ,\cdot \rangle^P_k$, but will omit the indices if there is no ambiguity. Cochain complexes (for which the boundary maps increase indices) also fall under this setting after re-indexing.

We need the operators to be densely defined and closed in order to have well-defined adjoints and to be able to use \cref{T^**=T}. \citet{Hilbert} assume that only finitely many Hilbert spaces are nonzero---a condition that we will not need, even though all examples we will consider satisfy it---and works without condition (ii). We will need this restriction for the Laplacians to be densely defined.

An important special case of Hilbert complexes is those for which the operators are defined everywhere and bounded. Such operators are closed, as we will show in \cref{lem:bdopclosed}, and include Hilbert complexes with finite-dimensional Hilbert spaces, as most complexes studied TDA are.

\begin{lemma} \label{lem:bdopclosed}
The following properties hold.
\begin{enumerate}
    \item[(a)] Every bounded operator $T:\Hh \to \Kk$ defined on a closed subspace of $\mathcal{H}$ is closed.
    \item[(b)] Let $T:\Hh \rightarrow \Kk$ be a linear operator and $\dim(\Hh)<\infty$ or $\dim(\Kk)<\infty$. Then $T$ is everywhere defined and bounded if and only if it is densely defined and closed.
\end{enumerate}
\end{lemma}

\begin{proof}
    To show (a), let $(u_n)\subset \dom T$ be a sequence converging to some $u\in \mathcal{H}$ and $(Tu_n)\subset \mathcal{K}$ converging to some $y\in \mathcal{K}$. Since $\dom T$ is closed, $u \in \dom T$. Note that
$$\|Tu_n-Tu\| = \|T(u_n-u)\| \leq \|T\|\|u_n-u\| \to 0, \text{ as }n\to \infty,$$
since $\|T\|<\infty.$ Therefore $y=Tu$. Hence $T$ is a closed operator.

For (b), we have just seen that boundedness implies closedness. For the converse, consider first the case that $\dim(\Hh)=n<\infty$. Since $T$ is densely defined, we find $n$ linearly independent vectors $v_1,\dots,v_n \in \dom(T)$; otherwise, $\dom(T)$ would be contained in an $(n-1)$ dimensional subspace of $\Hh$ and would therefore not be dense. By linearity, $\dom(T)=\Hh$. Now we can pick an orthonormal basis $\{u_1,\dots,u_n\}$ of $\Hh$. Then for all normal vectors $u \in \Hh$, we have $u=\sum_{k=1}^n \alpha_k u_k$ for some coefficients $\alpha_k \in \IK$ with $\sum_{k=1}^n |\alpha_k|^2 = 1$. Hence 
$$
\|T(u)\|\le \sum_{k=1}^n |\alpha_k|\|T(u_k)\| \le \sum_{k=1}^n \|T(u_k)\|.
$$
Thus, $T$ is bounded.

In the case that $\dim(\Kk)<\infty$, we have that $T^*:\Kk \rightarrow \Hh$ is densely defined and closed, and consequently, everywhere defined and bounded. Now take any $v \in \Kk$ and define the functional 
$$
f_v:\Hh \rightarrow \IC \qquad f_v(u):=\langle T^*u,v\rangle_\Kk.
$$
Notice that $f_v$ is linear. Moreover, it is bounded, since 
$$
||f_v(u)|| \le ||v|| \, ||T^* u|| \le ||v||\, ||T^*|| \, ||u||.
$$
Hence, by Riesz' Representation Theorem \citep[p.~90]{yosida}, $f_v(u)=\langle g,u\rangle_{\Hh}$ for some $g \in \Hh$ with $||g||=||f_v||\le ||v||\,||T^*||$. By the definition of adjoints, this implies $g=(T^*)^* v = Tv$. Hence, $T$ is everywhere defined and bounded, since $||Tv||=||g||\le ||v||\,||T^*||$.
\end{proof}
 
Just as for chain complexes, the homology of Hilbert complexes is an object of central interest \citep{Hilbert}. The following version turns out to align well with the properties of Hilbert complexes.

\begin{definition}\textnormal{\citep[Reduced Homology,][p. 466]{homologyclosure}} \label{def:redhom}
Given a Hilbert complex $P$, its \textit{reduced homology groups} are $\overline{H_k}(P) := \ker(d_k)/\overline{\im(d_{k+1})}$ for every $k \in \IZ$.
\end{definition}

In the literature of chain complexes of Banach spaces \citep{loeh}, both reduced homology and homology (i.e., $H_k(P)=\ker(d_k)/\im(d_{k+1})$) are studied and have similar properties. However, from a category theoretic perspective, it turns out that reduced homology is the more natural concept to study; see \cref{appendix:category} for further details. Note that whenever the images of the boundary maps are closed, these two versions of homology coincide, which is, in particular, the case when working with finite-dimensional Hilbert spaces.

\subsection{The Chain Laplacian} \label{sub:chainlap}

We have seen in \cref{sec:motivation} that for chain complexes of Hilbert spaces, we can define combinatorial Hodge Laplacians whose kernel is canonically isomorphic to the homology of the chain complex. We will now establish these notions and results in our more general setting.  To this end, we fix a Hilbert complex $P$ for this section (and omit the index $P$ when referring to its boundary maps).
We want to define concatenations and sums of operators in terms of their quadratic forms using \cref{a-induced A}. Note that for every $k$, the adjoint $d^*_k$ of $d_k$ is well-defined.

\begin{definition}[Up-Chain Laplacian]
\label{def:kth up chain laplacian}
For every $k$, consider the sesquilinear form and its associated quadratic form $$\delta^P_{+,k}[u,v]=\langle d_{k+1}^*u,d_{k+1}^*v\rangle,\qquad \delta_{+,k}^P[u]=\|d^*_{k+1}u\|^2$$ on $C_k^P$ with $d[\delta^P_{+,k}]=\dom d^*_{k+1}$. We call the unique self-adjoint operator $\Delta_{+,k}^P$ induced by $\delta^P_{+,k}$ the \textit{$k$th up chain Laplacian}. 
\end{definition}

Note that $\Delta_{+,k}^P$ is well defined as argued in \cref{remark:T*T}. Moreover, \cref{lm:equivalence of T^*T} also applies to this case. Hence $\Delta^P_{+,k}\equiv d_{k+1}\circ d^*_{k+1}$. The advantage of having defined $\Delta^P_{+,k}$ via a quadratic form is that \cref{a-induced A} ensures that $\dom \Delta^P_{+,k}$ is dense in $C_k^P$.

\begin{definition}[Down-Chain Laplacian]\label{def:kth down chain laplacian}
    For every $k$, consider the sesquilinear form and its associated quadratic form $$\delta^P_{-,k}[u,v]=\langle d_{k}u,d_{k}v\rangle,\qquad \delta_{-,k}^P[u]=\|d_{k}u\|^2$$ on $C_k^P$ with $d[\delta^P_{-,k}]=\dom d_{k}$. We call the unique self-adjoint operator $\Delta_{-,k}^P$ induced by $\delta^P_{-,k}$ the \textit{$k$th down chain Laplacian}. 
\end{definition}
Following similar arguments, $\Delta_{-,k}^P$ is well-defined and we will use $d_k^*\circ d_k$ to represent it. Finally, we can define our full chain Laplacian.

\begin{definition}[(Full) Chain Laplacian]
\label{def:kth chain Laplacian}
    Consider the sesquilinear form and its associated quadratic form
$$
\delta^P_k[u,v]=\delta_{-,k}^P[u,v]+\delta_{+,k}^P[u,v], \qquad \delta^P_k[u]=\delta_{-,k}^P[u]+\delta_{+,k}^P[u]
$$
on $C^P_k$ with $d[\delta^P_k]=\dom d_k\cap\dom d_{k+1}^*$. We call the unique self-adjoint operator $\Delta^P_k$ induced by $\delta^P_k$ the \textit{$k$th chain Laplacian}.
\end{definition}

Note $\delta^P_k$ is again closed and hence $\Delta^P_k$ is well-defined: Let $(u_n)\subset d[\delta^P_k]$ be Cauchy with respect to $\|\cdot\|_{g,2}=\sqrt{\delta^P_k[\cdot]+2\|\cdot\|^2}$. Then $(u_n)$ is also Cauchy with respect to $\|\cdot\|_-:=\sqrt{\delta^P_{-,k}[\cdot]+\|\cdot\|^2}$ and $\|\cdot\|_+:=\sqrt{\delta^P_{+,k}[\cdot]+\|\cdot\|^2}$. Since both $\delta_{-,k}^P$ and $\delta_{+,k}^P$ are closed, $u_n$ converges to the same $u$ with respect to both $\|\cdot\|_-$ and $\|\cdot\|_+$. Indeed, if $u_n\to v_1$ with respect to $\|\cdot\|_-$, and $u_n\to v_2$ with respect to $\|\cdot\|_+$, we must have that $\|u_n-v_1\|^2\to 0$ and $\|u_n-v_2\|^2\to 0$ since both $\delta_{-,k}^P$ and $\delta_{+,k}^P$ are nonnegative. Thus $v_1=v_2$. Hence $u\in\dom d_k\cap\dom d_{k+1}^*=d[\delta^P_k]$, and $u_n\to u$ with respect to $\|\cdot\|_{g,2}$, which shows it is closed. Note that such $\Delta^P_k$ is also non-negative. 

\begin{remark}\label{remark:not the sum}
    $\Delta ^P_k$ is one of the exception we mentioned in \cref{remark:not always equivalent}.Even though $\Delta^P_{+,k}$ and $\Delta^P_{-,k}$ are equivalent to $d_{k+1}\circ d^*_{k+1}$ and $d^*_{k}\circ d_{k}$ respectively, $\Delta^P_{k}$ is not always equivalent to $d_{k+1}\circ d^*_{k+1}+d_k^*\circ d_k$, as their domains may differ. In general, $\dom \Delta_{+,k}^P\cap\dom\Delta_{-,k}^P$ is a proper subset of $\dom\Delta_k^P$, and hence may not be dense. Defining $\Delta_k^P$ via the quadratic form ensures that it is densely defined, and hence self-adjoint. Additionally, note that $\Delta^P_{k}$ does coincide with $d_{k+1}\circ d^*_{k+1}+d_k^*\circ d_k$ on $\dom \Delta_{+,k}^P\cap\dom\Delta_{-,k}^P$. This implies that our construction is a generalization of finite dimensional Hilbert spaces or bounded operators cases, since either case implies both $\dom \Delta_{+,k}^P\cap\dom\Delta_{-,k}^P$ and $\dom\Delta_k^P$ are the whole space and so $\Delta^P_{k}\equiv d_{k+1}\circ d^*_{k+1}+d_k^*\circ d_k$.
\end{remark}

Having made clear that $\Delta_k^P$ is in general not the sum of $\Delta_{+,k}^P$ and $\Delta_{-,k}^P$, we will nevertheless use the misleading notation, 
    \[
    \Delta^{P}_k=\underbrace{d_{k+1} \circ (d_{k+1})^{*}}_{\textcolor{red}{\Delta^{P}_{+, k}}} + \underbrace{(d_k)^{*} \circ d_k}_{\textcolor{blue}{\Delta^{P}_{-, k}}},
    \]
which we choose to maintain because on $\dom \Delta_{+,k}^P\cap\dom\Delta_{-,k}^P$ and especially in the important special case of bounded, everywhere defined operators, $\Delta_k^P$ is indeed the sum of these two components.
\begin{remark}\label{rek:all quadratic form induced}
    Despite the potential confusion arising from a slight abuse of notation, in all following discussions, all self-adjoint operators must be understood as those induced by quadratic forms, as in \cref{def:kth up chain laplacian}, \ref{def:kth down chain laplacian}, and \ref{def:kth chain Laplacian}.
\end{remark}

The terms \emph{up}- and \emph{down}-Laplacian refer to the fact that we go ``up'' (respectively ``down'') in the indices of the Hilbert complex to compute them. This is visualized in \cref{fig:lap}.

\begin{figure}[H]
\begin{center} 
\begin{tikzcd}[column sep = +4em] 
\cdots \arrow[r] & C^{P}_{k+1} \arrow[r, "d_{k+1}", shift left, red] & C^{P}_{k} \arrow[r, "d_k", shift left, blue] \arrow[l, "(d_{k+1})^{*}", shift left, red] & C^{P}_{k-1} \arrow[l, "(d_k)^{*}", shift left, blue] \arrow[r] & \cdots
\end{tikzcd}
\caption{The chain Laplacian. The colored parts represent the respective up and down chain Laplacian: $\red{\Delta^{P}_{+, k}}$ and $\textcolor{blue}{\Delta^{P}_{-, k}}$.} \label{fig:lap}
\end{center}
\end{figure}

We now study some properties of the chain Laplacian, in particular its relation to homology. To start we observe the following equalities:
\begin{lemma} \label{lem:kerlap}
    The kernels of up-, down-, and full-persistent Laplacian satisfy
    $$
    \ker(\Delta_k^P) = \ker(d_k^P) \cap \ker\left((d_{k+1}^P)^*\right)=\ker(\Delta_{-,k}^P) \cap \ker\left((\Delta_{+,k}^P)^*\right).
$$
\end{lemma}

\begin{proof}
    For the first equality, consider $u \in d[\delta_k^P]$ such that $\Delta_k^P u =0$. Consequently, we have 
    $$
    0=\langle d_ku,d_kv\rangle+\langle d_{k+1}^*u,d_{k+1}^*v\rangle \quad \forall\ v \in d[\delta_k^P].
    $$
    Inserting $v=u$ yields $u \in \ker(d_k^P) \cap \ker\left((d_{k+1}^P)^*\right)$.
    
    Conversely, for $u \in \ker(d_k^P) \cap \ker\left((d_{k+1}^P)^*\right)$, $f=0$ satisfies the definition of $\dom(\Delta_k^P)$. Thus, $u \in \dom(\Delta_k^P)$ and $\Delta_k^P u=0$.
    
    The same reasoning shows that $\ker(\Delta_{-,k}^P)=\ker(d_k^P)$ and $\ker(\Delta_{+,k}^P)=\ker((d_{k+1}^P)^*)$ and therefore the second equality.
\end{proof}

We also require a version of the classical Hodge decomposition result.
\begin{theorem}\textnormal{\citep[Hodge Decomposition,][Lemma~2.1]{Hilbert}}
    \label{thm:magical}
    Let $U,V,W$ be Hilbert spaces. Consider densely defined, closed linear maps $A:V \to W,\, B:U \to V$ such that $\im B\subset \ker A \subset \dom(A)$ and $\dom(A) \cap \dom(B^*)$ is dense in $V$. Then an \emph{orthogonal Hodge decomposition} is given by
    \begin{align}
        V= \lefteqn{\overbrace{\phantom{\overline{\im(A^*)}\oplus\ker(BB^*+A^*A)}}^{\ker(B^*)}}\overline{\im(A^*)}\oplus\underbrace{\ker(BB^*+A^*A)\oplus\overline{\im(B)}}_{\ker(A)}. \label{eqn:hodgedecomp}
    \end{align}
\end{theorem}
\begin{proof}
    First we want to show that $\overline{\im(B)} \perp \overline{\im(A^*)}$. By continuity of inner products, it suffices to show $\im(B) \perp \im(A^*)$. This holds, since for all $u \in \dom(B)$ and $w \in \dom(A^*)$,
    $$
        \langle Bu, A^* w \rangle = \langle AB \,u,w \rangle =0
    $$
    by the definition of adjoints, where we use that $Bu \in \im(B)\subset \dom(A)$.
    
    By \cref{cor:decompV}, we have the decompositions
   $$
        V=\ker(B^*)\oplus \overline{\im(B)} = \overline{\im(A^*)} \oplus \ker(A).
    $$
    This verifies the components of \cref{eqn:hodgedecomp} that are indicated by the braces. Combining them using the orthogonality of $\overline{\im(B)}$ and $\overline{\im(A^*)}$ yields
    $$
        V= \overline{\im(A^*)} \oplus \ker(A) \cap \ker(B^*)\oplus \overline{\im(B)}
    $$
    The same argument as in the proof of \cref{lem:kerlap} allows us to conclude the desired result.
\end{proof}

The Hodge decomposition immediately implies the following relation between Laplacians and reduced homology.

\begin{cor}[Generalized Hodge Theorem] \label{thm:ker_hom_1}
    Let $P$ be a Hilbert complex. The kernels of the Laplacians are isomorphic to the corresponding reduced homology groups:
    $$
        \ker(\Delta_k^{P}) \cong \overline{H_k}(P) = \ker(d^P_k) / \overline{\im (d^P_{k+1})}
    $$
    for every $k$. The isomorphism sends $v \mapsto [v]$.
\end{cor} 
\begin{proof}
    The Hodge decomposition applied to $U=C_{k+1}^P,\, V=C_k^P,\, W=C_{k+1}^P$, $A=d_k^P$ and $B=d_{k+1}^P$ yields
    $$
        C_k^P = \overline{\im\left(\left(d_k^P\right)^*\right)} \oplus \underbrace{\ker\left(\Delta_k^{P} \right)\oplus \overline{\im\left(d_{k+1}^{P} \right)}}_{\ker\left(d_k^P \right)}.
    $$
    Hence,
   $$
        \ker\left(\Delta_k^{P} \right)\cong\left(\ker\left(\Delta_k^{P} \right)\oplus \overline{\im\left(d_{k+1}^{P} \right)}\right) \big/ \overline{\im\left(d_{k+1}^{P} \right)} = \ker\left(d_k^P\right) \big/ \, \overline{\im\left(d_{k+1}^{P} \right)}
    $$
    with the canonical isomorphism sending $v \mapsto [v]$. 
\end{proof}

Thus, we can think about the reduced homological information as being ``contained in'' the zero eigenspace of the chain Laplacian. One motivation for studying Laplacians comes from the expectation that the spectrum outside the zero eigenspace encodes further geometric information; while we do not pursue this direction here, it forms an important context for our results. 
The following lemma summarizes some facts about spectra of chain Laplacians.

\begin{lemma} \label{lem:properties}
$\sigma(\Delta_k^P), \sigma(\Delta_{+,k}^P)$, and $\sigma(\Delta_{-,k}^P)$ are closed subsets of $\R_{\ge 0}$.
\end{lemma}

This is because self-adjointness implies that all these spectra are real by \citet[Lemma 1.6]{schroedingerop}, and nonnegativity implies that they do not contain negative numbers by \cref{thm:ari1.28}. 

\subsection{The Persistent Chain Laplacian} \label{sec:perslap}

We now return to the persistent setting to generalize the persistent combinatorial Laplacians presented previously in \cref{sec:persmot}. Recall that these were defined for a pair of nested simplicial complexes $K \subset L$ such that $\iota_k:C_k^K\to C_k^L$ are isometries for each $k$. We will now generalize this for Hilbert complexes.

\begin{definition} \label{def:inclusion}
    A Hilbert complex $P=(C^P_\bullet,d^{P}_{\bullet})$ is \textit{contained} in a Hilbert complex $Q=(C^Q_\bullet,d^{Q}_{\bullet})$ and written $P \subset Q$ or $\iota: P \hookrightarrow Q$ if for all $k$, there exist linear $\iota_k: C^{P}_{k} \rightarrow C^{Q}_{k}$, such that
\begin{enumerate}
    \item[(a)] $\iota_{k-1} \circ d^{P}_k = d^{Q}_k \circ \iota_k$, in particular their domains agree: $\dom(d_k^Q) \cap \im(\iota_k)=\iota_k(\dom(d_k^P))$
    \item[(b)] $\langle v,w \rangle^{P}_k = \langle \iota_k v,\iota_k w\rangle^{Q}_k$ for all $v,w \in C_k^P$. 
\end{enumerate}
\end{definition}
Note that (b), the preservation of inner products implies that $\iota_k$ is injective. We will usually identify $C_k^P$ with $\iota_k(C_k^P)$, which is the copy of $C_k^P$ that lives inside $C_k^Q$. Then (a) can be rephrased as $d^Q_k \lvert_{C^P_k} = d^P_k$.

If we are given a pair of Hilbert complexes and want to check if they satisfy \cref{def:inclusion}(b), the polarization identities \cref{eq:polc,eq:polr}
ensure that it is enough to check that the $\iota_k$ preserve the metric. So we only need to check that $\|v\|_k^P=\|\iota_k \, v\|_k^Q$ for all $v \in C_k^P$.\\

The setting described in \cref{def:inclusion} occurs in many situations: In \cref{sec:motivation}, we saw how in finite-dimensional cases, we can define inner products on a pair of chain complexes with a chain map in between them (i.e., a collection of maps that commutes with the boundary maps). Hence, considering the homology of a pair of finite simplicial complexes $S_1 \subset S_2$ places us in the setting of \cref{def:inclusion}. The same holds when considering their \emph{cohomology}.
\begin{example} \label{ex:cohomology}
    In the case of cohomology, we have surjective maps $\pi_k:C^k(S_2) \rightarrow C^k(S_1)$ that commute with the coboundary maps $\partial^k_{S_i}:C^{k-1}(S_i) \to C^k(S_i)$. After choosing compatible inner products on the chain groups, taking adjoints $\iota_k:=(\pi_k)^*$ of the $\pi_k$ and the coboundary maps yields a situation as in \cref{def:inclusion}. One way to choose such compatible inner products is as follows: For every $k$, and any inner product $\langle\cdot,\cdot\rangle_{S_2}$ defined on $C^k(S_2)$, consider $A:=\ker \pi_k$ and $B:=A^{\perp_{\langle\cdot,\cdot\rangle_{S_2}}}$. Then $\pi_k|_{B}$ is a linear isomorphism. Define $T:=(\pi_k|_{B})^{-1}:C^k(S_1)\to B\subset C^k(S_2)$ and an inner product $\langle\cdot,\cdot\rangle_{S_1}$ on $C^k(S_1)$ by a pullback along $T$:
    $\langle u,v\rangle_{S_1}:=\langle Tu,Tv\rangle_{S_2}.$
    It suffices to show that $\pi_k^*=T$ with such setting. Indeed, for any $x=a+b\in C^k(S_2)=A\oplus B$ and $y\in C^k(S_1)$, note that
    $$\langle\pi_kx,y\rangle_{S_1}=\langle \pi_k b,y\rangle_{S_1}=\langle T^{-1}b,y\rangle_{S_1}=\langle TT^{-1}b,Ty \rangle_{S_2}=\langle b,Ty\rangle_{S_2}=\langle x,Ty\rangle_{S_2}.$$
    Then by uniqueness of the adjoint, $\pi_k^*=T$. Hence the maps $\iota_k$ satisfy \cref{def:inclusion}(b). 
    
Moreover, they commute with the boundary maps, because $$\iota_{k-1} \circ d_k^{S_1}= (\pi_{k-1})^* \circ (\partial^k_{S_1})^* = ( \partial^k_{S_1} \circ \pi_{k-1})^* = ( \pi_{k} \circ\partial^k_{S_2} )^*=(\partial^k_{S_2})^* \circ (\pi_k)^* = d_{k-1}^{S_2} \circ \iota_k,$$
satisfying \cref{def:inclusion}(a).
\end{example} 

Once we have a Hilbert complex that is contained in another one, we can define an operator for this pair which generalizes the combinatorial persistent Laplacians defined previously in \cref{def:comb_PL}.

\begin{definition}[Generalized Persistent Chain Laplacian]
\label{def:pers_chain_lap}
    Consider a pair of Hilbert complexes, $P \subset Q$. \Cref{fig:perslap} shows the relevant spaces. We define the auxiliary space 
    $$
    C^{P,Q}_{k+1} := \overline{\{c \in \dom(d^{Q}_{k+1}) | d^{Q}_{k+1}(c) \in \iota_k(C^{P}_k)\}}
    $$
    and equip it with the restriction of the inner product of $C^{Q}_{k+1}$.\\
    Moreover, we define the map $d^{P,Q}_{k+1}:C^{P,Q}_{k+1} \rightarrow C_k^P$ 
    with
    $$
    \dom(d^{P,Q}_{k+1})=\dom(d^Q_{k+1})\cap C^{P,Q}_{k+1}=\{c \in \dom(d^{Q}_{k+1}) | d^{Q}_{k+1}(c) \in \iota_k(C^{P}_k)\}
    $$
    by
    $$
    d^{P,Q}_{k+1} := d^{Q}_{k+1}|_{\dom(d^{P,Q}_{k+1})}.
    $$
    If for the pair $P,Q$, $\dom(d_k^P) \cap \dom(d_{k+1}^{P,Q})^*$ is dense in $C_k^P$, then we define the \emph{(full) $k$th persistent Laplacian} to be $\Delta^{P,Q}_k : C^{P}_k \to C^{P}_k$ given by 
\[
\Delta^{P,Q}_k := \underbrace{d^{P,Q}_{k+1} \circ (d^{P,Q}_{k+1})^{*}}_{\red{\Delta^{P,Q}_{+, k}}} + \underbrace{(d^{P}_k)^{*} \circ d^{P}_k}_{\textcolor{blue}{\Delta^{P,Q}_{-, k}}}.
\]
The terms $\Delta^{P,Q}_{+, k}$ and $\Delta^{P,Q}_{-, k}$ are called the \emph{$k$th up persistent chain Laplacian} and \emph{$k$th down persistent chain Laplacian}, respectively.
\end{definition}

\begin{figure}
\begin{center}
\begin{tikzpicture}[x=1cm,y=1cm]
        \clip(0,-0.8) rectangle (9.5,5);
        \draw [-{Stealth[length=2mm]}] (1.6,4) --  node[above=1pt] {$d_{k+1}^{P}$} (4.5,4);
        \draw [color=blue,{Stealth[length=2mm]}-] (5.5,3.9) --  node[below=1pt] {$\left(d_{k}^{P}\right)^*$} (8.4,3.9);
        \draw [color=blue,-{Stealth[length=2mm]}] (5.5,4.1) --  node[above=1pt] {$d_{k}^{P}$} (8.4,4.1);
        \draw [-{Stealth[length=2mm]}] (1.6,0) --  node[below=1pt] {$d_{k+1}^{Q}$} (4.5,0);
        \draw [-{Stealth[length=2mm]}] (5.4,0) --  node[below=1pt] {$d_{k}^{Q}$} (8.4,0);
        \draw [color=red,-{Stealth[length=2mm]}] (2.9,2) --  node[anchor=south east] {$d_{k+1}^{P,Q}$} (4.7,3.8);
        \draw [color=red,{Stealth[length=2mm]}-] (3,1.8) -- (4.8,3.6);
        \draw [right hook-{Stealth[length=2mm]}] (1,3.6) -- node[left]{$\iota_{k+1}$} (1,0.4);
        \draw [right hook-{Stealth[length=2mm]}] (5,3.6) -- node[right]{$\iota_{k}$}(5,0.4);
        \draw [right hook-{Stealth[length=2mm]}] (9,3.6) -- node[left]{$\iota_{k-1}$}(9,0.4);
        \draw [right hook-{Stealth[length=2mm]}] (2.2,1.2) -- (1.2,0.3);
        \draw[color=red] (4.2,2) node {$\left(d_{k+1}^{P,Q}\right)^*$};
        \draw[color=black] (1,4) node {$C_{k+1}^{P}$};
        \draw[color=black] (5,4) node {$C_k^{P}$};
        \draw[color=black] (9,4) node {$C_{k-1}^{P}$};
        \draw[color=black] (9,0) node {$C_{k-1}^{Q}$};
        \draw[color=black] (5,0) node {$C_k^{Q}$};
        \draw[color=black] (1,0) node {$C_{k+1}^{Q}$};
        \draw[color=black] (2.5,1.5) node {$C_{k+1}^{P,Q}$};
    \end{tikzpicture}
\caption{Defining the generalized persistent chain Laplacian.} \label{fig:perslap}
\end{center}
\end{figure}
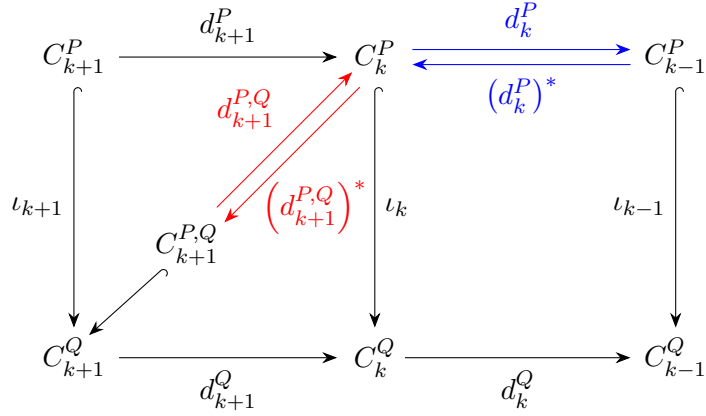

Recall from \cref{rek:all quadratic form induced} that the above is defined using quadratic forms. Part (a) of \cref{lem:pershodge} ensures that we can indeed define $\Delta_k^{P,Q}$ via a quadratic form.

Notice that the persistent chain Laplacian is a generalization of the chain Laplacian: if $P=Q$, we have $C^{P,Q}_{k+1} = C^{Q}_{k+1}$ and $d^{P,Q}_{k+1} = d^{Q}_{k+1}$.

We now state some observations about the newly defined objects.

\begin{lemma} \label{lem:pershodge}
For $C^{P,Q}_{k+1}$, $d^{P,Q}_{k+1}$, and $\Delta^{P,Q}_k$ as in \Cref{def:pers_chain_lap}, the following hold:

\begin{enumerate} 
    \item[(a)] $d^{P,Q}_{k+1}$ is a densely defined, closed operator.
    \item[(b)] $\im(d^{P,Q}_{k+1}) \subset \ker(d_k^P)$.
    \item[(c)] $\iota_{k+1}(C_{k+1}^P) \subset C_{k+1}^{P,Q}$.
    \item[(d)]  All $\Delta_k^{P,Q}$ are self-adjoint and non-negative and thus have non-negative spectra.
\end{enumerate}
\end{lemma}

\begin{proof}
To see that $d_{k+1}^{P,Q}$ is closed, let $(u_n)\subset \dom d^{P,Q}_{k+1}$ be Cauchy with respect to the graph norm of $d^{P,Q}_{k+1}$. Then $(u_n)\subset \dom d^{Q}_{k+1}$ and it is also Cauchy with respect to the graph norm of $d^{Q}_{k+1}$ by the definition of $d^{P,Q}_{k+1}$. Since $d^{Q}_{k+1}$ is closed, $u_n$ has a limit $u\in \dom(d^Q_{k+1})$. Moreover, given $\|d^{Q}_{k+1}u_n-d^{Q}_{k+1}u\|+\|u_n-u\|\to 0$, we must have $\|u_n-u\|\to 0$. Thus $u_n\to u$ with respect to the usual norm. Since $C^{P,Q}_{k+1}$ is closed, $u\in C^{P,Q}_{k+1}$. Hence, $u\in \dom(d^Q_{k+1})\cap C^{P,Q}_{k+1}$, so $d_{k+1}^{P,Q}$ is closed.

For (b), take any $v \in \dom(d_{k+1}^{P,Q})$. Then 
\begin{align}
    \iota_{k} \circ d_{k+1}^{P,Q} \, v&=d_{k+1}^Q  \,v \in \dom(d_k^Q) \qquad \text {and} \nonumber\\
    \iota_{k-1} \circ d_k^P \circ d_{k+1}^{P,Q} \, v&=d_k^Q \circ \iota_{k} \circ d_{k+1}^{P,Q} \, v = d_k^Q \circ d_{k+1}^Q  \,v =0. \label{eq:from_def_pcl}
\end{align}
The first equality of \cref{eq:from_def_pcl} is due to \cref{def:pers_chain_lap} (a). The injectivity of $\iota_{k-1}$ allows us to conclude.

Parts (c) and (d) are immediate.
\end{proof}

These results allow us to obtain one of our main contributions, namely, a persistent version of \cref{thm:ker_hom_1}, which yields the generalized Hodge theorem that establishes isomorphism between the kernels of the Laplacians and the corresponding reduced homology groups.  As in the case of classical homology, an inclusion $\iota: P \rightarrow Q$ induces maps $\iota_{*}: \overline{H_{k}}(P) \to \overline{H_{k}}(Q)$ for every $k$ in reduced homology.

\begin{theorem} \label{thm:ker_hom_2}
Let $P \subset Q$ be a pair of Hilbert complexes. For every $k$, if $\im(d_{k+1}^Q)$ is closed, then the kernel of the $k$th persistent chain Laplacian is isomorphic to the pushforward of $\iota$ to homology level:
    $$
        \ker(\Delta_k^{P,Q}) \cong \im\big(\iota_*: \overline{H_k}(P) \rightarrow \overline {H_k}(Q)\big)=\iota_k(\ker(d_k^P)) \big/ \, \big( \overline{\im(d_{k+1}^Q)} \cap \iota_k(\ker\left(d_k^P\right))\big).
    $$
    The isomorphism sends $v \mapsto [\iota_k(v)]$.
\end{theorem}

\begin{proof}
    By \cref{lem:pershodge}(b), we can apply the Hodge decomposition to obtain 
    $$
        \ker(\Delta_k^{P,Q}) \cong \ker(d_k^P) \big/ \, \overline{\im(d_{k+1}^{P,Q} )}
    $$
    with the canonical isomorphism sending $v \mapsto [v]$. Since $\iota_k$ is an injection, we may apply it and embed the right-hand side into $C_k^Q$:
    \begin{align}
        \ker(\Delta_k^{P,Q}) \cong \iota_k\big(\ker(d_k^P)\big) \big/ \, \iota_k\big(\overline{\im(d_{k+1}^{P,Q} )}\big). \label{eq:pfkerishom}
    \end{align}
    Now the isomorphism sends $v \mapsto [\iota_k(v)]$. By definition, 
    $$\im(d_{k+1}^Q) \cap \iota_k (C_k^P)=\iota_k (\im(d_{k+1}^{P,Q})).$$
    Thus, also $\iota_k (\im(d_{k+1}^{P,Q}))$ is closed and since $\iota_k$ preserves closedness (by \cref{closedstaysclosed}) we have
    $$
        \iota_k\big(\overline{\im(d_{k+1}^{P,Q} )}\big)=\iota_k (\im(d_{k+1}^{P,Q} ))= \im(d_{k+1}^Q) \cap \iota_k (C_k^P)=\overline{\im(d_{k+1}^Q)} \cap \iota_k (C_k^P) \supset \overline{\im(d_{k+1}^Q)} \cap \iota_k (\ker(d_k^P)).
    $$
    
    On the other hand, by \cref{lem:pershodge}(b) and the definition of $d_{k+1}^{P,Q}$, we have $$\iota_k (\im(d_{k+1}^{P,Q} )) \subset \overline{\im(d_{k+1}^Q)} \cap \iota_k (\ker(d_k^P)).$$
    Since the right-hand side is closed, the inclusion remains true when taking the closure of the left-hand side, which yields $\iota_k (\overline{\im(d_{k+1}^{P,Q}} ) = \overline{\im(d_{k+1}^Q)} \cap \iota_k (\ker(d_k^P)).$
    
    Combining these findings with \cref{eq:pfkerishom} concludes the proof.
\end{proof}

Without $\im(d_{k+1}^Q)$ being closed, the statement of \cref{thm:ker_hom_2} need not hold true, which we demonstrate in the following example.

\begin{example} \label{ex:kernotPH}
    Consider the following pair of Hilbert complexes:
    \begin{figure}[H]
    \centering
    \begin{tikzcd}
        C_2^P=\{0\}  \arrow[r,"0"] \arrow[d,hook] & C_1^P=\{f(x)=c\, | \, c \in \R \} \arrow[r,"0"] \arrow[d,hook] & C_0^P=\{0\} \arrow [d,hook]\\
        C_2^Q=L^2([0,1]) \arrow[r,"f \mapsto x f"] & C_1^Q=L^2([0,1]) \arrow[r,"0"] & C_0^Q=\{0\}
    \end{tikzcd}
\end{figure}
Then $C^{P,Q}_2=\{0\}$. Moreover, a simple limit construction shows that
\begin{align*}
    \overline{\im(d_2^Q)}&=C_1^Q\\
    \iota_1(\im(d_2^{P,Q}))&=\im(d_2^Q)\cap \iota_1 (C_1^P)=\{0\}.
\end{align*}
Consequently,
\begin{align*}
    \ker(\Delta_1^{P,Q})&\cong \ker(d_1^P) \big/ \overline{\im(d_2^{P,Q})} = C_1^P \\
    &\not \cong \{0\}=\iota_1(\ker(d_1^P)) \big/ \, \big( \overline{\im(d_{2}^Q)} \cap \iota_1(\ker\left(d_1^P\right))\big)=\im\big(\iota_*: \overline{H_1}(P) \rightarrow \overline {H_1}(Q)\big).
\end{align*}
\end{example}

In the finite-dimensional setting, \cref{thm:ker_hom_2} tells us that the zero eigenspace of persistent Laplacians corresponds to the image of the inclusion map on reduced homology level. However, \cref{ex:kernotPH} shows that this is in general not true. This yields two different ways of defining reduced persistent homology for Hilbert complexes, which both generalize the finite-dimensional version, which we will now present. 

The first way we now consider defines reduced persistent homology as kernel of persistent chain Laplacians:

\begin{definition}[Reduced Kernel Persistent Homology]
    Let $P,Q$ be two Hilbert complexes with $P \subset Q$. The \emph{reduced kernel persistent homology groups} for every $k$ are
    $$
    \overline{H_{\ker,k}}(P,Q) := \ker(\Delta_k^{P,Q}) = \iota_k\big(\ker(d_k^P)\big) \big/ \overline{\im(d_{k+1}^Q) \cap \iota_k (C_k^P)}.
    $$
\end{definition}

The second way is via the image of the pushforward of the inclusion map to the level of reduced homology.

\begin{definition}[Reduced Image Persistent Homology]
    Let $P,Q$ be two Hilbert complexes with $P \subset Q$. The \emph{reduced image persistent homology groups} for every $k$ are 
    $$
    \overline{H_{\im,k}}(P,Q) := \im(\iota_{*} : \overline{H_k}(P) \to \overline{H_k}(Q)).
    $$
\end{definition}

Both of these approaches to defining reduced persistent homology for Hilbert complexes yield the reduced homology of $P$ if $P=Q$.

\subsection{Examples of Persistent Chain Laplacians and their Spectra} \label{sub:examples}

We conclude this section by providing two concrete examples of persistent chain Laplacians. In particular, we present and study \emph{persistent cosheaf Laplacians}, which are a direct generalization of the Laplacians from \cref{sec:motivation}; and Laplacians derived from de Rham complexes of smooth manifolds, which were the main motivation to lift the generality of our setting from bounded operators to densely defined operators on Hilbert spaces. We note that additional types of Laplacians considered in the literature also fall within our framework, such as the persistent hyperdigraph Laplacians introduced in \citet{hyperdigraph}.

\subsubsection{Persistent Cosheaf Laplacians} \label{subsub:cosheaf}
Persistent cosheaf Laplacians generalize the persistent Laplacians derived from filtrations of simplicial complexes as described in \cref{sec:motivation}. Intuitively, cellular cosheaves can be constructed on top of simplicial complexes by assigning vector spaces to each face of the simplicial complex and maps to subface relations in a compatible way (see \S4 of \citet{curry} for full details). They can therefore be seen as a generalization of simplicial complexes. We will work with the following version.

\begin{definition}\textnormal{\citep[Cellular Cosheaf,][\S9.10]{ghrist}}
    Let $K$ be a finite simplicial complex. A \textit{cellular cosheaf} $\Ff$ is an assignment of a finite-dimensional $\IK$ vector space $\Ff(\sigma)$, called a \emph{stalk}, to each face $\sigma \in K$ and a linear map $\Ff_{\tau \subset \sigma}: \Ff(\sigma) \rightarrow \Ff(\tau)$ to each subface relation in a functorial manner, i.e., such that
    $$
        \Ff_{\rho \subset \tau} \circ \Ff_{\tau \subset \sigma}=\Ff_{\rho \subset \sigma} \qquad \forall\ \rho \subset \tau \subset \sigma \in K.
    $$
    The cellular cosheaf is a \emph{cellular Hilbert cosheaf} if all stalks are equipped with inner products.
\end{definition}

\Cref{fig:cellcosheaf} shows an example of a cellular cosheaf. If we equip each stalk with standard inner products, it becomes a Hilbert cosheaf.

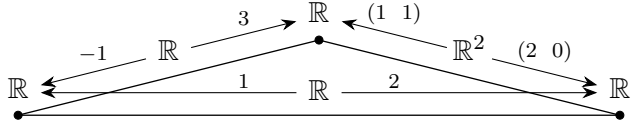
\begin{figure}
    \centering
   \begin{tikzpicture}[x=1cm,y=1cm]
	\clip(0.5,0.5) rectangle (9.5,3);
    \draw [line width=.5pt] (1,1)-- (5,2);
    \draw [line width=.5pt] (5,2)-- (9,1);
    \draw [line width=.5pt] (1,1)-- (9,1);
    \draw [{Stealth[length=2mm]}-] (1.3,1.4) -- (2.7,1.75);
    \draw [-{Stealth[length=2mm]}] (3.3,1.9) -- (4.7,2.25);
    \draw [{Stealth[length=2mm]}-] (5.3,2.25) -- (6.7,1.9);
    \draw [-{Stealth[length=2mm]}] (7.3,1.75) -- (8.7,1.4);
    \draw [{Stealth[length=2mm]}-] (1.3,1.3) -- (4.7,1.3);
    \draw [-{Stealth[length=2mm]}] (5.3,1.3) -- (8.7,1.3);
    \draw[color=black, anchor=south] (1,1.1) node {$\R$};
    \draw[color=black, anchor=south] (5,2.1) node {$\R$};
    \draw[color=black, anchor=south] (9,1.1) node {$\R$};
    \draw[color=black, anchor=south] (3,1.6) node {$\R$};
    \draw[color=black, anchor=south] (7,1.6) node {$\R^2$};
    \draw[color=black, anchor=south] (5,1.07) node {$\R$};
    \begin{scriptsize}
    \draw[color=black, anchor=south] (2,1.6) node {$-1$};
    \draw[color=black, anchor=south] (4,2.1) node {$3$};
    \draw[color=black, anchor=south] (6,2.1) node {$(1 \ \ 1)$};
    \draw[color=black, anchor=south] (8,1.6) node {$(2 \ \ 0)$};
    \draw[color=black, anchor=south] (4,1.25) node {$1$};
    \draw[color=black, anchor=south] (6,1.25) node {$2$};
	\draw [fill=black] (1,1) circle (1.5pt);
 	\draw [fill=black] (9,1) circle (1.5pt);
	\draw [fill=black] (5,2) circle (1.5pt);
    \draw[color=black, anchor=east] (1,1) node {};
    \draw[color=black, anchor=north] (5,2) node {};
    \draw[color=black, anchor=west] (9,1) node {};
    \end{scriptsize}
\end{tikzpicture}
    \caption{A cellular cosheaf over a graph that is a cycle with three vertices.}
    \label{fig:cellcosheaf}
\end{figure}

For every cellular cosheaf $\Ff$ over $K$, we can define cosheaf homology by considering the following chain complex: Let
$$
    C_k^\Ff:=\bigoplus_{\sigma^k \in K} \Ff(\sigma).
$$

Using the sign function \cref{eq:sign}, we can define the linear \emph{boundary maps} $d_k^\Ff:C_k^\Ff \rightarrow C_{k-1}^\Ff$ acting on $\Ff(\sigma^k)$ via
$$
    v \mapsto \sum_{\tau^{k-1}\subset \sigma} [\sigma:\tau] \ \Ff_{\tau \subset \sigma} \, v.
$$
This gives indeed a chain complex, i.e., that $d^\Ff_{k-1} \circ d_k^\Ff=0$ (see \S6 of \citet{curry}). We can define homology and, after fixing inner products on the chain groups, Laplacians for this chain complex as we have done for Hilbert complexes. 

For cellular Hilbert cosheaves, an immediate choice of inner product on $C_k^\Ff$ is to interpret the direct sum $\bigoplus_{\sigma^k \in K} \Ff(\sigma)$ as a direct sum of Hilbert spaces.

\begin{example}\label{ex:constantcosheaf}
    One of the simplest cellular Hilbert cosheaves over a given simplicial complex $K$ is the one that assigns $\IK$ (with standard inner product) to every face and the identity map to every subface relation. Call this the constant cosheaf $\underline \IK$. The cosheaf homology and Laplacians of the constant cosheaf are then precisely the simplicial homology and the Hodge Laplacians of $K$ as described in \cref{sec:motivation}.
\end{example}

Next, we move on to defining persistent homology and persistent Laplacians for cellular cosheaves. To do this, we consider pairs of cosheaves. In \citet{pers_sheaf}, different ways of studying \emph{persistent (co)sheaf homology} are presented. One way is starting with a fixed simplicial complex, and then considering evolving cosheaves over it with arbitrary maps between them (referred to as ``type A'' by \citet{pers_sheaf}).  Another is to fix a cosheaf over a largest simplicial complex and then pull it back to the subspaces that form a filtration (``type T''), or can combine them \citep[\S4.3]{pers_sheaf}.  This last approach is the one that we adopt here, however, we restrict to the special case where maps between the sheaves are inclusions.

\begin{definition}\textnormal{\citep[Inclusion of Cosheaves, adaptation of][\S3.1]{sheaves}}
    Consider two finite simplicial complexes $K \subset L$ and Hilbert cosheaves $\Ff$ over $K$ and $\Gg$ over $L$. We say $\Ff \subset \Gg$, if there are linear isometries $i_\sigma:\Ff(\sigma) \rightarrow \Gg(\sigma)$ for all $\sigma \in K$ such that for all $\tau \subset \sigma \in K$, we have $i_\tau \circ \Ff_{\tau \subset \sigma}=\Gg_{\tau \subset \sigma} \circ i_\tau$.
\end{definition}

It is readily verifiable that any pair of Hilbert cosheaves $\Ff \subset \Gg$ gives rise to a pair of Hilbert complexes $C_\bullet^\Ff \subset C_\bullet^\Gg$.  From this pair of Hilbert complexes, we can then define persistent homology and persistent Laplacians as in previous discussions.

In fact, the \emph{persistent cosheaf Laplacians} arising in this way can be any Hermitian positive semi-definite operator.

\begin{example}[Arbitrary Hermitian Positive Semi-Definite Matrix as a Persistent Sheaf Laplacian.]
    Let $A \in \IK^{n \times n}$ be an arbitrary Hermitian positive semi-definite matrix. We find a unitary matrix $U \in \IC^{n \times n}$ and a diagonal $D \in \R^{n \times n}$ such that $A=UDU^*$. Let
    $$
        B:=U\sqrt D U^*
    $$
    where $\sqrt D$ has the square roots of the diagonal entries of $D$ on its diagonal. Then $B^* B=A$.
    
    Now let us define the following sheaves: Let $K=L=$ \begin{tikzpicture}
        \node[circle, fill=black, minimum size=5pt, inner sep=0pt, label=above:$a$] (a) at (0,0) {};
        \node[circle, fill=black, minimum size=5pt, inner sep=0pt, label=above:$b$] (b) at (1,0) {};
        \draw[black,line width=2pt] (a) -- (b);
    \end{tikzpicture}
    and let $\Ff=\Gg$ be given by $\Ff(b)=\Ff(ab)=\IK^n$, $\Ff(a)=0$ and $\Ff_{b \subset ab}=B$, expressed in the standard bases. Let all the inner products on $C_k^\Ff$ be the standard ones given by the identity matrix. Then $d_1^\Ff$ is given by $B$ and consequently 
    $$
    \Delta_{1}^{\Ff,\Gg}=\Delta_{1,\text{down}}^{\Ff,\Gg}
    $$
    has $A$ as matrix representation.
\end{example}

\subsubsection{De Rham Complexes} \label{subsub:derham}

The de Rham complex of smooth manifolds is a cochain complex \citep[see][\S17]{lee} which can be turned into a chain complex by relabeling indices. If the manifold (which may have boundary) is compact and orientable, an inner product can be defined on the chain groups of this complex \citep[see][\S1.2 and 2.6]{schwarz}. The presentation in \citet[\S3]{Hilbert} deals with the more general setting of defining Hilbert complexes from elliptic complexes on arbitrary Riemannian manifolds. However, in this example, we restrict our attention to a compact orientable Riemannian $n$-manifold $M$ with boundary $\partial M$ (which may be empty). For such manifolds, the de Rham cohomology is known to coincide with cohomology defined via other cohomology theories.

\begin{definition}[\textnormal{De Rham Cohomology, \citet[\S17]{lee}}]
    Let $M$ be a compact, smooth Riemannian manifold with or without boundary and $d_p:\Omega^p(M) \to \Omega^{p+1}(M)$ denote the exterior derivative for smooth $p$-forms. Then the \emph{de Rham Cohomology} groups of $M$ are given by
    $$
        H^p(M) \coloneqq \ker(d_p)\big/ \im(d_{p-1}). 
    $$
\end{definition}

\begin{prop}[\textnormal{De Rham Theorem, \cite{dRSamelson}}]
    Let $M$ be a compact, smooth Riemannian manifold with or without boundary. Then the de Rham cohomology groups are isomorphic to the singular cohomology groups of $M$.
\end{prop}
This result has first been shown for manifolds without boundary \citep{deRham}, however, we note that the homotopy equivalence of manifolds with boundary to their interior \citep[Theorem 9.26]{lee} combined with the homotopy invariance of de Rham cohomology \citep[Theorem 17.11]{lee} also establish it for general $M$ (with or without boundary). Some proofs, such as \citep{dRSamelson} work directly with manifolds with boundary. 

The spaces $\Omega^k(M)$ of smooth $k$-forms are not Hilbert spaces, because they are not complete with respect to their inner product
$$
    \langle \alpha ,\beta \rangle = \int_M \alpha \wedge \star \beta.
$$
Here $\star:\Omega^k(M) \rightarrow \Omega^{n-k}(M)$ is the \emph{Hodge star operator} \citep[see][\S16]{lee}.  Hence, we consider the completions of $\Omega^k(M)$, $L^2\Omega^k(M)$. The exterior derivatives can be extended to closed densely defined operators on these spaces in the following way \citep[see][\S4.2]{finderham}: First we should remark that \citet{finderham} work with bounded domains in $\R^n$, which differs from the setting of our compact manifold $M$. However, we may also restrict to this case of bounded domains, since we can find an atlas for $M$ with finitely many charts and then, by looking at a subordinate partition of unity, integrating over $M$ amounts to a finite sum of integrals over bounded domains in $\R^n$. Hence, showing that a differential form has a property such as being in $L^2$ or being weakly differentiable is equivalent to showing the corresponding properties locally.

We define the coboundary operator $\delta: \Omega^k(M) \rightarrow \Omega^{k-1}(M)$ by \citep[see][problems 16-22]{lee}
\begin{align}
    \delta \omega =(-1)^{nk+n+1} \star d \star \omega. \label{defeq:delta}
\end{align}
In order to extend $d$, let $\Omega_c^{k+1}(M)$ denote the space of smooth $k+1$-forms with compact support in the interior of $M$ and define the following. 
\begin{definition}\textnormal{\citep[Weak Exterior Derivative,][\S4.2]{finderham}}
     $\alpha \in L^2\, \Omega^k(M)$ has the \emph{weak exterior derivative} $\beta \in L^2 \, \Omega^{k+1}(M)$ if
$$
    \langle \alpha , \delta \eta \rangle_{L^2}=\langle \beta , \eta \rangle_{L^2} \qquad \forall\ \eta \in \Omega_c^{k+1}(M)
$$
and it has no weak derivative if there is no such $\beta$.
\end{definition}

Let $H\, \Omega^k(M)$ be the space of forms in $L^2 \, \Omega^k(M)$ with weak exterior derivative in $L^2 \, \Omega^{k+1} (M)$. Then it turns out that 

\begin{center}
    \begin{tikzcd}
    \cdots \arrow[r,"d^M_{k-2}"] & L^2 \, \Omega^{k-1}(M) \arrow[r,"d^M_{k-1}"] &L^2 \, \Omega^k(M) \arrow[r,"d^M_k"] &L^2 \, \Omega^{k+1}(M) \arrow[r,"d^M_{k+1}"] &\cdots
\end{tikzcd}
\end{center}
becomes a Hilbert complex if we set $\dom(d_k):=H\, \Omega^k(M)$ \citep[\S4.2]{finderham}. Note that inserting $N=M$ into \cref{lem:derhamispair}(d) yields that $\dom(d_k^M) \cap \dom(d_{k-1}^M)^*$ is dense in $L^2 \, \Omega^k(M)$.  Denoting the adjoint of $d_k^M$ by $(d^*)_{k+1}^M$ gives us the correct indexing in the Hilbert complexes that arise with $d^*$.

In order to define persistent objects, we consider a closed submanifold $N \subset M$ with boundary $\partial N$ such that $\dim(M)=\dim(N)$ and let $i:N\rightarrow M$ be the embedding. Such manifolds are always properly embedded \citep[Proposition 4.22 and Proposition 5.5]{lee} and are called \emph{regular domains in M} \citep[p.~120]{lee}. We have the following two de Rham cochain complexes with pull-backs $i^*$ between them \citep[see][Proposition~14.26]{lee}:
\begin{figure}[H]
    \centering
    \begin{tikzcd}
    \Omega^{k-1}(M) \arrow[r,"d^M"]\arrow[d,"i^*"] &\Omega^k(M) \arrow[r,"d^M"] \arrow[d,"i^*"] &\Omega^{k+1}(M) \arrow[d,"i^*"]\\
    \Omega^{k-1}(N) \arrow[r,"d^N"] &\Omega^k(N) \arrow[r,"d^N"] &\Omega^{k+1}(N)
    \end{tikzcd}
    \caption{De Rham complexes of two manifolds $N \subset M$}
    \label{fig:pairofhc}
\end{figure}
This diagram commutes. We are aiming for a pair of Hilbert complexes $P \subset Q$ (in the sense of \cref{def:inclusion}); the inclusion maps for such a pair should go from the ``smaller'' space to the ``larger'' one. Thus, the maps $i^*$ in \cref{fig:pairofhc} are oriented in the wrong direction for what we require. Since the spaces are not finite-dimensional, we cannot do the same construction as in \cref{ex:cohomology}. Nevertheless, it offers a useful intuition, which we illustrate in \cref{fig:pairofhc}. Since $\dim(N)=\dim(M)$, we have that for all $k$, the restriction $i^*$ extends to a well-defined bounded map 
\begin{align*}
    i^*:L^2 \, \Omega^k(M) &\rightarrow L^2 \, \Omega^k(N)\\
    [\alpha] &\mapsto [i^* \alpha].
\end{align*}
This map commutes with the exterior derivative in the sense that for all $\eta \in \dom(d^M)$, $i^* \, \eta \in \dom(d^N)$ and $d^N i^* \eta = i^* d^M \eta$. To see this, consider $\xi \in \Omega_c^k(N)$ and its extension by zero $\xi' \in \Omega_c^k(M)$: Since $\delta \xi'$ is the extension by zero of $\delta \xi \in \Omega_c^{k-1}(M)$, we have
$$
    \langle i^* d^M \eta, \xi \rangle_N = \langle d^M \eta, \xi' \rangle_M = \langle \eta, \delta \xi'\rangle_M = \langle i^* \eta , \delta \xi \rangle_N.
$$

Moreover, $i^*$ has an everywhere defined adjoint which we claim is the map
\begin{align*}
    e_k: L^2 \, \Omega^k(N) &\rightarrow L^2 \, \Omega^k(M)\\
    [\omega] &\mapsto [e_k (\omega)] \qquad \text{with}\\
    e_k(\omega) (x)&=\begin{cases}
        \omega (x) & x \in N\\
        0 & x \in M \setminus N.
    \end{cases}
\end{align*}
Notice that the extension maps $e_k$ are well-defined. Moreover, the following holds:
\begin{lemma}
    The extension maps $e_k$ are adjoint to $i^*$ for all $k$.
\end{lemma}
\begin{proof}
    Consider any $\alpha \in L^2 \Omega^k(M)$ and $\beta \in L^2 \Omega^k(N)$. Then
    $$
        \langle \alpha, e_k \beta \rangle_M = \int_M \alpha \wedge \star (e_k \beta) = \int_N i^*\alpha \wedge \beta + 0 = \langle i^* \alpha, \beta \rangle_N.
    $$
\end{proof}

Since $i^*$ commutes with the exterior derivative, our strategy is to show that their adjoints, $e$ and $d^*$, also commute. More precisely, we show that the conditions of \cref{def:inclusion} hold.

\begin{lemma} \label{lem:derhamispair}
The following are true:
    \begin{enumerate}
        \item[(a)] $\dom((d^*)_k^M)\cap \im(e_k)=e_k(\dom((d^*)_k^N))$;
        \item[(b)] $e_{k-1} \circ (d^*)_k^N=(d^*)_k^M \circ e_k$;
        \item[(c)] $\langle \alpha, \beta \rangle^N_{L^2}=\langle e_k \alpha ,e_k \beta \rangle^M_{L^2} \qquad \forall\  \alpha, \beta \in L^2 \, \Omega^k(N)$;
        \item[(d)] $\dom((d^*)_k^N) \cap \dom((d^*)_{k+1}^{N,M})^*$ is dense in $L^2 \, \Omega^k(N)$.
    \end{enumerate}
\end{lemma}

\begin{proof}
    For (a) we will show two directions. First, consider $e_k(\gamma) \in e_k(\dom((d^*)_k^N))$. We need to show that $e_k(\gamma) \in \dom((d^*)_k^M)$. Let $(d^*)_k^N \, \gamma = \psi$ and consider any $\eta \in \dom(d^M)$. We have just shown that $i^* \, \eta \in \dom(d^N)$ and $d^N i^* \eta = i^* d^M \eta$. Consequently,
    \begin{align}
        \langle e_{k-1} \psi , \eta \rangle_{L^2}^M = \langle \psi, i^* \eta \rangle_{L^2}^N = \langle \gamma, d^N i^* \eta \rangle_{L^2}^N = \langle e_k \gamma , d^M \eta\rangle_{L^2}^M. \label{eq:showsb}
    \end{align}
    Hence, by definition, $e_k(\gamma) \in \dom((d^*)_k^M)$ and $(d^*)_k^M e_k \gamma = e_{k-1} \psi$.
    
    Conversely, consider some $e_k (\gamma) \in \dom((d^*)_k^M)$ and let $\psi := (d^*)_k^M \, e_k (\gamma)$. We need to show that $\gamma \in \dom((d^*)_k^N)$. More precisely, we will show that $i^* \psi$ satisfies the defining property for $(d^*)_k^N \gamma$.
    
    We first show that $\psi$ is entirely supported in $N$. 
\begin{claim}
    $\psi = e_k i^* \psi$. \label{claim:psi}
\end{claim}
    \textit{Proof of claim:} For every smooth $\eta \in \Omega^k(M)$, we have
    $$
        \langle \psi,\eta \rangle_{L^2}^M = \langle (d^*)_k^M \, e_k (\gamma),\eta \rangle_{L^2}^M = \langle e_k \gamma, d^M\eta \rangle_{L^2}^M=\langle \gamma, i^* d^M\eta \rangle_{L^2}^N = \langle \gamma, d^N i^* \eta \rangle_{L^2}^N.
    $$
    In particular, for any $\eta \in \Omega^k(M)$ with $i^* \eta=0$, we have $\langle \psi, \eta \rangle_{L^2}^M=0$. If the claim were false, we would need $||\psi-e_k i^* \psi||>0$. By picking a chart which gives a positive contribution to $||\psi-e_k i^* \psi||$, we may consider the situation in Euclidean space, i.e., we work with a bounded subset $S \subset \R^n$. Using the slice criterion \citep[Theorem 5.8]{lee}, the boundary of $N$ corresponds to $\{x_1=0\}\cap S$ and $N$ corresponds to $\{x_1>0\}\cap S$. Then, by the dominated convergence theorem, 
    $$
    \lim_{t \rightarrow 0} \int_{\{x_1<-t \}\cap S} |\psi(x)|^2 d^nx = \int_{\{x_1<0\}\cap S} |\psi(x)|^2 d^nx>0.
    $$
    Hence, for some $s >0$, $\int_{\{x_1<-s \}\cap S} |\psi(x)|^2 d^nx=\varepsilon >0$. Since we can find a smooth function $f:S \rightarrow [0,1] $ such that $$
    f(x)=\begin{cases}
        0 & x_1>0\\
        1 & x_1<-s,
    \end{cases}
    $$
    we can construct an $\eta \in \Omega^k(M)$ with $i^* \eta=0$ such that $\langle \psi, \eta \rangle_{L_2}^M>0$, which proves the claim by contradiction: To do this, we smoothly approximate $\psi$ on $S$ by some $\varphi \in \Omega^k(S)$ such that $||\psi-\varphi||_S < \frac{\varepsilon}{3 \, ||\psi||_S}$. (Here, $||\cdot||_S$ means the norm of a function restricted to $S$.) \citet[\S5.3, Theorem 3]{evans} guarantees the existence of such $\varphi$. Picking $\eta:=f \cdot \varphi$ ensures that $i^* \eta=0$ and yields
    \begin{align}
        \langle \psi, \eta \rangle_{L_2}^M = \int_{\{ x_1< -s \} \cap S} \langle \eta(x),\psi(x) \rangle d^n x + \int_{\{0> x_1> -s \} \cap S} \langle \eta(x),\psi(x) \rangle d^n x \label{eq:pfderham}.
    \end{align}
    We can estimate the following using Cauchy--Schwarz:
    \begin{align*}
        \int_{\{ x_1< -s \} \cap S} \langle \eta(x),\psi(x) \rangle d^n x &= \int_{\{ x_1< -s \} \cap S} \langle \psi(x),\psi(x) \rangle d^n x + \int_{\{ x_1< -s \} \cap S} \langle \varphi (x) - \psi(x),\psi(x) \rangle d^n x\\
        &\ge \varepsilon - ||\psi||_S \, ||\varphi - \psi||_S = \frac{2 \varepsilon}{3} . \\
        \int_{\{0> x_1> -s \} \cap S} \langle \eta(x),\psi(x) \rangle d^n x&=\int_{\{0> x_1> -s \} \cap S} f(x) \langle \psi(x),\psi(x) \rangle d^n x \\
        &\qquad \qquad + \int_{\{0> x_1> -s \} \cap S} \langle f(x) (\varphi(x) - \psi(x)),\psi(x) \rangle d^n x\\
        &\ge 0 -  ||\psi||_S \, ||\varphi - \psi||_S = \frac{ - \varepsilon}{3}
    \end{align*}
    Substituting this approximation back into \cref{eq:pfderham} gives $\langle \psi, \eta \rangle_{L_2}^M \ge \frac{\varepsilon}{3}$, and consequently \cref{claim:psi} holds.

    Now, equipped with this result, we can show that $\gamma \in \dom((d^*)_k^N)$: Consider any $\eta \in \dom(d^N)$. The extension theorem \citep[p.~254]{evans} allows us to lift $\eta$ to $\dom(d^M)$. Concretely, it provides an element $\tilde \eta \in \dom (d^M)$ such that $i^* \tilde \eta = \eta$ almost everywhere on $N$. Consequently,
    \begin{align*}
        \langle i^* \psi, \eta \rangle_{L^2}^N&=  \langle i^* \psi, i^*\tilde\eta \rangle_{L^2}^N = \langle e_k i^* \psi, \tilde\eta \rangle_{L^2}^M = \langle \psi, \tilde\eta \rangle_{L^2}^M = \langle (d^*)^M_k e_k \gamma, \tilde\eta \rangle_{L^2}^M = \langle e_k \gamma, d^M\tilde\eta \rangle_{L^2}^M \\
        &= \langle \gamma, i^* d^M\tilde\eta \rangle_{L^2}^N = \langle \gamma, d^N i^*\tilde\eta \rangle_{L^2}^N = \langle \gamma, d^N \eta \rangle_{L^2}^N.
    \end{align*}
    Here $\langle \gamma, d^N i^*\tilde\eta \rangle_{L^2}^N = \langle \gamma, d^N \eta \rangle_{L^2}^N$, because $i^*\tilde \eta = \eta$ almost everywhere on $N$ and $d^N$ is defined via inner products and hence insensitive to changes on sets of measure zero. This proves (a). 
    
    Part (b) follows directly from \cref{eq:showsb}.
    
    Part (c) follows from the definition of $e_k$.
    
    For part (d), it suffices to show that $\Omega_c^k(N) \subset \dom((d^*)_k^N) \cap \dom((d^*)_{k+1}^{N,M})^*$. 
    The inclusion $\Omega_c^k(N) \subset \dom((d^*)_k^N)$ is apparent from the definition of $d_k^N$. Concretely, for smooth $\phi \in \Omega_c^k(N)$, we have $d^* \phi = \delta \phi$. For showing $\Omega_c^k(N) \subset \dom((d^*)_{k+1}^{N,M})^*$, consider $\phi \in \Omega_c^k(N)$. We can extend $\phi$ by zero to obtain $\tilde \phi \in \Omega_c^k(M)$.
    Then for any $\eta \in \dom((d^*)_{k+1}^M)$ satisfying $(d^*)_{k+1}^M\eta= e_k i^* (d^*)_{k+1}^M\eta$, the definition of $(d^*)_{k+1}^{N,M}$ implies
    \begin{align*}
        \langle d \tilde \phi, \eta \rangle_{L^2}^M &= \langle  \tilde \phi, (d^*)_{k+1}^M \eta \rangle_{L^2}^M = \langle  \tilde \phi, e_k i^* (d^*)_{k+1}^M\eta \rangle_{L^2}^M = \langle  i^* \tilde \phi, i^* (d^*)_{k+1}^M\eta \rangle_{L^2}^N = \langle \phi, i^* (d^*)_{k+1}^M \eta \rangle_{L^2}^N \\
        &= \langle \phi, (d^*)_{k+1}^{N,M} \eta \rangle_{L^2}^N.
    \end{align*}
    This shows $\phi \in \dom((d^*)_{k+1}^{N,M})^*$ with $((d^*)_{k+1}^{N,M})^* \phi = d \tilde \phi$ and allows us to conclude (d).
\end{proof}

We have established that \cref{fig:pairofmf} below indeed shows a pair of Hilbert complexes $C^N \subset C^M$ consistent with \cref{def:cc,def:inclusion}.

\begin{figure}[H]
    \centering
    \begin{tikzcd}
    L^2 \, \Omega^{k+1}(N) \arrow[r,"(d^*)_{k+1}^N"] \arrow[d,"e_{k+1}"] &L^2 \, \Omega^k(N) \arrow[r,"(d^*)_k^N"] \arrow[d,"e_k"] &L^2 \, \Omega^{k-1}(N) \arrow[d,"e_{k-1}"]\\
    L^2 \, \Omega^{k+1}(M) \arrow[r,"(d^*)_{k+1}^M"] &L^2 \, \Omega^k(M) \arrow[r,"(d^*)_k^M"] &L^2 \, \Omega^{k-1}(M)
    \end{tikzcd}
    \caption{Pair of De Rham--Hilbert complexes $C^N \subset C^M$  coming from two manifolds $N \subset M$.}
    \label{fig:pairofmf}
\end{figure}
Recall that for our construction we assumed $n:=\dim(N)=\dim(M):=m$. Indeed, without this assumption, there is no straightforward way to define the map $e$: Even for smooth functions in Euclidean space, $e$ would need to be a map $L^2(\R^n) \rightarrow L^2(\R^m)$ such that 
$$
\int_{\R^n} |f(x)|^2 \, d^n \,x= \int_{\R^m} |ef(x)|^2 \, d^m \,x.
$$
For this to be possible, $ef$ must be supported on a set of positive measure, thus, outside of $\R^n$. But then $e$ would not be the adjoint of $i^*$ and \cref{lem:derhamispair} would fail to hold.

The workaround to still be able to work with submanifolds of smaller dimension is to slightly ``inflate'' them to a manifold with boundary that has the same dimension as the surrounding manifold. For example, instead of looking at a circle in some bounded subset of $\R^2$, we can consider a thin ring. The tubular neighborhood theorem \citep[see][p.~346]{spivak} ensures that such an inflation is possible for compact submanifolds $N$ as long as $N$ lies within the interior of $M$, i.e., $\partial M \cap N = \varnothing$.

Hilbert complexes arising from manifolds in the manner described above have a \emph{compactness property} \citep{finderham}: $\dom(d^*)^M_k \cap \dom d_k^M$ not only lies densely in $L^2 \Omega^k(M)$, but the inclusion is a compact operator \citep[see][\S4.2]{finderham}. Therefore, the following result applies to such Hilbert complexes.

\begin{prop}\label{prop:compact embedding}
    Consider three Hilbert spaces $U,V,W$ and closed, densely defined operators $A:U \rightarrow V$ and $B:V \rightarrow W$. If the embedding from $\dom A^*\cap \dom B$ to $V$ is compact, then
    \begin{enumerate}
        \item[(a)] $\sigma_e(AA^*+B^*B)=\varnothing$;
        \item[(b)] The only potential essential eigenvalue of $AA^*$ and of $B^*B$ is zero: $\sigma_e(AA^*), \sigma_e(B^*B) \subset \{0\}$.
    \end{enumerate}
\end{prop}
\begin{proof}
    Part (a) follows from Corollary 1.21 of \citet{schroedingerop}. We return to the proof of (b) after establishing a key result further on in \cref{subsec:spectraldecomp}.
\end{proof}

\subsubsection{Examples of Spectra of Persistent Chain Laplacians} 

We now give examples of spectra of persistent chain Laplacians, focusing on the essential spectrum, as the discrete part is well understood from the perspective of linear algebra.

First, they can consist of isolated points: for example, multiplication by some constant $\lambda \in \R$ is self-adjoint and hence, if all Hilbert spaces are the same infinite-dimensional Hilbert space $\Hh$ and maps are multiplication by $\lambda$, the Laplacians will have infinite-dimensional $\lambda^2$ eigenspaces and therefore $\lambda^2$ will be an essential eigenvalue.

The following more interesting examples show that the essential spectrum can contain points that are not eigenvalues. In these examples, our operators are bounded and defined everywhere. Recall from previous discussions in \cref{sub:chainlap} that in this case, the self-adjoint operators induced from quadratic forms coincide with the sum and composition of operators. In each example, we are given the spaces $C_k^1$ (short for $C_k^{P_1}$), $C_k^2$ and $C_{k+1}^{2,3}$ and the map $d_{k+1}^{2,3}$. \Cref{fig:exspectra} shows one possibility of how to construct three Hilbert complexes $P_1 \subset P_2 \subset P_3$ from our given spaces and map: Specifically, we take $C_{k+1}^1=C_{k+1}^2=0$, $C_{k+1}^3=C_{k+1}^{2,3}$, $C_k^3=C_k^2$, and $d_{k+1}^3=d_{k+1}^{2,3}$. Then by definition, $C_{k+1}^{1,3}=\{v \in C_{k+1}^{2,3}: d_{k+1}^{2,3} v \in C_k^1 \}$ and $d_{k+1}^{1,3}=d_{k+1}^{2,3}|_{C_{k+1}^{1,3}}$. 
\begin{figure}[H]
    \centering
    \begin{tikzcd}[column sep = +5em]
        C_{k+1}^1=\{0\}  \arrow[r,"0"] \arrow[d,hook] & C_k^1 \arrow[d,hook] \\
        C_{k+1}^2=\{0\}  \arrow[r,"0"] \arrow[d,hook] & C_k^2 \arrow[d,hook]\\
        C_{k+1}^3=C_{k+1}^{2,3} \arrow[r,"d_{k+1}^3=d_{k+1}^{2,3}"] & C_k^3=C_k^2
    \end{tikzcd}
    \caption{Hilbert complexes $P_1\subset P_2 \subset P_3$ that realize given spaces $C_k^1,\, C_k^2$ and $C_{k+1}^{2,3}$ and the map $d_{k+1}^{2,3}$.}
    \label{fig:exspectra}
\end{figure}

\begin{example}\label{ex:espec}
    Consider the Hilbert spaces 
    $$
    C_k^2=C_{k+1}^{2,3}=\ell^2, \qquad C_k^1=\left\{\left(r,\frac r 2,\frac r 3,\dots \right): r \in \R \setminus \{0\}\right\} \subset \ell^2,
    $$
    and let 
    $$
    d_{k+1}^{2,3}: (a_n)_{n \in \N} \mapsto \left(\frac {a_n} n \right)_{n \in \N}.
    $$
    Since $(1)_{n \in \N} \not \in \ell^2$, we have $C_{k+1}^{1,3}=\{0\}$ and therefore, $\Delta_{+,k}^{1,3}=0$.
    
    On the other hand, since $d_{k+1}^{2,3}$ is self-adjoint,
    $$
        \Delta_{+,k}^{2,3}: (a_n)_{n \in \N} \mapsto \left(\frac {a_n} {n^2} \right)_{n \in \N}
    $$
    has precisely the eigenvalues $\frac{1}{n^2}$ for $n\in \N$. Thus, $\Delta_{+,k}^{1,3}$ has an eigenvalue 0 that is smaller than all the eigenvalues of $\Delta_{+,k}^{2,3}$.
    
    However,
    $0 \in \sigma_e(\Delta_{+,k}^{2,3})$ by \cref{rem:spectrum} since 0 is an accumulation point of eigenvalues. In fact, we have $\sigma_e(\Delta_{+,k}^{2,3})=\{0\}$, which follows from $\sigma(\Delta_{+,k}^{2,3})=\{0,\frac 1 {n^2}: n \in \N\},
    $
    and the fact that for any $\lambda \in \R \setminus \{0,\frac 1 {n^2}: n \in \N\}$, the inverse of $\Delta_{+,k}^{2,3}-\lambda $ is the operator that maps
    $$
        (b_n)_n \mapsto \left(\frac{b_n}{\frac{1}{n^2}-\lambda}\right)_n,
    $$
    which is bounded.
\end{example}

The next example shows that the essential spectrum can also be a bounded continuum.

\begin{example} \label{ex:espec2}
    Consider the Hilbert spaces 
$$
        C_{k+1}^{2,3} =C_{k}^2=L^2([0,1]), \qquad
        C_{k}^1=\{c:c \in \R \} \subset C_{k}^2
$$
and let $d_{k+1}^{2,3}: f \mapsto xf.$ Then $C_{k+1}^{1,3}=\{0\}$, since $\frac{1}{x} \not \in L^2([0,1]),$ so $\Delta_{+,k}^{1,3}=0$. On the other hand, $d_{k+1}^{2,3}$ is self-adjoint, so $\Delta_{+,k}^{2,3}$ is multiplication by $x^2$, which does not have any eigenfunctions, thus the point spectrum of $\Delta_{+,k}^{2,3}$ is empty.
    
    Hence, here, the spectrum agrees with the essential spectrum. To find the (essential) spectrum, observe that for $\lambda \in \R$, the inverse of $\Delta_{+,k}^{2,3}-\lambda$ (if it exists) is the operator mapping 
    $$
    g \mapsto \frac g {x^2- \lambda}.
    $$
    This map exists and is bounded precisely for $\lambda >1$ and for $\lambda <0$.\\
    Hence, $\sigma(\Delta_{+,k}^{2,3})=\sigma_e(\Delta_{+,k}^{2,3})=[0,1]$.
\end{example}

\section{Spectral Properties: Monotonicity and Stability} \label{sec:stab}

We now present the main contributions of this paper, which establish crucial properties of the spectra of the persistent chain Laplacians which we defined in \Cref{sec:laplacians}, namely \emph{monotonicity} and \emph{stability}. These properties are essential guarantees for meaningfulness and faithfulness of the spectra of persistent chain Laplacians to serve as general, unifying geometric invariants for underlying objects from which they are derived, such as manifolds, graphs, simplicial complexes or cellular sheaves.  

We begin in \cref{subsec:distances}, by first formally defining the concepts of stability and monotonicity, before we prove monotonicity of the spectral functions of up and down persistent chain Laplacians further on in \cref{subsec:udmon}. We then exploit the infinite-dimensional version of the fact that the positive eigenvalues of $MM^\top$ and $M^\top M$ agree for any matrix $M$ to be able to relate the spectra of up and down (persistent) chain Laplacians with those of full ones in \cref{subsec:spectraldecomp}. Finally, in \cref{subsec:fullmon}, we show that stability is a direct consequence of monotonicity and we moot monotonicity of full Laplacians.

\subsection{Filtrations of Hilbert Complexes and Interleaving Distances} \label{subsec:distances}

In \cref{sec:perslap}, we studied pairs of Hilbert complexes, while in persistent homology, we study homology with respect to \emph{filtrations}, that is, chains of nested objects; for instance, in \cref{sec:persmot}, we considered nested simplicial complexes. We thus now consider filtrations of Hilbert complexes to bring our previous discussions to the level of persistence to then study monotonicity and stability. The setting of filtrations lends a natural distance for our main results.

\begin{definition}[Filtration of Hilbert Complexes]
A \textit{filtration of Hilbert complexes} is a collection of Hilbert complexes $\Pp =\{P_t\}_{t \in \mathbb{R}}$ such that $P_s \subset P_t$ for $s \leq t$, in the sense of \cref{def:cc}.
\end{definition}

The persistent Laplacians are defined for any pair $P_s \subset P_t$ (as long as $\dom(d_k^{P_s})\cap \dom\left((d_{k+1}^{P_s,P_t})^*\right)$ is dense in $C_k^{P_s}$). To simplify notation, we write $\Delta^{s,t}_k$ for $\Delta^{P_s, P_t}_k$ which refers to the $k$th persistent Laplacian of this pair.
 
In order for the spectral information to be geometrically meaningful, we require \emph{stability}. As in the case of classical persistence discussed in \cref{sec:motivation}, here, stability means that a bounded perturbation of the filtration results in a bounded perturbation of the spectra. To make this precise, we require notions of distances for both objects involved, that is, we require a distance for filtrations and a distance for spectra.

For filtrations, the standard distance \citep{chazal_proximity_2009,mww} which we state for Hilbert complexes below.

\begin{definition}[Filtration Interleaving Distance] \label{defn:intfil}
    Let $\Pp = \{P_{t}\}_{t \in \mathbb{R}}, \Qq = \{Q_{t}\}_{t \in \mathbb{R}}$ be two filtrations of Hilbert complexes.
    The \emph{interleaving distance} between $\Pp$ and $\Qq$ is
    \[
    d_{I}(\Pp, \Qq) := \text{inf}\{\varepsilon \geq 0 \rvert P_{t} \subset Q_{t+\varepsilon} \text{ and } Q_{t} \subset P_{t+\varepsilon} \text{ for all } t\} \in \R_{\ge 0} \cup \{\infty\}.
    \]
\end{definition}

This means that two filtrations are considered close if they can be interleaved as shown in \cref{fig:intfil} for small $\varepsilon$.

\begin{figure}[h]
    \centering
    \begin{tikzcd}[row sep = +2em, column sep = +3em]
        \dots  \arrow[r,hook] & P_t \arrow[r,hook] \arrow[rd,hook] & P_{t+\varepsilon} \arrow[r,hook] \arrow[rd,hook] & P_{t+2\varepsilon} \arrow 
        [r,hook] &\dots \qquad \Pp \\
        \dots  \arrow[r,hook] & Q_t \arrow[r,hook] \arrow[ru,hook] & Q_{t+\varepsilon} \arrow[ru,hook] \arrow[r,hook] & Q_{t+2\varepsilon} \arrow 
        [r,hook] &\dots \qquad \Qq                       
    \end{tikzcd}
    \caption{Interleaved filtrations $\Pp,\Qq$.} \label{fig:intfil}
\end{figure}

In order to define a distance for the spectra, we will first extract functions that summarize the persistent spectra from filtrations of Hilbert complexes.

\begin{definition}[Persistent Spectral Counting Functions] \label{def:persspecctng}
Given a filtration $\Pp$ of Hilbert complexes, the persistent spectral counting functions are:
\begin{align*}
    \lambda^{s,t}_{k,q}&:=\lambda_{k,q} (s,t):=\lambda_q(\Delta_{k}^{P_s,P_t}),\\
    \lambda^{s,t}_{+,k,q}&:=\lambda_{+,k,q} (s,t):=\lambda_q(\Delta_{+,k}^{P_s,P_t}),\\
    \lambda^{s,t}_{-,k,q}&:=\lambda_{-,k,q} (s,t):=\lambda_q(\Delta_{-,k}^{P_s,P_t}),
\end{align*}
with $\lambda_q(\cdot)$ as in \cref{eq:specctng}.
\end{definition}

In finite dimensions, these functions order and number the eigenvalues of the persistent Laplacians increasingly. In particular, the spectra can be reconstructed from their counting functions. In infinite dimensions, the counting functions, in general, only contain partial information about the spectra.

We can view the $\lambda_{\bullet,k,q}$ as functions $f : \mathbf{Int} \rightarrow \R_{\ge 0}$. Here, $\mathbf{Int}$ denotes the set of closed real intervals $[s,t]$ with $s \le t$. \citet{mww} have defined an interleaving distance for such functions as follows.

\begin{definition}\textnormal{\citep[Function Interleaving Distance,][]{mww}} \label{defn:inteval}
     Let $f,g : \mathbf{Int} \to \mathbb{R}_{\geq 0}$ be two functions. For $I=[t,s] \in \text{\textbf{Int}}$, let $I^{\varepsilon} = [t-\varepsilon,s+\varepsilon]$ be the $\varepsilon$-thickening of $I$. Then the interleaving distance between $f$ and $g$ is 
     \[
     d_{I}(f,g) := \text{inf}\{\varepsilon \geq 0 \rvert f(I^{\varepsilon}) \geq g(I) \text{ and } g(I^{\varepsilon}) \geq f(I) \text{ for all } I \in \text{\textbf{Int}}\}.
     \]
\end{definition}

This distance is especially meaningful for functions that increase with the size of the interval; we call such functions \emph{monotonous}.

\begin{definition}[Monotonicity]
    We say that $f:\mathbf{Int} \rightarrow \R_{\ge 0}$ is \emph{monotonous} (or \emph{monotonously increasing}) if
    $$
    f(I) \le f(J) \qquad\ \forall\ I \subset J.
    $$
\end{definition}

As our first main theoretical contribution, we will establish monotonicity for some persistent spectral counting functions (\cref{thm:downmonotonicity,thm:upmonotonicity}) and bound the distance between some spectral functions that arise from different filtrations by the interleaving distance of the filtrations.

\subsection{Up and Down Monotonicity} \label{subsec:udmon}
In order to show monotonicity of persistent spectral counting functions, it is enough to restrict our attention to three Hilbert complexes $P_1 \subset P_2 \subset P_3$. To simplify notation, we replace $P_i$ by $i$ in the indices of Hilbert spaces and boundary maps. 

We first establish the monotonicity of the spectral functions of the down-Laplacians.
\begin{theorem}[Down Monotonicity]\label{thm:downmonotonicity}
    We have 
    \begin{align}
        \lambda^{1,2}_{-,k,q}&=\lambda^{1,3}_{-,k,q}, \label{eq:downmon1}\\
        \lambda^{2,3}_{-,k,q}&\le \lambda^{1,3}_{-,k,q}.\label{eq:downmon2}
    \end{align}
\end{theorem}
\begin{proof}
Since the two persistent down-Laplacians agree, \eqref{eq:downmon1} holds. The inequality 
    \eqref{eq:downmon2} follows from \cref{thm:ari1.28}: Let $a$ and $b$ be the quadratic forms of $\Delta_{-,k}^{1,3}$ and $\Delta_{-,k}^{2,3}$ respectively, so
    \begin{align*}
        a[u]&=\langle u,\, \Delta_{-,k}^{1,3} u \rangle \qquad \forall\ u \in \dom(d_k^1)=d[a]\\
        b[v]&=\langle v,\, \Delta_{-,k}^{v,3} v \rangle \qquad \forall\ v \in \dom(d_k^2)=d[b].
    \end{align*}
    Note that they agree on $\dom(d_k^1)$: Let $\iota_k:C_k^1 \rightarrow C_k^2$ be the inclusion from \cref{def:inclusion}. Then for $u \in d[a]$ we have $\iota_k(u) \in d[b]$ and
    $$
        \langle \Delta_{-,k}^{1,3} u,\, u \rangle_1 = ||d_k^1 u||_1^2= ||d_k^2 \iota_k u||_2^2 = \langle \Delta_{-,k}^{2,3} \iota_k u,\, \iota_k u \rangle_2.
    $$
    Since $\iota_k$ is injective, it preserves linear independence.
    Consequently, the infimum in \cref{thm:ari1.28} for $\Delta_{-,k}^{2,3}$ is at most as big as the one for $\Delta_{-,k}^{1,3}$, which concludes the proof.
\end{proof}

Next, we turn our attention to the up-Laplacians:
\begin{figure}[H]
    \centering
\begin{tikzpicture}
    [x=1cm,y=1cm]
	\clip(0,-2.8) rectangle (6.5,5);
    \draw [-{Stealth[length=2mm]}] (1.6,4.1) --  node[above=1pt] {$d_{k+1}^{1}$} (5.5,4.1);
    \draw [-{Stealth[length=2mm]}] (1.6,-2) --  node[below=1pt] {$d_{k+1}^{3}$} (5.5,-2);
    \draw [-{Stealth[length=2mm]}] (3.1,2.6) -- node[left=10pt] {\scriptsize $d_{k+1}^{1,2}$}(5.5,4);
    \draw [-{Stealth[length=2mm]}] (4.05,0.5) -- (5.7,3.7);
    \draw [-{Stealth[length=2mm]}] (5.85,3.6) -- (4.2,0.4);
    \draw [{Stealth[length=2mm]}-] (3.2,2.4) -- node[below=5pt] {\scriptsize $\left(d_{k+1}^{1,2}\right)^*$}(5.6,3.8);
    \draw [right hook-{Stealth[length=2mm]}] (1,3.6) -- (1,1.4);
    \draw [right hook-{Stealth[length=2mm]}] (1,0.6) -- node[left=1pt]{$\iota_{k+1}^{2,3}$} (1,-1.6);
    \draw [right hook-{Stealth[length=2mm]}] (6,3.6) -- (6,1.4);
    \draw [right hook-{Stealth[length=2mm]}] (6,0.6) -- (6,-1.6);
    \draw [right hook-{Stealth[length=2mm]}] (2.1,2.2) -- (1.2,1.3);
    \draw [right hook-{Stealth[length=2mm]}] (2.9,1.9) -- node[left=1pt]{$i$} (3.7,0.4);
    \draw [right hook-{Stealth[length=2mm]}] (3.5,-0.2) -- (1.3,-1.8);
    \begin{scriptsize}
    \draw[color=black] (5.2,1.1) node {$\left(d_{k+1}^{1,3}\right)^*$};
    \draw[color=black] (4.3,1.8) node {$d_{k+1}^{1,3}$};
    \end{scriptsize}
    \draw[color=black] (1,4) node {$C_{k+1}^{1}$};
    \draw[color=black] (6,4) node {$C_k^{1}$};
    \draw[color=black] (6,1) node {$C_k^{2}$};
    \draw[color=black] (1,1) node {$C_{k+1}^{2}$};
    \draw[color=black] (6,-2) node {$C_k^{3}$};
    \draw[color=black] (1,-2) node {$C_{k+1}^{3}$};
    \draw[color=black] (2.7,2.3) node {$C_{k+1}^{1,2}$};
    \draw[color=black] (4,0) node {$C_{k+1}^{1,3}$};
\end{tikzpicture}
    \caption{Setting of the proof of \cref{eq:upmon1} of \cref{thm:upmonotonicity}.} 
    \label{fig:proof_up_1}
\end{figure}

\begin{theorem}[Up Monotonicity]\label{thm:upmonotonicity}
    We have 
    \begin{align}
        \lambda^{1,2}_{+,k,q}&\le \lambda^{1,3}_{+,k,q},\label{eq:upmon1}\\
        \lambda^{2,3}_{+,k,q}&\le \lambda^{1,3}_{+,k,q}.\label{eq:upmon2}
    \end{align}
\end{theorem}

\begin{proof}[Proof of \eqref{eq:upmon1}.] We first prove \cref{eq:upmon1}; 
    \cref{fig:proof_up_1} shows a diagram of the spaces and maps involved in the proof.
    In order to generalize the proof of Theorem 5.2 in \citet{mww} using \cref{thm:ari1.28}, we need to show that Claims \ref{claim:doms} and \ref{claim:norms} below hold. Denote by $\delta_{+,k}^{1,2}$ and $\delta_{+,k}^{1,3}$ the quadratic forms associated with $\Delta_{+,k}^{1,2}$ and $\Delta_{+,k}^{1,3}$, respectively.
    \begin{claim}
        Their domains satisfy $d[\delta^{1,2}_{+,k}]=\dom (d^{1,2}_{k+1})^*\supset d[\delta^{1,3}_{+,k}]=\dom (d^{1,3}_{k+1})^*$. \label{claim:doms}
    \end{claim}
    \begin{claim}
        For all $u \in \dom(d_{k+1}^{1,3})^*$, we have $\|(d^{1,2}_{k+1})^*u\|^2 \le \|(d^{1,3}_{k+1})^*u\|^2$. \label{claim:norms}
    \end{claim}
    Once we show these claims, we are done, because then by \cref{thm:ari1.28}, we have that
    \begin{align*}
        \lambda^{1,2}_{+,k,q}&= \inf_{\substack{u_1,\dots,u_q \subset \dom(d_{k+1}^{1,2})^* \\ \text{lin.~indep.}}} \ \sup_{0 \not= u \in \text{span}\{u_1,\dots,u_q\} } \ \frac{||(d_{k+1}^{1,2})^* \, u||^2}{||u||^2} \\ &\le \inf_{\substack{u_1,\dots,u_q \subset \dom(d_{k+1}^{1,3})^* \\ \text{lin.~indep.}}} \ \sup_{0 \not= u \in \text{span}\{u_1,\dots,u_q\} } \ \frac{||(d_{k+1}^{1,2})^* \, u||^2}{||u||^2}\\
        &\le \inf_{\substack{u_1,\dots,u_q \subset \dom(d_{k+1}^{1,3})^* \\ \text{lin.~indep.}}} \ \sup_{0 \not= u \in \text{span}\{u_1,\dots,u_q\} } \ \frac{||(d_{k+1}^{1,3})^* \, u||^2}{||u||^2} = \lambda^{1,3}_{+,k,q}.
    \end{align*}

    \textit{Proof of claims:}  We can restrict $\iota_{k+1}^{2,3}$ to $C_{k+1}^{1,2}$ which gives us the isometry $i$ that embeds $C_{k+1}^{1,2}$ into $C_{k+1}^{1,3}$. Since $i(C^{1,2}_{k+1}) \cong C_{k+1}^{1,2}$ is closed, \cref{lem:orthogonal_decomposition} yields an orthogonal decomposition,
    $$
    C^{1,3}_{k+1} = i(C^{1,2}_{k+1}) \oplus \big(i(C_{k+1}^{1,2})\big)^\perp.
    $$
    Let $\tilde p: C^{1,3}_{k+1} \to i(C^{1,2}_{k+1})$ be the orthogonal projection induced by this decomposition and $p=i^{-1}\circ \tilde p:C^{1,3}_{k+1} \to C^{1,2}_{k+1}$. Then $p=i^*$. Moreover, because $d^{1,2}_{k+1}$ and $d^{1,3}_{k+1}$ are both restrictions of $d^3_{k+1}$, for every $v \in \dom(d_{k+1}^{1,2})$ we have $i(v) \in \dom(d_{k+1}^{1,3})$ and $d^{1,2}_{k+1} (v)= d^{1,3}_{k+1} \circ i(v)$.

    Now consider any
    $u\in\dom (d^{1,3}_{k+1})^*$. We can decompose
    $$(d^{1,3}_{k+1})^*u=i(f)+g \in i(C^{1,2}_{k+1}) \oplus \big(i(C_{k+1}^{1,2})\big)^\perp$$ with $f=p \circ (d_{k+1}^{1,3})^* \, u  \in C^{1,2}_{k+1}$.
    
    Hence, for all $v\in \dom d^{1,2}_{k+1}$,
    $$
        \langle u, d_{k+1}^{1,2} v \rangle_k^1 = \langle u, d_{k+1}^{1,3} \circ i \, v \rangle_k^1 = \langle (d_{k+1}^{1,3})^* u, iv \rangle _{k+1}^{1,3} = \langle p \circ (d_{k+1}^{1,3})^* u, v \rangle _{k+1}^{1,2} = \langle f, v \rangle _{k+1}^{1,2}.
    $$
    Thus, $u \in \dom(d_{k+1}^{1,2})^*$ which shows \cref{claim:doms}.
    
    Moreover, $(d_{k+1}^{1,2})^* u = f= p \circ (d_{k+1}^{1,3})^* \, u $. Consequently, \cref{claim:norms} follows:
    $$
    ||(d_{k+1}^{1,3})^* \, u||^2 = ||i(f)||^2 + ||g||^2 = ||f||^2 + ||g||^2 \ge ||f||^2= ||(d_{k+1}^{1,2})^* \, u||^2.
    $$
\end{proof}

\begin{remark} \label{rem:fullmon1213}
    Our proof of \cref{thm:upmonotonicity} also applies to the spectral functions of the full persistent Laplacians---meaning $\lambda_{k,q}^{1,2} \le \lambda_{k,q}^{1,3}$---because both claims also hold for the full Laplacians: Let $\delta_k^{1,2}$ and $\delta_k^{1,3}$ be the quadratic forms of $\Delta_k^{1,2}$ and $\Delta_k^{1,3}$, respectively. Then from \cref{claim:doms} it follows that $$ d[\delta_k^{1,2}]=\dom (d_{k+1}^{1,2})^*\cap \dom(d_k^1) \supset \dom (d_{k+1}^{1,3})^*\cap \dom(d_k^1) = d[\delta_k^{1,3}].$$
    Moreover, \cref{claim:norms} implies that for all $u \in d[\delta_k^{1,3}]$,
    $$
    \delta_k^{1,2}[u]=||(d_{k+1}^{1,2})^*\, u||^2 + ||d_k^1 \, u||^2 \le ||(d_{k+1}^{1,3})^*\, u||^2 + ||d_k^1 \, u||^2 = \delta_k^{1,3}[u].
    $$
\end{remark}

Next we will show \cref{eq:upmon2}. \Cref{fig:proof_up_2} displays the maps and spaces involved. The following lays the foundations for our proof of \cref{eq:upmon2}.
\begin{figure}
    \centering
\begin{tikzpicture}[x=1cm,y=1cm]
	\clip(0,-2.8) rectangle (6.7,5);
    \draw [-{Stealth[length=2mm]}] (1.6,4) --  node[above=1pt] {$d_{k+1}^{1}$} (5.5,4);
    \draw [-{Stealth[length=2mm]}] (1.6,-2) --  node[below=1pt] {$d_{k+1}^{3}$} (5.5,-2);
    \draw [-{Stealth[length=2mm]}] (3.4,2.5) --  node[anchor=south east] {$d_{k+1}^{1,3}$} (5.7,3.8);
    \draw [-{Stealth[length=2mm]}] (3.5,0.3) --  node[anchor=south east] {$d_{k+1}^{2,3}$} (5.6,1);
    \draw [-{Stealth[length=2mm]}] (5.6,0.8) -- (3.5,0.1);
    \draw [{Stealth[length=2mm]}-] (3.5,2.3) -- (5.8,3.6);
    \draw [right hook-{Stealth[length=2mm]}] (1,3.6) -- (1,1.4);
    \draw [right hook-{Stealth[length=2mm]}] (1,0.6) -- (1,-1.6);
    \draw [right hook-{Stealth[length=2mm]}] (6,3.6) -- node[anchor=west] {$\iota_k^{1,2}$} (6,1.4);
    \draw [right hook-{Stealth[length=2mm]}] (6,0.6) -- (6,-1.6);
    \draw [right hook-{Stealth[length=2mm]}] (2.5,1.7) -- (1.2,-1.7);
    \draw [right hook-{Stealth[length=2mm]}] (3,1.6) -- (3,0.4);
    \draw [right hook-{Stealth[length=2mm]}] (2.8,-0.3) -- (1.3,-1.8);
    \draw[color=black] (5.2,2.5) node {$\big(d_{k+1}^{1,3}\big)^*$};
    \draw[color=black] (5,0) node {$\big(d_{k+1}^{2,3}\big)^*$};
    \draw[color=black] (1,4) node {$C_{k+1}^{1}$};
    \draw[color=black] (6,4) node {$C_k^{1}$};
    \draw[color=black] (6,1) node {$C_k^{2}$};
    \draw[color=black] (1,1) node {$C_{k+1}^{2}$};
    \draw[color=black] (6,-2) node {$C_k^{3}$};
    \draw[color=black] (1,-2) node {$C_{k+1}^{3}$};
    \draw[color=black] (3,2) node {$C_{k+1}^{1,3}$};
    \draw[color=black] (3,0) node {$C_{k+1}^{2,3}$};
\end{tikzpicture}
    \caption{Setting of proof of \cref{eq:upmon2}.} 
    \label{fig:proof_up_2}
\end{figure}
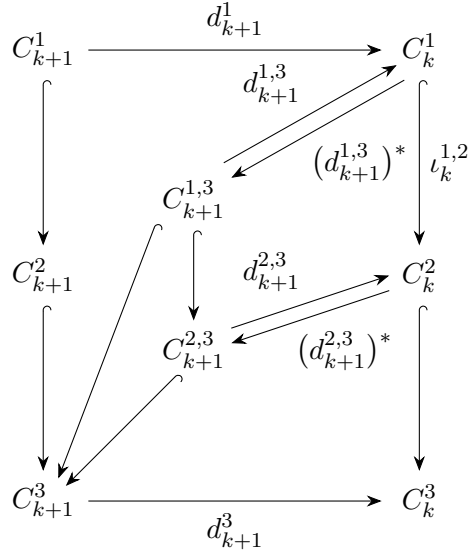

\begin{prop} \label{prop:imclosed}
Let $T:\Hh \rightarrow \Kk$ be a densely-defined, closed operator between Hilbert spaces:
\begin{enumerate}
    \item[(a)] $\im(T)$ is closed if and only if $\im(TT^*)$ is closed;
    \item[(b)] If $\im(T)$ is not closed, then $0 \in \sigma_e(TT^*)$;
    \item[(c)] If $\im(d_{k+1}^{2,3})$ is a closed subspace of $C_{k}^2$, then $\dim\big(\ker\big(d_{k+1}^{2,3}\big)^*\big) \ge \dim\big(\ker\big(d_{k+1}^{1,3}\big)^*\big)$.
\end{enumerate}
\end{prop}

\begin{proof}
($a$) is a known result; the proof given in \citet{closedrange} for bounded operators carries over to our setting: By \cref{lem:imperp}, we know that
$$
\overline{\im(T)}=\ker(T^*)^\perp=\ker\left(T T^*\right)^\perp=\overline{\im\left(T T^*\right)}.
$$
The middle equality $\ker(T^*)^\perp=\ker\left(T T^*\right)^\perp$ follows from an argument similar to \cref{lem:kerlap}. Now from $\im\left(TT^*\right)\subset\im(T)$, we can deduce that if $\im\left(T T^*\right)$ is closed, then
$$
\im(T) \supset \im\left(T T^*\right)=\overline{\im\left(T T^*\right)}=\overline{\im(T)}.
$$
Hence $\im(T)=\overline{\im(T)}$ is closed.

For the converse, it suffices to show that if $\im(T)$ is closed, then $\im\left(T T^*\right)\supset\im(T).$ Assume $y=Tx\in \im (T)$ for some $x\in \dom T$. Let
$$x=k+x' \in \ker(T) \oplus (\ker(T))^\perp.$$
Since $T x=T x'$, we may assume $x\in(\ker(T))^\perp$ without loss of generality. 

Note that $\im(T)$ being closed implies the closedness of $\im(T^*)$ by the closed range theorem \citep[\S5.3]{yosida}. Hence, by \cref{lem:imperp}, $(\ker(T))^\perp=\im(T)^*$. Thus, $x=T^*z$ for some $z\in \dom T^*$. Consequently, $y=T T^*z$ concludes the proof of (a).\\

To show (b): If $\im(T)$ is not closed, by (a), neither is $\im TT^*$. This yields the claim using \cref{rem:spectrum}(c).

To show (c), we will prove the following stronger statement, \cref{lem:imdensedone}.
\end{proof}

\begin{lemma} \label{lem:imdensedone}
    If $\iota_k^{1,2} (\im(d_{k+1}^{1,3}))= \iota_k^{1,2} (C_k^1) \, \cap \, \im(d_{k+1}^{2,3})$ is dense in $\iota_k^{1,2} (C_k^1) \cap \overline{\im(d_{k+1}^{2,3})}$, then $$\dim(\ker(d_{k+1}^{2,3})^*) \ge \dim(\ker(d_{k+1}^{1,3})^*).$$
\end{lemma}

\begin{proof}
    Let $\{v_1,\dots,v_n\} \subset \ker((d_{k+1}^{1,3})^*)$ be linearly independent. We are done if we can find $n$ linearly independent elements 
    in $\ker((d_{k+1}^{2,3})^*)$. By \cref{cor:decompV}, we have
    $$
        C_k^2=\overline{\im(d_{k+1}^{2,3})} \oplus \big(\im(d_{k+1}^{2,3})\big)^\perp.
    $$
    This yields the decompositions
        $\iota_k^{1,2} (v_i)=\lim_{j \rightarrow \infty} d_{k+1}^{2,3}(w_{i,j}) + u_i$ for all $i$,
    where\\ $w_{i,j} \in \dom(d_{k+1}^{2,3})\subset C_{k+1}^{2,3}$ and $u_i \perp \im(d_{k+1}^{2,3})$.
    
    In order to show that the $u_i$ are linearly independent, let us assume we have coefficients $k_i$ such that
    $$
        0=\sum_{i=1}^n k_i \, u_i= \sum_i k_i \iota_k^{1,2} ( v_i) - \sum_i \lim_{j} d_{k+1}^{2,3}(k_i w_{i,j}).
    $$
    Hence, by our assumption, and since isometries preserve closed subspaces (see \cref{closedstaysclosed}), 
    $$
    \iota_k^{1,2} (v):=\sum_{i} k_i \iota_k^{1,2} (v_i) \in \iota_k^{1,2} (C_k^1) \cap \overline{\im(d_{k+1}^{2,3})}=\iota_k^{1,2} \left( \overline{\im(d_{k+1}^{1,3})}\right).
    $$
    Thus $v \in \overline{\im(d_{k+1}^{1,3})}$ and we can express $v=\lim_{j\rightarrow \infty} d_{k+1}^{1,3} w_j$ for some $w_j \in \dom(d_{k+1}^{1,3})$.
    
    By construction, $v\in \ker\big(\big(d_{k+1}^{1,3}\big)^*\big)$, so we obtain
    $$
        \langle v, v\rangle_k^1=\Big\langle v,\, \lim_j {d_{k+1}^{1,3} w_j}\Big\rangle_k^1 = \lim_j \Big\langle v,\, {d_{k+1}^{1,3} w_j}\Big\rangle_k^1=\lim_j \Big\langle (d_{k+1}^{1,3})^* v,\, w_j\Big\rangle_{k+1}^{1,3}=0.
    $$
    The linear independence of the $v_i$ yields $k_i=0$ for all $i$.
    
    Thus the $u_i \in \big(\im(d_{k+1}^{2,3})\big)^\perp = \ker\big(\big(d_{k+1}^{2,3}\big)^*\big)$ (by \cref{lem:imperp}) are linearly independent.
\end{proof}
This has \cref{prop:imclosed}(c) as special case, because if $\im(d_{k+1}^{2,3})$ is closed, we have that\\
$\iota_k^{1,2} (\im(d_{k+1}^{1,3}))= \iota_k^{1,2} (C_k^1) \cap \overline{\im(d_{k+1}^{2,3})}$.

Equipped with \cref{prop:imclosed} we can now complete the proof of up monotonicity, \cref{thm:upmonotonicity}:

\begin{proof}[Proof of \cref{eq:upmon2}]
    By \cref{cor:decompV}, we may decompose
    \begin{align*}
        C_{k+1}^{1,3}&=\ker(d_{k+1}^{1,3})\oplus \overline{\im\big(d_{k+1}^{1,3}\big)^*},\\
        C_{k+1}^{2,3}&=\ker(d_{k+1}^{2,3})\oplus \overline{\im\big(d_{k+1}^{2,3}\big)^*}.
    \end{align*}
    Note that by construction, $\ker(d_{k+1}^{1,3}) \cong \ker(d_{k+1}^{2,3})$. Since replacing $C_{k+1}^{1,3}$ and $C_{k+1}^{2,3}$ with $\overline{\im(d_{k+1}^{1,3})^*}$ and $\overline{\im(d_{k+1}^{2,3})^*}$ respectively does not change $\Delta_{+,k}^{1,3}$ or $\Delta_{+,k}^{2,3}$, we may without loss of generality assume that $\ker(d_{k+1}^{2,3})=\{0\}$. 
    
    If $\im(d_{k+1}^{2,3})$ is not closed, \cref{prop:imclosed}(b) implies that there is nothing further to prove. Thus we may assume the opposite, from which we can deduce that
    $$
    \im(d_{k+1}^{1,3})=(\iota_k^{1,2})^{-1}\left(\im(d_{k+1}^{2,3})\cap \iota_k^{1,2}(C_k^1)\right)
    $$
    and by \cref{prop:imclosed}(a) $\im (\Delta^{1,3}_{+,k})$ are closed as well.
    \begin{claim} \label{claim:ess0}
        $0 \in \sigma_e(\Delta^{1,3}_{+,k}) \Rightarrow 0 \in \sigma_e(\Delta^{2,3}_{+,k}).$
    \end{claim}
    
    \textit{Proof of claim:} We know from \cref{rem:spectrum} that there are three possibilities for $0$ to be in $\sigma_e(\Delta^{1,3}_{+,k})$ corresponding to parts (c)(i), (c)(ii), and (c)(iii) i.e. $\im (\Delta^{1,3}_{+,k})$ is not closed, $0$ is an eigenvalue of infinite multiplicity, or $0$ is an accumulation point of eigenvalues. The case $(\Delta^{1,3}_{+,k})$ is not closed is already excluded by our assumption that $\im(d_{k+1}^{2,3})$ is closed.
    
    If $(\Delta^{1,3}_{+,k})$ has infinite-dimensional kernel, \cref{prop:imclosed}(c) implies $0 \in \sigma_e(\Delta^{2,3}_{+,k})$ (note that $\ker(\Delta^{1,3}_{+,k})=\ker\big(d_{k+1}^{1,3}\big)^*$).
    
    Finally, if $0$ is an accumulation point of eigenvalues of $\Delta^{1,3}_{+,k}$, by \cref{lem:changeorder}, $0$ is also an accumulation point of eigenvalues of $ (d_{k+1}^{1,3})^* d_{k+1}^{1,3}$ and therefore it is in the essential spectrum. Hence, \cref{thm:downmonotonicity} applied to $(d_{k+1}^{1,3})^* d_{k+1}^{1,3}$ and $(d_{k+1}^{2,3})^* d_{k+1}^{2,3}$ shows that $0 \in \sigma_e((d_{k+1}^{2,3})^* d_{k+1}^{2,3})$. Again by \cref{rem:spectrum}, there are three cases to consider. Since we have assumed $\ker(d_{k+1}^{2,3})=\{0\}$, the kernel of $(d_{k+1}^{2,3})^* d_{k+1}^{2,3}$ cannot be infinite dimensional. If $\im((d_{k+1}^{2,3})^* d_{k+1}^{2,3})$ were not closed, by \cref{prop:imclosed}(a),  $\im((d_{k+1}^{2,3})^*)$ would also not be closed. By the closed range theorem, this would imply that $\im(d_{k+1}^{2,3})$ would also not be closed, contradicting our assumption. Thus, there must be a sequence of positive eigenvalues of $(d_{k+1}^{2,3})^* d_{k+1}^{2,3}$ converging to zero. By \cref{lem:changeorder}, they are also eigenvalues of $\Delta_{+,k}^{2,3}$ and \cref{claim:ess0} follows.
    
    \begin{claim} \label{claim:m_bigeq}
    The infima of the essential spectra satisfy $m:=m(\Delta^{1,3}_{+,k})\ge m(\Delta^{2,3}_{+,k})$.
    \end{claim}
    We have just shown this in the case that $m(\Delta^{1,3}_{+,k})=0$. Now if $m(\Delta^{1,3}_{+,k})>0$, we know by \cref{lem:changeorder} that $m \in \sigma_e\big(\big(d_{k+1}^{1,3}\big)^*\, d_{k+1}^{1,3}\big)$. Then \cref{thm:downmonotonicity} implies that there is a $0\le \lambda \le m$ with $\lambda \in \sigma_e\big(\big(d_{k+1}^{2,3}\big)^*\, d_{k+1}^{2,3}\big)$.
    
    If $\lambda>0$, we can apply \cref{lem:changeorder} again to obtain $\lambda \in \sigma_e(\Delta^{2,3}_{+,k})$ and hence $m(\Delta^{2,3}_{+,k})\le \lambda \le m$. If $\lambda=0$, this implies $0 \in \sigma_e(\Delta^{2,3}_{+,k})$ as we have seen in the proof of \cref{claim:ess0}. This completes the proof of \cref{claim:m_bigeq}.
    
It remains to show the following:

\begin{claim}
If for $r\ge 0$, $\Delta_{+,k}^{1,3}$ has $n$ eigenvalues $\le r$, then either $\Delta_{+,k}^{2,3}$ has the same, or $m(\Delta_{+,k}^{2,3})\le r$. Here, we count eigenvalues with their multiplicities.
\end{claim}

Note that \cref{prop:imclosed}(c) yields that $\Delta_{+,k}^{2,3}$ has at least as many zero eigenvalues as $\Delta_{+,k}^{1,3}$. Let $n'$ be the number of positive eigenvalues (with multiplicity) of $\Delta_{+,k}^{1,3}$ that are at most $r$. By \cref{lem:changeorder}, all of these are also positive eigenvalues of $\big(d_{k+1}^{1,3}\big)^*\, d_{k+1}^{1,3}$. \Cref{thm:downmonotonicity} implies that $(d_{k+1}^{2,3})^*\, d_{k+1}^{2,3}$ has at least $n'$ eigenvalues $\le r$ or it has an essential eigenvalue $\le r$. In the first case, since these eigenvalues are by construction positive ($\ker(d_{k+1}^{2,3})=\{0\}$), they are also eigenvalues of $\Delta_{+,k}^{2,3}$ by \cref{lem:changeorder}. In the second case, we find, as in the proof of \cref{claim:m_bigeq}, that $m(\Delta_{+,k}^{2,3})\le r$. This concludes the proof of the claim as well as the proof of up monotonicity.
\end{proof}

Note that in \cref{ex:espec}, monotonicity in the sense of \cref{thm:upmonotonicity} does not hold if we only consider point spectra. Properties such as monotonicity in our sense and \cref{thm:ari1.28} motivate the definition of the spectrum adopted in \cref{subsec:spectra}.

\subsection{Spectral Decomposition of Laplacians} \label{subsec:spectraldecomp}

Full Laplacians consist of two components, the up- and the down-Laplacian; this is true, even if strictly speaking they are not the sum of up- and down-Laplacians, as previously discussed in \cref{rek:all quadratic form induced}.  Nevertheless, there is a strong connection between the full Laplacian and its up and down components, as we shall see. This configuration is special, since composing the up- and down-Laplacians in either order results in the zero map. In this section, we will show that for operators of this type, we can reconstruct the nonzero spectra of the full Laplacian from those of the up- and down-Laplacian.

\begin{theorem} \label{thm:split}
    Consider three Hilbert spaces $U,V,W$ and densely defined, closed linear operators $A:U \rightarrow V$ and $B: V \rightarrow W$ with $\im A\subset \ker(B) \subset \dom B$, and $\dom A^*\cap \dom B$ dense in $V$:

\begin{center}
    \begin{tikzcd}[column sep=+3em]
    U \arrow[r,"A",shift left] & V \arrow[r,"B",shift left] \arrow[l,"A^*",shift left] & W\arrow[l,"B^*",shift left]
\end{tikzcd}
\end{center}    
    
    Then $AA^*+B^*B|_{\ker(AA^*+B^*B)^\perp}$ is unitarily equivalent to 
    $$
    AA^*\oplus B^*B|_{\ker(AA^*\oplus B^*B)^\perp}=AA^*|_{\ker(AA^*)^\perp} \oplus B^*B|_{\ker(B^*B)^\perp}.
    $$
\end{theorem}

\begin{proof}
Consider the operator $T:V\to U\oplus W$ with $\dom(T)=\dom A^*\cap\dom B$ and $Tu=(A^*u,Bu)$. It is densely defined and closed, thus it has an adjoint $T^*:U\oplus W\to V$ defined by $T^*(u,v)=Au+B^*v$, by uniqueness of adjoints.
Let us define the quadratic form $t[u]:=\|Tu\|^2$ for all $u\in \dom A^*\cap\dom B$, which is densely defined, lower semi-bounded, and closed. Denote the induced self-adjoint operator by $T^*T$. Its domain is
$$
    \dom T^*T=\{u\in \dom A^*\cap\dom B:\exists\ f\in V \text{ s.t. }
    \forall\ v\in \dom A^*\cap\dom B, \langle Tu,Tv\rangle_{U\oplus W}=\langle f,v\rangle_{V}\},
$$
which is the same as $\dom (AA^*+B^*B)$. Hence, by the uniqueness of $f$, we have $T^*T=AA^*+B^*B$.

Similarly, we consider the quadratic form $t^*[u]=\|T^*u\|^2$ for all $u\in \dom A\oplus \dom B^*$. It induces the self-adjoint operator $TT^*$ with
\begin{align}\label{TT^*def 1}
    \dom TT^*=&\{u\in \dom A\oplus \dom B^*:\exists\ f\in U\oplus W
    \text{ s.t. }\\&\nonumber\forall\ v\in \dom A\oplus \dom B^*, \langle T^*u,T^*v\rangle_{V}=\langle f,v\rangle_{U\oplus W}\}.
\end{align}

\begin{claim}\label{cl:equivalence of domain}
    The domain $\dom TT^*$ given by \cref{TT^*def 1} is equivalent to
\begin{align}\label{TT^*def 2}
    \dom TT^*=&\{u\in \dom A:\exists\ g\in U \text{ s.t. }\forall\ v\in \dom A,\langle Au,Av\rangle_V=\langle g,v\rangle_U\}\nonumber\\
    \quad\oplus&\{u\in \dom B^*:\exists\ h\in W \text{ s.t. }\forall\ v\in \dom B^*,\langle B^*u,B^*v\rangle_V=\langle h,v\rangle_W\}.
\end{align}
\end{claim} 
\noindent
\textit{Proof of claim:} First, note that $\im A\subset \ker B$ implies $\im B^*\subset \ker(A^*) \subset \dom A^*$. Indeed, if we choose $v=B^*w\in \im B^*$ for some $w\in \dom B^*$, then for all $u\in \dom A$, we have 
$$
\langle v,Au \rangle =\langle B^*w,Au \rangle =\langle w,BAu \rangle=0.
$$
Therefore, we can choose $\xi=0$, so that $\langle \xi,u \rangle =\langle v,Au \rangle$.

Let us consider $u=(u_1,u_2),\, v=(v_1,v_2)$ for some $u_1,v_1\in \dom A$ and $u_2,v_2\in \dom B^*$. Then
\begin{align*}
    \langle T^*u, T^*v\rangle&=\langle Au_1,Av_1\rangle+\langle Au_1,B^*v_2\rangle+\langle B^*u_2,Av_1\rangle+\langle B^*u_2,B^*v_2\rangle\\
    &=\langle Au_1,Av_1\rangle+\langle B^*u_2,B^*v_2\rangle.
\end{align*}
If such an $f\in U\oplus W$ in $(\ref{TT^*def 1})$ exists, it can be written as $f=(g,h)$ for some $g\in U$ and $h\in W$. Then the identity in \cref{TT^*def 1} becomes
$$\langle Au_1,Av_1\rangle_V+\langle B^*u_2,B^*v_2\rangle_V=\langle f,v\rangle_{U\oplus W}=\langle g,v_1\rangle_U +\langle h,v_2\rangle_W.$$
If this holds true for all $v\in \dom A\oplus \dom B^*$, we can choose $v_2=0$. This implies the existence of $g\in U$ such that $\langle Au_1,Av_1\rangle_V=\langle g,v_1\rangle_U$ for all $v_1\in\dom A$. A similar argument for $v_1=0$ yields $(\ref{TT^*def 2})$. The inverse inclusion can be seen from choosing $f=(g,h)$. Hence, the claim is proven.\\

Together with uniqueness in \cref{a-induced A}, \cref{cl:equivalence of domain} implies that $TT^*=A^*A\oplus BB^*$, where $A^*A$ and $BB^*$ are defined via the quadratic forms $a[u]=\|Au\|^2$ and $b^*[u]=\|B^*u\|^2$, respectively. Note that by Theorem $\ref{TT^*=T^*T}$, there exist unitary isomorphisms $M:\ker(AA^*)^\perp \to \ker(A^*A)^\perp$ and $N:\ker(B^*B)^\perp \to \ker(BB^*)^\perp$ such that
\begin{align}
    M\left(\dom(AA^*)\cap \ker(AA^*)^\perp\right)&=\dom(A^*A)\cap\ker(A^*A)^\perp,\label{eq:Mdom}\\
     N\left(\dom(B^*B)\cap \ker(B^*B)^\perp\right)&=\dom(BB^*)\cap\ker(BB^*)^\perp\label{eq:Ndom},
\end{align}
and
\begin{align}
    AA^*|_{\ker(AA^*)^\perp}&=M^*A^*A|_{\ker(A^*A)^\perp}M,\label{eq:Meq}\\
    B^*B|_{\ker(B^*B)^\perp}&=N^*BB^*|_{\ker(BB^*)^\perp}N.\label{eq:Neq}
\end{align}
Define $L:\ker(AA^*)^\perp\oplus\ker(B^*B)^\perp\to\ker(A^*A)^\perp\oplus\ker(BB^*)^\perp$ to be $L:=M\oplus N$. $L$ is again unitary since both $M$ and $N$ are. Consider the operator $AA^*\oplus B^*B$ defined on $V\oplus V$, and restrict it to $\ker(AA^*\oplus B^*B)^\perp$. Note that
$$\ker(AA^*\oplus B^*B)^{\perp}=(\ker AA^*\oplus \ker B^*B)^{\perp_{V\oplus V}}=\ker(AA^*)^\perp\oplus\ker(B^*B)^\perp,$$
similarly we have $\ker(A^*A\oplus BB^*)^\perp=\ker(A^*A)^\perp\oplus\ker(BB^*)^\perp$. Then write
\begin{align*}
    T_L:&=AA^*\oplus B^*B|_{\ker(AA^*)^\perp\oplus\ker(B^*B)^\perp}=AA^*\oplus B^*B|_{\ker(AA^*\oplus B^*B)^\perp},\\
T_R:&=A^*A\oplus BB^*|_{\ker(A^*A)^\perp\oplus\ker(BB^*)^\perp}=A^*A\oplus BB^*|_{\ker(A^*A\oplus BB^*)^\perp}.
\end{align*}
Note that 
\begin{align*}
    \dom T_L&=\left(\dom (AA^*)\cap \ker(AA^*)^\perp\right)\oplus\left(\dom(B^*B)\cap\ker(B^*B)^\perp\right),\\
    \dom T_R&=\left(\dom (A^*A)\cap \ker(A^*A)^\perp\right)\oplus\left(\dom(BB^*)\cap\ker(BB^*)^\perp\right).
\end{align*}
Then applying \eqref{eq:Mdom} and \eqref{eq:Ndom}, we have $L(\dom T_L)=\dom T_R$. Note by \eqref{eq:Meq} and \eqref{eq:Neq}, we also have
$$T_RL(x,y)=(A^*A(Mx),BB^*(Ny))=(MAA^*x,NBB^*y)=LT_L(x,y).$$
Hence, $T_L=L^*T_RL$. Therefore, 
$$T_R=A^*A\oplus BB^*|_{\ker(A^*A\oplus BB^*)^\perp}=TT^*|_{\ker(TT^*)^\perp}$$
is unitarily equivalent to
\begin{equation}\label{last eq}
    T_L=AA^*\oplus B^*B|_{\ker(AA^*\oplus B^*B)^\perp}=AA^*|_{\ker(AA^*)^\perp} \oplus B^*B|_{\ker(B^*B)^\perp},
\end{equation}
where the last equation of \eqref{last eq} follows from
$$
\ker(AA^*)^\perp =\ker(A^*)^\perp \perp \ker(B)^\perp = \ker(B^*B)^\perp.
$$
Again by Theorem $\ref{TT^*=T^*T}$, $TT^*|_{\ker(TT^*)^\perp}$ is unitarily equivalent to $T^*T|_{(\ker{T^*T})^\perp}$. Since unitary equivalence is an equivalence relation, the proof is complete.
\end{proof}

Concretely, \cref{thm:split} implies that the nonzero spectra of full Laplacians and full persistent Laplacians are determined by the nonzero spectra of the up and down parts. More precisely, we have the following result.

\begin{cor}[Splitting of Spectra]\label{cor:split}
    Consider three Hilbert spaces $U,V,W$ and densely defined, closed linear operators $A:U \rightarrow V$ and $B: V \rightarrow W$ with $\im A\subset \ker(B) \subset \dom B$, and $\dom A^*\cap \dom B$ dense in $V$. Then the spectra and eigenvalues (including multiplicity) of $AA^*+B^*B|_{\ker(AA^*+B^*B)^\perp}$ coincide with those of $AA^*|_{\ker(AA^*)^\perp}\oplus B^*B|_{\ker(B^*B)^\perp}$, which in turn are the union of those of $AA^*|_{\overline{\im(A)}}$ and $B^*B|_{\overline{\im(B^*)}}$. Moreover,
    \begin{align}
        0 \in \sigma_e(AA^*+B^*B) \ &\Leftrightarrow 0 \in \sigma_e(AA^*|_{\overline{\im(A)}}) \cup \sigma_e (B^*B|_{\ker(A^*)}) \label{eq:splitse2}\\
        &\Leftrightarrow 0 \in \sigma_e(AA^*|_{\ker(B)}) \cup \sigma_e (B^*B|_{\overline{\im(B^*)}}). \label{eq:splitse1}
    \end{align}
\end{cor}
\begin{proof}
    By \cref{thm:magical} and \cref{lem:imperp}, $\ker(AA^*)^\perp=\ker(A^*)^\perp =\overline{\im(A)}$ and $\ker(B^*B)^\perp = \overline{\im(B^*)}$.
    \cref{thm:split} implies that the spectra of $AA^*+B^*B|_{\ker(AA^*+B^*B)^\perp}$ and $AA^*|_{\ker(AA^*)^\perp}\oplus B^*B|_{\ker(B^*B)^\perp}$ agree.
    
    For the second part, let $\lambda \in \R$ and $n_A, n_B, n_{AB} \in \N \cup \infty$ be the multiplicities of $\lambda$ as eigenvalues of of $AA^*|_{\overline{\im(A)}}$, $B^*B|_{\overline{\im(B^*)}}$, and $AA^*|_{\overline{\im(A)}}\oplus B^*B|_{\overline{\im(B^*)}}$, respectively.
    Every eigenvector $v$ of $AA^*|_{\overline{\im(A)}}$ yields an eigenvector $(v,0)$ of $AA^*|_{\overline{\im(A)}}\oplus B^*B|_{\overline{\im(B^*)}}$. The same holds for $B^*B$. Thus, $n_A+n_B \le n_{AB}$.
    
    Conversely, given $k$ linearly independent eigenvectors $\{(v_i,w_i): 0< i \le k\}$ of $AA^*|_{\overline{\im(A)}}\oplus B^*B|_{\overline{\im(B^*)}}$, all $v_i$ and $w_i$ satisfy $AA^* v_i=\lambda v_i$ and $B^*B w_i= \lambda w_i$. We have
    \begin{align*}
        k&=\dim(\text{span}\{(v_i,w_i): 0< i \le k\}) \\
        &\le \dim(\text{span}\{(v_i,0): 0< i \le k\}) + \dim(\text{span}\{(0,w_i): 0< i \le k\})=n_A+n_B.
    \end{align*}
    Hence, $n_{AB}\le n_A + n_B$ and therefore, $n_{AB}=n_A+n_B.$

    Next we want to show that for $\lambda \in \R \setminus \{0\}$,
    $$
    \lambda \in \sigma_e(AA^*|_{\overline{\im(A)}}\oplus B^*B|_{\overline{\im(B^*)}}) \ \Leftrightarrow \lambda \in \sigma_e(AA^*|_{\overline{\im(A)}}) \cup \sigma_e (B^*B|_{\overline{\im(B^*)}}).
    $$
    By \cref{rem:spectrum}, there are three ways in which $\lambda$ can be an essential eigenvalue. By the preceding discussion, $\lambda$ is eigenvalue of infinite multiplicity respectively an accumulation point of eigenvalues of $AA^*|_{\overline{\im(A)}}\oplus B^*B|_{\overline{\im(B^*)}}$ if and only if it is of $AA^*|_{\overline{\im(A)}}$ or $B^*B|_{\overline{\im(B^*)}}$. Moreover, 
    $$
    \im\left(AA^*|_{\overline{\im(A)}}\oplus B^*B|_{\overline{\im(B^*)}} - \lambda \right)\!=\im\left(AA^*|_{\overline{\im(A)}} - \lambda \right)\oplus\im\left( B^*B|_{\overline{\im(B^*)}} - \lambda \right)\! \subset \overline{\im(A)} \oplus \overline{\im(B*)}
    $$
    is closed if and only if both $\im\left(AA^*|_{\overline{\im(A)}} - \lambda \right)$ and $\im\left( B^*B|_{\overline{\im(B^*)}} - \lambda \right)$ are closed. This concludes the proof concerning the nonzero spectra.

    It remains to show \cref{eq:splitse2,eq:splitse1}. By symmetry of these statements we will only show \cref{eq:splitse2}. Note that by the proof of \cref{lem:kerlap}, 
    $$
    \ker(AA^*+B^*B)=\ker(AA^*)\cap \ker(B^*B) = \ker(B^*B|_{\ker(A^*)}).
    $$ 
    Hence, zero is an eigenvalue of infinite multiplicity of $AA^*+B^*B$ if and only if it is of $B^*B|_{\ker(A^*)}$.
    
    \Cref{rem:spectrum}(c) yields that for all self-adjoint operators $T$,
    $$
    0 \in \sigma_e(T) \Leftrightarrow \dim \ker (T)=\infty \ \text{  or  }\ 0 \in \sigma_e(T|_{\ker(T)^\perp}).
    $$
    This is because the nonzero eigenvalues and the images of $T$ and $T_{\ker(T)^\perp}$ agree.
    
    Thus, by \cref{thm:split}, \cref{rem:spectrum} (c) and the preceding discussion, we have
    \begin{align*}
        &\qquad\ \  0 \in \sigma_e(AA^*+B^*B)\\ &\Leftrightarrow
        \left\{
      \begin{array}{l}
        \dim \ker (AA^*+B^*B)=\infty, \ \text{or}\\
        0 \in \sigma_e(AA^*+B^*B|_{\ker(AA^*+B^*B)^\perp})
      \end{array}
    \right.  \\
        &\Leftrightarrow 
    \left\{
      \begin{array}{l}
        \dim \ker (B^*B|_{\ker(A^*)})=\infty,\  \text{or}\\
        0 \in \sigma_e(AA^*|_{\ker(AA^*)^\perp}\oplus B^*B|_{\ker(B^*B)^\perp})
      \end{array}
    \right. \\
    &\Leftrightarrow 
    \left\{
      \begin{array}{l}
        \dim \ker (B^*B|_{\ker(A^*)})=\infty,\  \text{or}\\
        \im(AA^*|_{\ker(AA^*)^\perp}\oplus B^*B|_{\ker(B^*B)^\perp}) \text{ not closed},\  \text{or}\\
        0 \text{ is accumulation point of eigenvalues of } AA^*|_{\ker(AA^*)^\perp}\oplus B^*B|_{\ker(B^*B)^\perp}
      \end{array}
    \right. \\
    &\Leftrightarrow 
    \left\{
      \begin{array}{l}
        \dim \ker (B^*B|_{\ker(A^*)})=\infty,\  \text{or}\\
        \im(AA^*|_{\ker(AA^*)^\perp}) \text{ or } \im(B^*B|_{\ker(B^*B)^\perp}) \text{ not closed},\  \text{or}\\
        0 \text{ is accumulation point of eigenvalues of } AA^*|_{\ker(AA^*)^\perp} \text{ or } B^*B|_{\ker(B^*B)^\perp}
      \end{array}
    \right. \\
    &\Leftrightarrow
        \left\{
      \begin{array}{l}
        0 \in \sigma_e(B^*B|_{\ker(A^*)}), \ \text{or}\\
        0 \in \sigma_e(AA^*|_{\overline{\im(A^*)}}).\\
      \end{array}
    \right.\tag{$\ast$}\label{eq:star-equiv}
    \end{align*}
    This argument verifies \cref{eq:splitse2}. To obtain \eqref{eq:star-equiv}, we used the fact that images and nonzero eigenvalues of $B^*B|_{\ker(A^*)}$ and $B^*B|_{\ker(B^*B)^\perp}$ agree, because by \cref{thm:magical}, $\ker(A^*)=\ker(B^*B)^\perp \oplus \ker(AA^*)\cap \ker(B^*B)$.
\end{proof}

We are now ready to prove \cref{prop:compact embedding}(b).
\begin{proof}[Proof of \cref{prop:compact embedding}(b)]
    Observe that the nonzero essential spectrum of a self-adjoint operator $T$ does not change if we restrict it to $\ker(T)^\perp$. This can be seen from \cref{rem:spectrum}(c): none of the properties characterizing essential eigenvalues (i.e., (c)(i), (c)(ii), and (c)(iii)) changes under this restriction. Hence, \cref{cor:split} implies that every nonzero essential eigenvalue of $AA^*$ or $B^*B$ would be one of $AA^*+B^*B$.
\end{proof}

In finite dimensional settings, \cref{cor:split} simplifies to the following result.

\begin{cor}[Splitting of Eigenvalues in Finite Dimensions]
    Consider three Hilbert spaces $U,V,W$ and bounded linear operators $A:U \rightarrow V$ and $B: V \rightarrow W$ such that $BA=0$ and let $\dim(V) < \infty.$ Then the nonzero eigenvalues of $AA^*+B^*B$ (with multiplicity) are precisely those of $AA^*$, together with those of $B^*B$.
\end{cor}
Note that by \cref{lem:bdopclosed}(b) and since $\dim(V)<\infty$, considering bounded operators that are everywhere defined is equivalent to considering densely defined closed operators. 

\subsection{Full Monotonicity and Stability} \label{subsec:fullmon}
The up and down monotonicity \cref{thm:downmonotonicity,thm:upmonotonicity} imply the stability of up- and down-Laplacians, which we will now show. This is an important result that guarantees the viability of the eigenvalues as invariants and geometric summaries of input data.

\begin{theorem}[Up and Down Stability] \label{thm:udstab}
    Let $\Pp$ and $\Qq$ be filtrations of Hilbert complexes and let 
    $$
    \lambda_{+,k,q}^\Pp ([s,t]):=\lambda_q(\Delta_{+,k}^{P_s,P_t})
    $$
    for the up eigenvalues of $\Pp$ a similar expression for the same for down eigenvalues and for $\Qq$. Then
    \begin{align}
        d_I(\lambda_{+,k,q}^\Pp,\lambda_{+,k,q}^\Qq) &\le d_I(\Pp,\Qq), \label{eq:upstab} \\
        d_I(\lambda_{-,k,q}^\Pp,\lambda_{-,k,q}^\Qq) &\le d_I(\Pp,\Qq). \label{eq:downstab}
    \end{align}
\end{theorem}

\begin{proof}
    This proof is the same as for \citep[Theorem 5.9]{mww}. We will only show \cref{eq:upstab}; the proof of \cref{eq:downstab} is the same. Assume $\Pp$ and $\Qq$ are $\varepsilon$-interleaved and consider an interval $[s,t]$. Using the symmetric roles of $\Pp$ and $\Qq$, \cref{eq:upstab} is true if
    $$
        \lambda_{+,k,q}^\Pp ([s,t]) \le \lambda_{+,k,q}^\Qq ([s-\varepsilon,t+\varepsilon]).
    $$
    This is an immediate consequence of up monotonicity: Since $Q_{s-\varepsilon} \subset P_s \subset P_t \subset Q_{t+\varepsilon}$, we have
    $$
        \lambda_q(\Delta_{+,k}^{P_s,P_t}) \le \lambda_q(\Delta_{+,k}^{Q_{s-\varepsilon},P_t}) \le \lambda_q(\Delta_{+,k}^{Q_{s-\varepsilon},Q_{t+\varepsilon}}).
    $$
\end{proof}

In \cref{rem:fullmon1213}, we saw that the full spectral functions $\lambda_{k,q}$ admit the weaker version of monotonicity, in the sense that they grow when the input interval grows to the right. However, they are not monotonous in general; we now study an example where the full spectral functions do not admit full monotonicity.

\begin{example}\label{ex:nomonotonicity}
\begin{figure}
    \centering
    \begin{tikzpicture}[x=1cm,y=1cm]
	\clip(-1,-1.5) rectangle (7,3.5);
    \draw [-{Stealth[length=2mm]}] (1,2) --  node[above=1pt] {$\begin{pmatrix}
0\\
1
\end{pmatrix}$} (2,2);
    \draw [-{Stealth[length=2mm]}] (4,2) --  node[above=1pt] {$\begin{pmatrix}
1&0
\end{pmatrix}$} (5,2);
    \draw [-{Stealth[length=2mm]}] (4,0.5) --  node[below=1pt] {$\begin{pmatrix}
1&0&-1
\end{pmatrix}$} (5,0.5);
    \draw [-{Stealth[length=2mm]}] (1,0.5) --  node[below=1pt] {$\begin{pmatrix}
    0&1\\
1&0\\
0&1
\end{pmatrix}$} (2,0.5);
    \draw [right hook-{Stealth[length=2mm]}] (0,1.7) -- node[left]{} (0,0.8);
    \draw [right hook-{Stealth[length=2mm]}] (3,1.7) -- node[right]{}(3,0.8);
    \draw [right hook-{Stealth[length=2mm]}] (6,1.7) -- node[right]{}(6,0.8);
    \draw[color=black] (0,2) node {$C_{k+1}^{1,3}=\R$};
    \draw[color=black] (3,2) node {$C_k^{1}=\R^2$};
    \draw[color=black] (6,2) node {$C_{k-1}^{1}=\R$};
    \draw[color=black] (6,.5) node {$C_{k-1}^{2}=\R$};
    \draw[color=black] (3,.5) node {$C_k^{2}=\R^3$};
    \draw[color=black] (0,.5) node {$C_{k+1}^{2,3}=\R^2$};
\end{tikzpicture}
    \caption{Example setting in which the spectral function is not monotonous. All inclusions $\R^n \rightarrow \R^N$ include $\R^n$ as the first $n$ coordinates of $\R^N$.} 
    \label{fig:nofullmon}
\end{figure}
     Consider \cref{fig:nofullmon}. Note that we can find three Hilbert complexes leading to this configuration as previously described in \cref{fig:exspectra}. In this situation we get $\Delta_{k}^{1,3}=I_2$ and $\Delta_{k}^{2,3}=\text{diag}(2,1,2)$. Thus, $\lambda_2(\Delta_{k}^{2,3}) > \lambda_2(\Delta_{k}^{1,3})$, contradicting full monotonicity.
\end{example}

Next, we give an example where even for $q=1$ and for Hilbert complexes arising from simplicial complexes, the setting commonly used in TDA, monotonicity does not generally hold.

\begin{example}\label{ex:nomonotonicity2}
\begin{figure}
    \centering
   \begin{tikzpicture}[x=0.7cm,y=0.7cm]
    \clip (-9,-3.3) rectangle (9,3.2);
    \node[circle, fill=skyblue, minimum size=5pt, inner sep=0pt, label=above left:$c$] (c1) at (-4,0) {};
    \node[circle, fill=skyblue, minimum size=5pt, inner sep=0pt, label=left:$a$] (a1) at (-8,0) {};
    \node[circle, fill=skyblue, minimum size=5pt, inner sep=0pt, label=right:$d$] (d1) at (-6,3) {};            
    \node[circle, fill=skyblue, minimum size=5pt, inner sep=0pt, label=right:$b$] (b1) at (-6,-3) {}; 
    \node[circle, fill=black, minimum size=5pt, inner sep=0pt, label=right:$c$] (c2) at (8,0) {};
    \node[circle, fill=black, minimum size=5pt, inner sep=0pt, label={[xshift=0.4cm, yshift=-0.1cm]$a$}] (a2) at (4,0) {};
    \node[circle, fill=black, minimum size=5pt, inner sep=0pt, label=right:$d$] (d2) at (6,3) {};
    \node[circle, fill=black, minimum size=5pt, inner sep=0pt, label=right:$b$] (b2) at (6,-3) {}; 
    \draw[skyblue,line width=2pt] (a1) -- (c1) -- (b1) -- (a1) -- (d1);
    \node[circle, fill=blue, minimum size=5pt, inner sep=0pt, label={[xshift=-0.4cm, yshift=-0.1cm]$c$}] (c) at (2,0) {};
    \node[circle, fill=blue, minimum size=5pt, inner sep=0pt, label={[xshift=0.4cm, yshift=-0.1cm]$a$}] (a) at (-2,0) {};
    \node[circle, fill=blue, minimum size=5pt, inner sep=0pt, label=right:$d$] (d) at (0,3) {};
    \node[circle, fill=blue, minimum size=5pt, inner sep=0pt, label=right:$b$] (b) at (0,-3) {};   
    \draw[blue,line width=2pt] (a) -- (c) -- (b) -- (a) -- (d) -- (c) ;
    \draw[black,line width=2pt] (a2) -- (c2) -- (b2) -- (a2) -- (d2) -- (c2) ;
    \begin{scope}[on background layer]
    \fill[blue!20] (a.center) -- (b.center) -- (c.center) -- cycle;
    \fill[black!20] (a2.center) -- (b2.center) -- (c2.center) -- cycle;
    \fill[black!20] (a2.center) -- (c2.center) -- (d2.center) -- cycle;
    \end{scope}
    \draw (-3,0) node{\Large $\subset$};
    \draw (3,0) node{\Large $\subset$};
\end{tikzpicture}\\ \vspace{1cm}
\begin{tikzcd}
    C_2^1=\{0\} \arrow[r,"0"] \arrow[d,hook] & C_1^1=\R^4 \arrow[r,"d_1^1"]\arrow[d,hook] & C_0^1=\R^4 \arrow[d,hook] \\
    C_2^2=\R \arrow[r,"d_2^2"] \arrow[d,hook] & C_1^2=\R^5 \arrow[r,"d_1^2"] \arrow[d,hook] & C_0^2=\R^4\arrow[d,hook] \\
    C_2^3=\R^2 \arrow[r,"d_2^3"]  & C_1^3=\R^5 \arrow[r,"d_1^3"]& C_0^3=\R^4
\end{tikzcd}
    \caption{Three simplicial complexes $\textcolor{skyblue}{X_1} \subset \textcolor{blue}{X_2} \subset X_3$ such that the induced spectral function is not monotonous together with their chain complexes.} 
    \label{fig:nofullmon2}
\end{figure}
    Consider the three simplicial complexes $\textcolor{skyblue}{X_1} \subset \textcolor{blue}{X_2} \subset X_3$ and their chain complexes in \cref{fig:nofullmon2}. With respect to the standard bases (in lexicographic order), the maps are
    \begin{align*}
        d_2^2&=\begin{pmatrix}
            1 \\ -1 \\ 0 \\ 1 \\ 0 
        \end{pmatrix} &d_2^3&=\begin{pmatrix}
            1 & 0\\ -1 & 1\\ 0 & -1\\ 1 &0\\ 0& 1 
        \end{pmatrix}\\
        d_1^1&=\begin{pmatrix}
            -1 & -1 & -1 & 0 \\ 1&0&0&-1 \\ 0&1&0&1 \\ 0&0&1&0 
        \end{pmatrix} 
        & d_2^1=d_3^1&=\begin{pmatrix}
            -1 & -1 & -1 & 0 &0\\ 1&0&0&-1&0 \\ 0&1&0&1&-1 \\ 0&0&1&0&1 
        \end{pmatrix}
    \end{align*}
    We find $C_2^{1,3}=C_2^2$ and $C_2^{2,3}=C_2^3$ and therefore $d_{2}^{1,3}=(1,-1,0,1)^\top$ and $d_{2}^{2,3}=d_2^3$. Using standard inner products, we can compute the persistent Laplacians:
    \begin{align*}
        \Delta_1^{1,3}=\begin{pmatrix}
            1&-1&0&1\\
            -1&1&0&-1\\
            0&0&0&0\\
            1&-1&0&1
        \end{pmatrix} + \begin{pmatrix}
            2&1&1&-1\\
            1&2&1&1\\
            1&1&2&0\\
            -1&1&0&2
        \end{pmatrix} &=\begin{pmatrix}
            3&0&1&0\\
            0&3&1&0\\
            1&1&2&0\\
            0&0&0&3
        \end{pmatrix}\\
        \Delta_1^{2,3}=\begin{pmatrix}
            1&-1&0&1&0\\
            -1&2&-1&-1&1\\
            0&-1&1&0&-1\\
            1&-1&0&1&0\\
            0&1&-1&0&1
        \end{pmatrix} + \begin{pmatrix}
            2&1&1&-1&0\\
            1&2&1&1&-1\\
            1&1&2&0&1\\
            -1&1&0&2&-1\\
            0&-1&1&-1&2
        \end{pmatrix}&=
        \begin{pmatrix}
            3&0&1&0&0\\
            0&4&0&0&0\\
            1&0&3&0&0\\
            0&0&0&3&-1\\
            0&0&0&-1&3
        \end{pmatrix}
    \end{align*}
    Their eigenvalues (as ordered lists) are $\sigma(\Delta_1^{1,3})=(1,3,3,4)$ and $\sigma(\Delta_1^{2,3})=(2,2,4,4,4)$.  We see that the smallest (and third smallest) eigenvalues do not satisfy monotonicity.
\end{example}

Since monotonicity and stability are closely related, and, in particular, the former is required for the latter, having no monotonicity implies that there is also no stability for full persistent Laplacians, as we see in the following example.

\begin{example} \label{ex:nofullstab}
    Let $R_1 \subset R_2 \subset R_3$ be the three Hilbert complexes from \cref{ex:nomonotonicity2} that do not satisfy monotonicity. Consider the filtrations $\Pp=\{P_t\}_{t \in \R}$ and $\Qq=\{Q_t\}_{t \in \R}$ with
    $$
        P_t := \begin{cases}
            R_1 & t\le 0\\
            R_3 & 0<t
        \end{cases} \qquad Q_t:= \begin{cases}
            R_1 & t\le 0\\
            R_2 & 0<t\le 1\\
            R_3 & 1<t
        \end{cases}
    $$
    These filtrations are 1-interleaved, so $d_I(\Pp,\Qq)=1$. On the other hand, we have $\lambda_{1,1}^\Qq{[\frac 1 2,\frac 3 2]}=2$, which is the smallest eigenvalue of $\Delta_1^{2,3}$ from \cref{ex:nomonotonicity2}. But $\lambda_{1,1}^\Pp{(I)}\le 1$ for all intervals $I$. Thus, $d_I(\lambda_{1,1}^\Pp,\lambda_{1,1}^\Qq)=\infty$ contradicting stability: We cannot bound the distance of the spectral functions from above by an expression that depends on the interleaving distance of the filtrations. 
\end{example}

A natural takeaway from these examples is that main emphasis should be placed on the up- and down-Laplacians, which are stable, instead of the full Laplacians. By \cref{cor:split}, their spectra contain the same information as those of full Laplacians.

\subsubsection{A Condition for Full Monotonicity}\label{subsub:confitionfm}

Although, in general, we do not have full monotonicity of spectral functions of persistent Laplacians, we will now present a condition under which full monotonicity indeed holds. 

Consider three Hilbert complexes $P_1 \subset P_2 \subset P_3$ like in \cref{subsec:udmon}. We first show that \cref{prop:imclosed}(c) is true for full Laplacians.

    \begin{lemma}\label{closed range for full}
        If $\im d^{2,3}_{k+1}$ is closed, $\dim\ker\Delta^{1,3}_k\leq\dim\ker\Delta^{2,3}_k$.
    \end{lemma}
    
    \begin{proof}
        Assume that there are linearly independent
        $v_1,v_2,\dots,v_n \in \ker(\Delta^{1,3}_k)$. By \cref{lem:kerlap}, $v_j\in \ker((d^{1,3}_{k+1})^*)\cap\ker(d^1_k)$, for $1\leq j\leq n$. By \cref{cor:decompV} we can write 
        $$
        \iota_k^{1,2}(v_j)=x_j + u_j \in \overline{\im(d_{k+1}^{2,3})} \oplus \big(\im(d_{k+1}^{2,3})\big)^\perp \qquad \forall\ j.
        $$
        In the proof of \cref{lem:imdensedone}, we showed that $\{u_1,\dots,u_n\}$ is linearly independent.
        
        Moreover, since $\im \, d_{k+1}^{2,3}$ is closed, we have $u_j=d^{2,3}_{k+1}w_j+\iota_k^{1,2}(v_j)$ for some $w_j\in \dom d^{2,3}_{k+1}$. Hence, $d^2_k\circ d^{2,3}_{k+1}=0$ and $\iota_{k-1}^{1,2} \circ d^1_k= d^2_k \circ \iota_k^{1,2}$ imply that $\{u_1,\dots,u_n\}\subset \ker(d^2_k)$. This shows $\{u_1,\dots,u_n\}\subset \ker \Delta_k^{2,3}$ and completes our proof.
\end{proof}

We now present our sufficient condition for full monotonicity.
\begin{prop}\label{prop:fmdecompose}
    If three Hilbert complexes $P_1 \subset P_2 \subset P_3$ satisfy 
    $$
        \iota_k^{1,2} \left(\overline{\im((d_k^1)^*)}\right) \subset \ker(d_{k+1}^{2,3})^*
    $$
    for some $k$, then they satisfy full monotonicity for this $k$:
    $$
    \lambda_{k,q}^{2,3} \le \lambda^{1,3}_{k,q} \qquad \forall\ q \in \N.
    $$
\end{prop}

Intuitively, this condition expresses that the up and down persistent Laplacian ``fully decompose.'' Hence, up and down monotonicity yield full monotonicity. 

\begin{proof}
    Let $p_1:C^2_k\to \ker(\Delta^{2,3}_k)$ be the orthogonal projection and $P:=p_1\left(\iota_k^{1,2} \left(\overline{\im((d_k^1)^*)}\right)\right)$. Note that $P$ is a subspace of $\ker(\Delta_k^{2,3})$. Let $V:=\overline{P}$ and $U$ be the orthogonal complement of $V$ in $\ker \Delta_k^{2,3}$. 
        By \cref{thm:magical}, we have 
        \begin{align*}
             C_k^1  = & \overbrace{\red{\overline{\im(d_{k+1}^{1,3})}\oplus \ker \Delta_k^{1,3}}}^{\ker(d_k^1)} \oplus \ \textcolor{blue}{\overline{\im((d_k^1)^*)}}\\
             & \, C_k^2 = \underbrace{\red{\overline{\im(d_{k+1}^{2,3})}\oplus U} \oplus \textcolor{blue}{V}}_{\ker(d_k^2)} \textcolor{blue}{\oplus \ \overline{\im((d_k^2)^*)}}.
        \end{align*}
        The gist of this proof is to apply up monotonicity to the red part of this decomposition and down monotonicity to the blue part; see \cref{fig:updownsplit}. To do so, we need to verify that $\iota_k^{1,2}$ of the red and blue components of $C_k^1$ are contained in the corresponding components of $C_k^2$. By construction, we know that $\iota_k^{1,2} \left(\overline{\im((d_k^1)^*)}\right) \subset \overline{\im((d_k^2)^*)} \oplus V$ and that $\iota_k^{1,2}\left(\overline{\im(d_{k+1}^{1,3})}\right) \subset \overline{\im(d_{k+1}^{2,3})}$ and we are left to show the following claim.
        
        \begin{claim} \label{claim:decomp}
            $\iota_k^{1,2} (\ker \Delta_k^{1,3}) \subset \overline{\im(d_{k+1}^{2,3})}\oplus U$.
        \end{claim}
        \textit{Proof of claim:} Since $\overline{\im(d_{k+1}^{2,3})}\oplus U$ is the orthogonal complement of $V \oplus \overline{\im((d_k^2)^*)}$, it suffices to show that $$\iota_k^{1,2} (\ker \Delta_k^{1,3}) \perp V \oplus \overline{\im((d_k^2)^*)}.$$
        We know that $\iota_k^{1,2} (\ker d_k^1) \subset \ker(d_k^2)$ and therefore $\iota_k^{1,2} (\ker \Delta_k^{1,3}) \perp \overline{\im((d_k^2)^*)}=\ker(d_k^2)^\perp$.
        
        Now take any $w \in \ker \Delta_k^{1,3}$ and $v \in P$. Then there exists some $u\in \overline{\im((d_k^1)^*)}$ such that $p_1(\iota^{1,2}_ku)=v$. Let $p_2:C^2_k\to \overline{\im((d_k^2)^*)}$ be another orthogonal projection and denote $x:=p_2(\iota^{1,2}_ku)\in \overline{\im((d_k^2)^*)}$. Since
        $$\iota_k^{1,2} \left(\overline{\im((d_k^1)^*)}\right) \subset \overline{\im((d_k^2)^*)} \oplus \ker(\Delta_k^{2,3}),$$
         we have $\iota_k^{1,2} u=v+x$. Choose sequences $(u_n)_{n \in \N} \subset \dom(d_k^1)^*$ and $(x_n)_{n \in \N} \subset \dom(d_k^2)^*$ such that $(d^1_k)^*u_n\to u$, and $(d^2_k)^*x_n\to x$. Then we have
        \begin{align*}
            \langle \iota_k^{1,2} w,v \rangle_k^2 &= \langle \iota_k^{1,2} w,\iota_k^{1,2}u-x \rangle_k^2 = \lim_{n \rightarrow \infty} (\langle \iota_k^{1,2} w,\iota_k^{1,2} (d_k^1)^*u_n \rangle_k^2 - \langle \iota_k^{1,2} w,(d_k^2)^*x_n \rangle_k^2)\\
            &=\lim_{n \rightarrow \infty} (\langle w,(d_k^1)^*u_n \rangle_k^1 - \langle d_k^2 \, \iota_k^{1,2} w,x_n \rangle_{k-1}^2)=\lim_{n \rightarrow \infty} (\langle d_k^1 w,u_n \rangle_{k-1}^1 - \langle \iota_{k-1}^{1,2} \, d_k^1 \, w,x_n \rangle_{k-1}^2)\\
            &=0+0.
        \end{align*}
        Since $P$ is dense in $V$, by continuity of the inner product, we have $\iota_k^{1,2} (\ker \Delta_k^{1,3}) \perp V$ and thus \cref{claim:decomp} is  verified.\\
        
        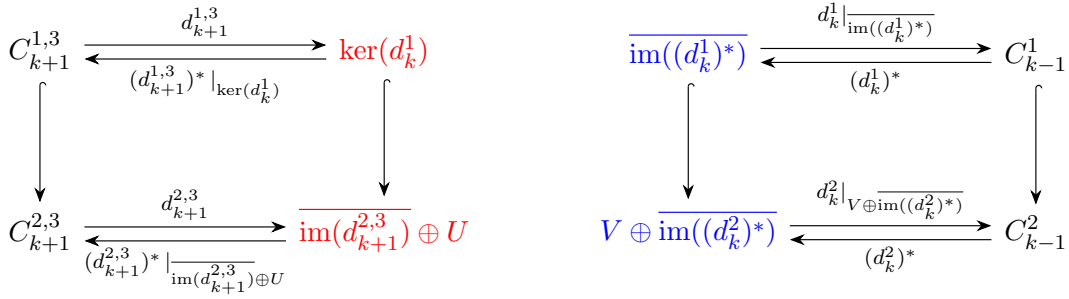
\begin{figure}
            \centering
                \begin{subfigure}[t]{0.4\textwidth}
                    \begin{tikzcd}[row sep = +4em, column sep = +7em]
                        C_{k+1}^{1,3} \arrow[r,"d_{k+1}^{1,3}",shift left] \arrow[d,hook] & \red{\ker(d_k^1)} \arrow[l,shift left,"(d_{k+1}^{1,3})^* \, |_{\ker(d_k^1)}"] \arrow[d,hook] \\
                        C_{k+1}^{2,3} \arrow[r,"d_{k+1}^{2,3}",shift left]  & \red{\overline{\im(d_{k+1}^{2,3})} \oplus U} \arrow[l,shift left,"(d_{k+1}^{2,3})^* \, |_{\overline{\im (d_{k+1}^{2,3})} \oplus U}"]
                    \end{tikzcd}
                \end{subfigure}
                \hspace{1cm}
                \begin{subfigure}[t]{0.4\textwidth}
                    \begin{tikzcd}[row sep = +4em, column sep = +7em]
                        \textcolor{blue}{\overline{\im((d_k^1)^*)}} \arrow[r,"d_{k}^{1}|_{\overline{\im((d_k^1)^*)}}",shift left] \arrow[d,hook] & C_{k-1}^1 \arrow[d,hook] \arrow[l,shift left, "(d_k^1)^*"]\\
                        \textcolor{blue}{V \oplus \overline{ \im((d_k^2)^*)}} \arrow[r,"d_{k}^{2}|_{V \oplus \overline{\im((d_k^2)^*)}}",shift left] & C_{k-1}^2 \arrow[l,shift left, "(d_k^2)^*"]
                    \end{tikzcd}
                \end{subfigure}
            \caption{In the proof of \cref{prop:fmdecompose} we apply up monotonicity to the left and down monotonicity to the right diagram.}
            \label{fig:updownsplit}
        \end{figure}
        
         To proceed, we assume without loss of generality that $\im(d_{k+1}^{2,3})$ is closed (as in the proof of \cref{eq:upmon2}); otherwise, \cref{prop:imclosed}(b) yields $0 \in \sigma_e \Big( \Delta_{+,k}^{2,3}\,|_{\overline{\im(d_{k+1}^{2,3})}} \Big)$ and thus by \cref{eq:splitse2}, $0 \in \sigma_e(\Delta_k^{2,3})$.
         
         Hence, \cref{closed range for full} applies and we may, again without loss of generality, assume that $$k_1:= \dim(\ker(\Delta_k^{1,3})) \le k_2:=\dim(\ker(\Delta_k^{2,3})) < \infty.$$         
        Thus, the spectra of $\Delta_k^{2,3}$ and $\Delta_k^{2,3}\, |_{(\ker(\Delta_k^{2,3}))^\perp}$ differ precisely by the multiplicity of the zero eigenvalue: it has multiplicity $k_2$ for $\Delta_k^{2,3}$ and zero for $\Delta_k^{2,3}\, |_{(\ker(\Delta_k^{2,3}))^\perp}$. In particular, the essential spectra are the same. Note that $0$ can still be an essential eigenvalue of $\Delta_k^{2,3}$. The same applies to $\Delta_k^{1,3}$. This observation allows us to conclude by proving the following two claims.
        \begin{claim} \label{claim:split1}
            The infima of the essential spectra satisfy $m:=m(\Delta_k^{1,3}) \ge m(\Delta_k^{2,3}).$
        \end{claim}
        \textit{Proof of claim:} If $m=0$, \cref{eq:splitse1} yields that $0 \in \sigma_e\Big(\Delta_{+,k}^{1,3}|_{\ker(d_k^1)}\Big) \cup \sigma_e\Big(\Delta_{-,k}^{1,3}|_{\overline{\im((d_k^1)^*)}}\Big).$ By up and down monotonicity applied as in \cref{fig:updownsplit}, we get 
        \begin{align}
             0\in \sigma_e \Big( \! \Delta_{+,k}^{2,3}|_{\overline{\im(d_{k+1}^{2,3})}\oplus U}\Big) &\cup \sigma_e \! \left(\! \Delta_{-,k}^{2,3}|_{V \oplus \overline{\im((d_k^2)^*)}}\right) \nonumber \\
             =\sigma_e\! \left(\! \Delta_{+,k}^{2,3}|_{\ker(d_k^2)}\right) &\cup \sigma_e\! \left(\! \Delta_{-,k}^{2,3}|_{\overline{\im((d_k^2)^*)}}\right). \label{eq:pfofclaim}
        \end{align} For \cref{eq:pfofclaim}, we used that changing the domain of the operators by the finite dimensional $V$ does not change the essential spectra. Consequently, $m(\Delta_k^{2,3})=0$.
        
        If $m>0$, our observation yields that $m$ is also an essential eigenvalue of $\Delta_k^{1,3}\, |_{(\ker(\Delta_k^{1,3}))^\perp}$ and, by \cref{cor:split}, also of $\Delta_{+,k}^{1,3}|_{\overline{\im(d_{k+1}^{1,3})}}$ or $\Delta_{-,k}^{1,3}|_{\overline{\im((d_k^1)^*)}}$. Up and down monotonicity and the fact that $k_1\le k_2 <\infty$ imply that there is an essential eigenvalue $m'\le m$ of $\Delta_{+,k}^{2,3}|_{\overline{\im(d_{k+1}^{2,3})}}$ or $\Delta_{-,k}^{2,3}|_{\overline{\im((d_k^2)^*)}}$. By \cref{cor:split}, $m'$ is essential eigenvalue of $\Delta_k^{2,3}$, proving \cref{claim:split1}.
        
        \begin{claim}
            If for $r \ge 0$, $\Delta_k^{1,3}$ has $n$ eigenvalues $\le r$ (counted with multiplicity), then either $\Delta_k^{2,3}$ also has $n$ eigenvalues $\le r$, or $m(\Delta_k^{2,3}) \le r$.
        \end{claim}

       \textit{Proof of claim:} We know that $\Delta_k^{1,3}\, |_{(\ker(\Delta_k^{1,3}))^\perp}$ has $n-k_1$ eigenvalues $\le r$. By \cref{cor:split}, $\Delta_{+,k}^{1,3}|_{\overline{\im(d_{k+1}^{1,3})}}$ and $\Delta_{-,k}^{1,3}|_{\overline{\im((d_k^1)^*)}}$ together also have $n-k_1$ eigenvalues $\le r$. Hence, $\Delta_{+,k}^{1,3}|_{\ker(d_k^1)}$ and $\Delta_{-,k}^{1,3}|_{\overline{\im((d_k^1)^*)}}$ together have $n$ eigenvalues $\le r$. By up and down monotonicity, the same is true for $\Delta_{+,k}^{2,3}|_{\overline{\im(d_{k+1}^{2,3})}\oplus U}$ and $\Delta_{-,k}^{2,3}|_{V \oplus \overline{\im((d_k^2)^*)}}$ (or one of them has an essential eigenvalue $\le r$ which would result in $m(\Delta_k^{2,3})\le r$, as in the proof of \cref{claim:split1}). Applying \cref{cor:split} once again yields the claim and therefore completes our proof.\\
\end{proof}

We now give a condition that is equivalent to the condition for full monotonicity in \cref{prop:fmdecompose}, and easier to verify in practice.

\begin{lemma}
    Let $p':C^2_k\to \iota^{1,2}_k(C^1_k)$ be the orthogonal projection and $p=(\iota_k^{1,2})^{-1}\circ p'$. Then 
    $$
        \iota_k^{1,2} \left(\overline{\im((d_k^1)^*)}\right) \subset \ker(d_{k+1}^{2,3})^* \quad
        \Leftrightarrow \quad p(\im d^{2,3}_{k+1})\subset \ker d^1_{k}.
    $$
\end{lemma}

    \begin{proof}
        By \cref{lem:imperp}, $\ker(d_{k+1}^{2,3})^*=\im(d_{k+1}^{2,3})^\perp$. Moreover, $p^*=\iota_k^{1,2}$. Therefore we have
        \begin{align*}
            \qquad\quad\iota_k^{1,2} \left(\overline{\im((d_k^1)^*)}\right) \subset \ker(d_k^{2,3})^*
            &\Leftrightarrow \ \im d^{2,3}_{k+1}\perp \iota^{1,2}_k(\im (d^1_k)^*),\\
            &\Leftrightarrow \ p(\im d^{2,3}_{k+1})\perp\im (d^1_k)^*,\\
            &\Leftrightarrow \ p(\im d^{2,3}_{k+1})\subset (\im (d^1_k)^*)^\perp=\ker d^1_k.
        \end{align*}
    \end{proof}
Having found a sufficient condition for full monotonicity in \cref{prop:fmdecompose}, a natural follow-on question is whether this condition is also necessary. The following example shows that it is not.
\begin{example}
    In \cref{fig:fmnoimker}, we have 
    $$
    \iota_k^{1,2} \left(\overline{\im((d_k^1)^*)}\right) = \{(a,0): a \in \R\}\not \subset \{(a,a):a \in \R\}= \ker(d_{k+1}^{2,3})^*.
    $$
    However, $\Delta_k^{1,3}=1$ and $\Delta_k^{2,3}=\renewcommand*{\arraystretch}{1.2}\begin{pmatrix}
        \frac 5 4 & \frac 3 4 \\ \frac 3 4 & \frac 5 4  
    \end{pmatrix}$ with eigenvalues $\frac 1 2 < 1 $ and 2.
\end{example}
\begin{figure}[H]
    \centering
    \begin{tikzpicture}[x=1cm,y=1cm]
	\clip(-1,-0.9) rectangle (7,2.5);
    \draw [-{Stealth[length=2mm]}] (1,2) --  node[above=1pt] {} (2,2);
    \draw [-{Stealth[length=2mm]}] (4,2) --  node[above=1pt] {$1$} (5,2);
    \draw [-{Stealth[length=2mm]}] (4,0.5) --  node[below=1pt] {$\begin{pmatrix}
1&1
\end{pmatrix}$} (5,0.5);
    \draw [-{Stealth[length=2mm]}] (1,0.5) --  node[below=1pt] {$\begin{pmatrix}
    \frac 1 2 \\ -\frac 1 2 
\end{pmatrix}$} (2,0.5);
    \draw [right hook-{Stealth[length=2mm]}] (0,1.7) -- node[left]{} (0,0.8);
    \draw [right hook-{Stealth[length=2mm]}] (3,1.7) -- node[right]{}(3,0.8);
    \draw [right hook-{Stealth[length=2mm]}] (6,1.7) -- node[right]{}(6,0.8);
    \draw[color=black] (0,2) node {$C_{k+1}^{1,3}=\{0\}$};
    \draw[color=black] (3,2) node {$C_k^{1}=\R$};
    \draw[color=black] (6,2) node {$C_{k-1}^{1}=\R$};
    \draw[color=black] (6,.5) node {$C_{k-1}^{2}=\R$};
    \draw[color=black] (3,.5) node {$C_k^{2}=\R^2$};
    \draw[color=black] (0,.5) node {$C_{k+1}^{2,3}=\R$};
\end{tikzpicture}
\caption{A configuration satisfying full monotonicity even though $\iota_k^{1,2} \left(\overline{\im((d_k^1)^*)}\right) \not \subset \ker(d_{k+1}^{2,3})^*$. All inclusions $\R^n \to \R^N$ include $\R^n$ as the first $n$ coordinates of $\R^N$.}
\label{fig:fmnoimker}
    \end{figure}

\section{Discussion}
\label{sec:discussion}

In this paper, we defined and presented the first systematic treatment of generalized chain Laplacians and persistent chain Laplacians for Hilbert complexes, providing a unified operator-theoretic framework that encompasses both Laplacians derived from de Rham complexes and their discrete, finite-dimensional counterparts. Under mild assumptions, we proved that the kernel of the persistent chain Laplacian is isomorphic to the persistent homology of a pair of Hilbert complexes, thus generalizing the classical combinatorial Hodge theorem to our broader setting. We then established monotonicity and stability for the up- and down-persistent Laplacians and showed that their spectra fully determine the spectra of the full persistent Laplacian. This result not only simplifies the analysis of persistent Laplacians but also provides a principled justification to focus on the up- and down-components. We further showed that, in general, full persistent Laplacians fail to satisfy these properties, and we gave a sufficient condition under which monotonicity and stability are recovered. Our results position persistent Laplacians as natural topological invariants that extend and enrich persistent homology while capturing additional geometric information through their spectra and lays the groundwork for a potential paradigm shift towards the use of Laplacians as fundamental topological and geometric descriptors.

\subsection{Directions for Future Research}

Our work provides a solid foundation for future investigations, which we now outline. Stability is a crucial property, however, it only provides a basic understanding of the spectra of up, down and full persistent Laplacians, both in the discrete setting that served as our initial motivation as well as in the manifold setting that we defined in \cref{subsub:derham}. 

A natural first direction for further research is to develop a geometric understanding of the spectra beyond the kernel corresponding to persistent homology. A possible approach to studying this question is via the persistent Cheeger inequality \citep[Theorem 4.20]{mww}, which relates the second smallest eigenvalue of persistent 0-Laplacians to the persistent Cheeger constant: a measure of ``how close a subgraph is to being disconnected inside the larger graph.'' It seems plausible that analogous Cheeger-type inequalities also hold for higher-order persistent Laplacians; see, for example, Corollary 2 in \citet{cheeger_sc}, which establishes a Cheeger inequality for simplicial complexes.  Similarly, we expect that the largest persistent eigenvalue will be related to how ``dense" or ``curved" the simplicial complex or manifold gets. In the case of Laplacians of graphs, \citet[\S4]{graph_1} presents some results relating the largest eigenvalue with the maximal degree.  Through our manifold example from \cref{subsub:derham}, we expect that further results from spectral graph theory can be carried over towards a better understanding or alternative perspective on the geometry of manifolds.

Another possible direction for future studies is the exploration of whether there exist further conditions for stability of full persistent Laplacians than those presented in \cref{subsub:confitionfm}. A stronger result to achieve would be a condition which is necessary as well as sufficient. However, an important difficulty in finding such a condition is that there are parametric configurations, which, depending on the parameter, are either stable or not: Replacing the $\pm \frac 1 2$ in \cref{fig:fmnoimker} by $\pm r$, we obtain $\Delta_k^{1,3}=1$ and $\Delta_k^{2,3} = \begin{pmatrix}
    1+r^2 & 1-r^2\\ 1-r^2 & 1+r^2
\end{pmatrix}$ with the eigenvalues $2$ and $2r^2$. Hence, this configuration is stable if and only if $|r| \le \frac 1 {\sqrt 2}$.

Studying the stability of persistent eigenvectors (as opposed to eigenvalues) is another question that is related to and an extension of our work. In general, eigenvectors are stable as long as the eigenvalues are isolated; see \citet[\S2.4 and 2.5]{numla}. By looking at spectral measures, which in the finite dimensional setting corresponds to projecting onto the direct sum of several different eigenspaces, it should be possible to develop stability results for these projections, for example, following the approach of \citet{powerspectrum}.

Beyond stability, there are also possible extensions and considerations for our work.  One direction is to explore the relaxation of the notion of inclusions in \cref{def:inclusion}, for example, following the approach of \citet{pl_simpmaps}, where the only requirement is for the maps between the chain complexes to be weight-preserving simplicial maps (instead of isometries).
Another direction is computational: Bridging the inherently finite-dimensional setting with the continuous one will require discretization together with approximation guarantees, for instance along the lines of \S5 of \citet{finderham}.

Our work also suggests interesting problems to study from the perspective of operator theory.  Specifically, we have seen that  \cref{prop:fmdecompose} does not give a sufficient and necessary condition for the monotonicity to hold. We may reinterpret the problem from a pure operator theory framework: Let $U,V,W$ be (possibly infinite dimensional) Hilbert spaces with densely defined, closed operators $A:U\to V$ and $B:V\to W$ such that $\im A\subset \ker (B)$ and $\dom A^*\cap \dom B$ is dense in $V$. Now choose some closed subspace $\tilde{V}\subset V$ such that we can define the following:
\begin{enumerate}
    \item[(i)] $\tilde{U}:=\overline{\{u\in U:Au\in \tilde{V}\}},$
    \item[(ii)] $\tilde{A}:=A|_{\tilde{U}}\quad \tilde{U}\to \tilde{V}$,
    \item[(iii)] $\tilde{W}$ is a closed subspace of $W$ containing $B(\tilde{V}\cap \dom B)$ such that if we define $$\tilde{B}:=B|_{\tilde{V}}\quad \tilde{V}\to \tilde{W},$$ then $\dom \tilde{A}^*\cap\dom \tilde{B}$ is dense in $\tilde{V}$.
\end{enumerate}
Then we define $AA^*+B^*B$ and $\tilde{A}\tilde{A}^*+\tilde{B}^*\tilde{B}$ via quadratic forms. An interesting question is whether there exists a sufficient and necessary condition for the choice of $\tilde{V}$, such that the choice of $\tilde{W}$ can be any closed subspace of $W$ containing $B(\tilde{V}\cap \dom B)$ and satisfies
$\lambda_n(AA^*+B^*B)\leq\lambda_n(\tilde{A}\tilde{A}^*+\tilde{B}^*\tilde{B}),$
where $\lambda_n$ is understood as in \eqref{eq:specctng}.  Effectively, this setup interprets a problem from computational topology in operator theory form.

\section*{Acknowledgments}

We are especially grateful to Vin de Silva, In\'{e}s Garc\'{i}a-Redondo, Ari Laptev,  Daniel Platt, and Daniel Ruiz Cifuentes for many helpful conversations. We also wish to thank Benjamin Briggs and Travis Schedler for input and feedback on our work. 

A.W.~is funded by a London School of Geometry and Number Theory--Imperial College London PhD studentship, which is supported by the Engineering and Physical Sciences Research Council [EP/S021590/1]. A.M.~is supported by the EPSRC AI Hub on Mathematical Foundations of Intelligence: An ``Erlangen Programme'' for AI No. EP/Y028872/1.

\begin{appendices} \crefalias{section}{appendix}
\section{Supplementary Material on Functional Analysis}
\label{appendix:operator_theory}

This appendix provides additional background material and auxiliary results on functional analysis that support the main developments of the paper.

We begin with an important but straightforward result used in our main construction, followed by a collection of frequently-used results from functional analysis.

\begin{lemma} \label{closedstaysclosed}
    Given two Hilbert spaces $\mathcal{H},\mathcal{K}$, let $D$ be a closed subspace of $\mathcal{H}$ and isometry $f:\mathcal{H}\to\mathcal{K}$ an isometry defined on $D$. Then $f(D)$ is closed in $\mathcal{K}$.
\end{lemma}
\begin{proof}
        Take any convergent sequence $(f(x_n))\subset f(D)$ such that $f(x_n)\to y$ for some $y\in \mathcal{K}$. Since $f$ is an isometry, 
        \begin{align*}
            \|x_n-x_m\|^2=\|f(x_n)-f(x_m)\|^2\to 0.
        \end{align*}
        Thus, $(x_n)$ is Cauchy. Given $D$ is closed, hence complete, there exists some $x\in D$ with $x_n\to x$. Given $f$ is an isometry, hence continuous, we have $f(x_n)\to f(x)$. Thus $y=f(x)\in f(D)$.
    \end{proof}

Next we present some decompositions of Hilbert spaces. The following is a standard result from functional analysis. 

\begin{lemma}\label{lem:orthogonal_decomposition}
    Let $V$ be a Hilbert space, $U\subset V$ a closed linear subspace and let $$U^\perp:=\{v \in V: \ \langle v,u \rangle=0 \ \forall\ u \in U \}$$ be the orthogonal complement of $U$. Then 
    $
        V=U \oplus U^\perp.
    $
\end{lemma}

\begin{proof}
     By \citet[Corollary 5.4]{funcana}, we have a projection map $P_U:V \rightarrow U$ with the property that $\langle v - P_U \, v,u \rangle =0$ for all $v \in V$ and $u \in U$. Thus, we can write any $v \in V$ as $$v=P_U \, v + (v-P_U\, v) \in U \oplus U^\perp.$$ \vspace{-1.35cm}\\
\end{proof}

\begin{lemma}\label{lem:imperp} 
Let $A:U \rightarrow V$ be a densely defined linear map between Hilbert spaces, then
    $$
    \Big(\overline{\im(A)}\Big)^\perp=(\im(A))^\perp = \ker(A^*).
    $$ 
\end{lemma}
\begin{proof}
    For the first equality, note that for any subset $M \subset V$, we have $M^\perp=\overline M^\perp$ by continuity of inner products.
    
    $(\im(A))^\perp \supset \ker(A^*)$ follows from the definition of adjoints. Conversely, let $v\in(\im (A))^\perp$. Then $\langle v,Au \rangle=0$ for all $u\in\dom A$. Thus, $v\in\dom A^*$ since we can choose $g=0$ and have $\langle g,u \rangle=\langle v,Au \rangle$ for all $u \in \dom(A)$. Therefore, by the uniqueness of $g$, we have $A^*v=g=0$.
\end{proof}

Composing \cref{lem:orthogonal_decomposition} and \cref{lem:imperp} above yields the following result.
\begin{cor}\label{cor:decompV}
    Let $A:U \rightarrow V$ be a densely defined linear map between Hilbert spaces, then
    $$V=\overline{\im(A)} \oplus \ker(A^*).$$
\end{cor}

\section{Defining Homology Using Category Theory}
\label{appendix:category}

Here we provide additional background on defining the homology of a Hilbert complex from the perspective of category theory. The purpose is to provide further context on how reduced homology is a natural concept to consider, perhaps even more so than standard homology. The details of the definitions here can be found in \citet{maclane}. 

We are interested in the category \textbf{Hilb} of Hilbert spaces with continuous linear maps as morphisms; we require them to be everywhere defined in order to allow for their composition. \textbf{Hilb} is an example of an \textit{additive category}, which in particular has a zero object $0$, i.e., an object such that there is precisely one morphism $\Hh \to 0$ for each $\Hh \in \textbf{Hilb}$.  \textbf{Hilb} has many additional interesting and desirable properties \citep{cathilb}, but we are particularly interested in the existence of $0$, which allows us to define the concepts of kernel, cokernel, and image in category theoretic terms. In general, these will always be objects from \textbf{Hilb}. However, the image of continuous maps $f:\Hh \to \Kk$ between Hilbert spaces is not necessarily closed, and therefore may not be a Hilbert space itself. We now overview how $\overline{\im(f)}$ takes on the role of the image of $f$ in \textbf{Hilb}.

\begin{definition}\textnormal{\cite[Kernel,][p.~191]{maclane} }
    The \emph{kernel} of a morphism $f: A \to B$ in a category with zero is an object $K$ together with a morphism $k: K \to A$ such that $f \circ k = 0$ and for every object $K'$ and morphism $k': K' \to A$ that satisfies $f \circ k' = 0$, there is a unique morphism $g: K' \to K$ such that $k \circ g = k'$. 
\end{definition}
A diagram illustrating the kernel is shown in \cref{fig:kernel}.

\begin{figure}[h]
\begin{center}
\begin{tikzcd}
A \arrow[rrrr, "f"] &  &  &  & B \\
                    &  &  &  &   \\
                    &  & K \arrow[lluu, "k"] \arrow[rruu, "0"']  &  &   \\
                    &  &  &  &   \\
                    &  & K' \arrow[lluuuu, "k'", bend left] \arrow[uu, "\exists \text{!} g"] \arrow[rruuuu, "0"', bend right] &  &  
\end{tikzcd}
\caption{The kernel.} \label{fig:kernel}
\end{center}
\end{figure}
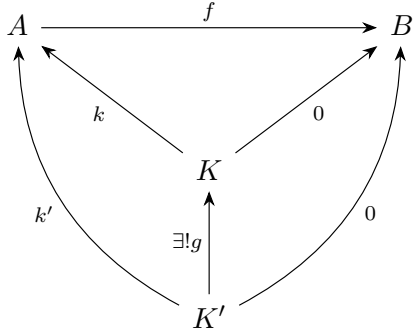

In \textbf{Hilb}, the kernel of a morphism agrees precisely with the usual notion of kernel.
However, this agreement is not true for the cokernel. In general, the cokernel is illustrated in \cref{im:defcoker}) and defined as follows.
\begin{definition}\textnormal{\citep[Cokernel,][p.~64]{maclane}}
    The \emph{cokernel} of a morphism $f: A \to B$ in a category $\mathcal{C}$ is an object $C$ together with a morphism $q: B \to C$ such that $q \circ f = 0$ and for every $X \in \mathcal{C}$ and morphism $q': B \to X$ with $q' \circ f = 0$, there is a unique morphism $g: C \to X$ such that $g \circ q = q'$.
\end{definition}

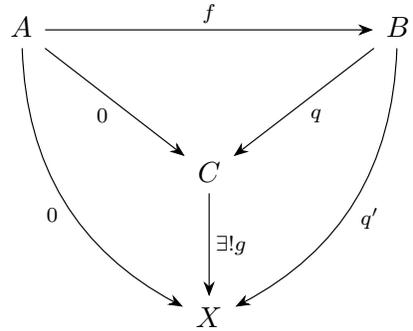
\begin{figure}[H]
\begin{center}
\begin{tikzcd}
A \arrow[rrrr, "f"] \arrow[rrdd, "0"'] \arrow[rrdddd, "0"', bend right] & & & & B \arrow[lldd, "q"] \arrow[lldddd, "q'", bend left] \\
  &  &   &  &  \\&  & C \arrow[dd, "\exists \text{!} g"] &  &   \\&  &  &  &  \\&  & X                                                              &  &                                                    
\end{tikzcd}
\caption{The cokernel.} \label{im:defcoker}
\end{center}
\end{figure}

Thus, in our category \textbf{Hilb}, the cokernel needs to be a Hilbert space. In general, the usual notion of cokernel ``$B/\im(f)$'' is not a Hilbert space and therefore not the cokernel of $f$. It turns out that $B/\overline{\im(f)}$ is a Hilbert space and indeed $\text{coker}_{\textbf{Hilb}}(f)=B/\overline{\im(f)}$.

For additive categories such as \textbf{Hilb}, $\im_\textbf{Hilb}$, is defined as the kernel of this cokernel map $q$ \citep[see][p.~8]{tohoku}. Consequently, the image of $f$ in the category \textbf{Hilb} is $\im_\textbf{Hilb}(f)=\overline{\im(f)}$. Thus, for a Hilbert complex $P$, the framework of category theory yields the following notion of homology:
$$
     H_{k,\textbf{Hilb}}(P) := \ker_{\textnormal{\textbf{Hilb}}}(d_k) / \im_{\textnormal{\textbf{Hilb}}}(d_{k+1}) = \ker(d_k)/\overline{\im(d_{k+1})},
$$
which is precisely the reduced homology we defined in \cref{def:redhom}.

\end{appendices}


\vfill\eject

\bibliographystyle{authordate3}

\bibliography{Laplacian_ref}


\end{document}